\documentclass[12pt, reqno]{amsart}

\usepackage{mathrsfs}
\newtheorem{theorem}{Theorem}[section]
\newtheorem{lemma}[theorem]{Lemma}
\newtheorem{prop}[theorem]{Proposition}
\newtheorem{cor}[theorem]{Corollary}
\newtheorem{definition}[theorem]{Definition}
\newtheorem{remark}[theorem]{Remark}
\newtheorem*{theorem*}{Theorem}

\theoremstyle{definition}

\usepackage{soul}
\setstcolor{red}
\DeclareMathOperator{\Log}{Log}
\usepackage{marginnote}
\usepackage[pagebackref,hypertexnames=false, colorlinks, citecolor=red,linkcolor=blue, urlcolor=red]{hyperref}

\theoremstyle{remark}

\numberwithin{equation}{section}

\newcommand{\overbar}[1]{\mkern 1.5mu\overline{\mkern-1.5mu#1\mkern-1.5mu}\mkern 1.5mu}

\def\norm#1{\left\|{#1}\right\|}
\usepackage{lmodern}
\usepackage{helvet}

\usepackage{enumitem}

\usepackage{color}
\usepackage{comment}
\usepackage{fullpage}

\usepackage{ulem}

\usepackage{tikz-cd}

\begin{document}

\title[Homogeneous analytic Hilbert modules]{Homogeneous analytic Hilbert modules -- \\the case of non-transitive action}

\author[S. Biswas]{Shibananda Biswas}
\address[S. Biswas]{Department of Mathematics and Statistics, Indian Institute of Science Education and Research (IISER) Kolkata, Mohanpur - 741246, Nadia, West Bengal, India}
\email{shibananda@iiserkol.ac.in}
\thanks{The first named author gratefully acknowledges the Academic and Research Fund provided by IISER Kolkata for supporting this research. The author would also like to acknowledge the support provided by the FIST grant (SR/FST/MS-II/2019/51) at IISER Kolkata provided by the Government of India.}

\author[P. Deb]{Prahllad Deb}
\address[P. Deb]{Department of Mathematics, IIIT-Delhi, New Delhi - 110020, India}

\email{prahllad@iiitd.ac.in}
\thanks{The second named author gratefully acknowledges the postdoctoral fellowship in the Department of Mathematics at Ben-Gurion University of the Negev, as well as the professional development fund provided by IIIT-Delhi, for supporting this research.}

\author[S. Hazra]{Somnath Hazra}
\address[S. Hazra]{Institute for Advancing Intelligence\\ TCG Centres for Research and Education in Science and Technology \\ Kolkata - 700091}
\email{somnath.hazra@tcgcrest.org}

\author{Dinesh Kumar Keshari}
\address[D. Keshari]{National Institute of Science Education and Research (NISER), An OCC of Homi Bhabha National Institute, School of Mathematical Sciences, Bhubaneswar, Post-Jatni\\ Khurda - 752050, India}

\email{dinesh@niser.ac.in}
\thanks{The fourth named author gratefully acknowledges the generous support provided by NISER for facilitating the in person discussions of the authors at its campus.}

\author[G. Misra]{Gadadhar Misra}
\address[G. Misra]{Indian Statistical Institute, Bangalore - 560059\\           and Indian Institute of Technology, Gandhinagar - 382055, India}

\email{gm@isibang.ac.in}

\date{\today}

\dedicatory{}

\begin{abstract} This work investigates analytic Hilbert modules $\mathcal{H}$, over the polynomial ring, consisting of holomorphic functions on a $G$-space $\Omega \subset \mathbb{C}^m$ that are homogeneous under the natural action of the group $G$. In a departure from the past studies of such questions, here we don't assume transitivity of the group action.  The primary finding reveals that unitary invariants such as curvature and the reproducing kernel of a homogeneous analytic Hilbert module can be deduced from their values on a fundamental set $\Lambda$ of the group action. Next, utilizing these techniques, we examine the analytic Hilbert modules associated with the symmetrized bidisc $\mathbb{G}_2$ and its homogeneity under the automorphism group of $\mathbb{G}_2$. It follows from one of our main theorems that none of the weighted Bergman metrics on the symmetrized bidisc is K\"{a}hler-Einstein.
\end{abstract}

\maketitle

\bibliographystyle{amsplain}
\baselineskip = 15pt

\section{Introduction}

Let $G$ be a locally compact second countable group and $\Omega \subset \mathbb  C^m$ be a bounded open connected $G$-space. We assume that the group $G$ acts on $\Omega$ holomorphically, that is, the map $\alpha_g(z):= g\cdot z$, $g\in G, z\in \Omega$, is bi-holomorphic for each fixed $g\in G$ and $g \mapsto \alpha_g$ is a homomorphism.  We also assume that the map extends to a holomorphic map in some neighbourhood (depending on $g$) of the closure $\overbar{\Omega}$ of $\Omega$. Finally, we let $g^1:=\alpha_g^1, \ldots , g^m:=\alpha_g^m$, where $\alpha_g^i$, $1\leq i \leq m$, are the coordinate functions of the holomorphic map $\alpha_g: \Omega \to \Omega$.

Let $\mathcal{H}$ be a Hilbert space consisting of holomorphic functions defined on $\Omega$. Assume that $\boldsymbol{M}:=(M_1, \ldots , M_m)$, the commuting tuple of operators of multiplication by the coordinate functions $z_1, \ldots , z_m$ on $\mathcal{H}$ is bounded.
The group $G$ acts on the commuting $m$-tuple $\boldsymbol{M}$  by setting 
 $g\cdot \boldsymbol{M}= (\alpha_g^1(\boldsymbol{M}), \ldots , \alpha_g^m(\boldsymbol{M}))$. This is then the $m$-tuple of multiplication operators $(M_{\alpha_g^1}, \ldots , M_{\alpha_g^m})$. We will abbreviate this to $g\cdot \boldsymbol{M} = (M_{g^1} , \ldots , M_{g^m})$, where $M_{g^i}$ is a well defined bounded linear operators on $\mathcal H$. 
This is easy to see: Each $g^i$, $1\leq i \leq m$, is a holomorphic function in some open set containing $\overline{\Omega}$ and therefore, $g^i(\boldsymbol{M})$ defined using the holomorphic functional calculus is a bounded operator. It follows that $g\cdot \boldsymbol{M}$ is nothing but the $m$-tuple  $\boldsymbol{M}_g:= (M_{g^1} , \ldots , M_{g^m})$ of multiplication operators by the functions $g^1, \ldots, g^m$ on $\mathcal{H}$. 

\begin{definition} Let $\Omega$ be a $G$-space and $\mathcal{H}$ be a Hilbert space consisting of holomorphic functions defined on $\Omega$. 
We say that the commuting tuple $\boldsymbol{M}$ acting on $\mathcal{H}$ is \textit{homogeneous} with respect to the action of the group $G$ if, for all $g \in G$, there exists a unitary operator $U_g$ such that 
\[U_g^* \boldsymbol{M}_g U_g : = (U_g^* M_{g^1}U_g, \ldots , U_g^* M_{g^m} U_g) = (M_1, \ldots, M_m) = \boldsymbol{M}.\]
\end{definition} 

In this paper, we study multiplication operators that are homogeneous under such a $G$-action, without assuming that the action is transitive. In the non-transitive case: one can no longer track the situation via a single point, and must instead keep track of a whole set of orbit representatives. Let $A$ be an indexing set and $\Lambda_\alpha$, $\alpha\in A$  be a partition of $\Omega$ into $G$-orbits. Let $\Lambda$ be the disjoint union of the sets consisting of exactly one element from each $G$-orbit $\Lambda_\alpha$, $\alpha\in A$. Thus, \[G\cdot \Lambda :=\{g \cdot z \mid g\in G, \, z\in \Lambda \}=\Omega.\] Such a set $ \Lambda $ is called a fundamental set. Note that the choice of $ \Lambda $, in general, is not unique. If the group $ G $ acts on $ \Omega $ transitively, then the fundamental set $\Lambda$ is a singleton by definition. 

Now, to check if a commuting tuple $\boldsymbol{M}$ is homogeneous, and eventually to find all such tuples, one must identify a complete set of unitary invariants, determine how they transform under the $G$-action, and examine their behavior when restricted to the fundamental set. If $\mathcal{I}(\boldsymbol{M})$ denotes a complete set of unitary invariant for $\boldsymbol{M}$, we would like to express the invariants $\mathcal{I}(\boldsymbol{M}_g)$ in terms of $\mathcal{I}(\boldsymbol{M})$, that is, find a rule $\Phi_g$ to obtain the set of invariants $\mathcal{I}(\boldsymbol{M}_g)$ from $\mathcal{I}(\boldsymbol{M})$. One may think of such a relationship as a change of variable. If we assume that $\boldsymbol{M}$ is unitarily equivalent to $\boldsymbol{M}_g$, $g\in G$, then we have 
\begin{equation} \label{changeV}
\mathcal{I}(\boldsymbol{M}) = \mathcal{I}(\boldsymbol{M}_g) = \Phi_g (\mathcal{I}(\boldsymbol{M})),\,\, g\in G,
\end{equation}
where the first equality is forced since we are assuming that $\boldsymbol{M}$ and $\boldsymbol{M}_g$ are equivalent (and $\mathcal{I}$ is a set of complete unitary invariants) and the second is the change of variable formula for $\mathcal{I}$. Here, we exploit the relation given in \eqref{changeV} to determine its homogeneity using two types of invariants, namely, the reproducing kernel $K$ modulo a ``change of scale'' and the curvature $\mathcal{K}$ induced by $K$. All these notions are discussed in what follows. We show that these invariants are determined from their restriction to the fundamental set $\Lambda$.

Let ${\mathcal M}_n(\mathbb C)$ denote the vector space of all $n\times n$ complex matrices and $\langle ~,~ \rangle_{{\mathbb C}^n}$ be the standard inner product in ${\mathbb C}^n$. Let $K:\Omega\times\Omega \rightarrow {\mathcal M}_n(\mathbb C)$ be a function holomorphic in the first variable and anti-holomorphic in the second. We say that $(\mathcal{H}, K)$ is a reproducing kernel Hilbert space if $\mathcal{H}$ consists of $\mathbb C^n$-valued holomorphic functions defined on $\Omega$ and $K(\cdot, w)\zeta$ is in $\mathcal H$, $w\in \Omega$, $\zeta\in{\mathbb C}^n$ and finally, $K$ has the reproducing property, namely, $$\langle f,K(\cdot,w)\zeta\rangle =
\langle f(w),\zeta\rangle_{{\mathbb C}^n},\, w\in \Omega.$$
The linear span of $K(\cdot, w)\zeta$ , $w\in \Omega$, $\zeta\in{\mathbb C}^n$, is  dense in $\mathcal{H}$. The function $K$ is said to be the reproducing kernel of the Hilbert space $\mathcal{H}$ and it possesses the very important property, namely the non-negative definiteness property: For all subsets $\{w_1, \ldots , w_p\}$ of $\Omega$ and $\zeta_1,\ldots,\zeta_p \in{\mathbb C}^n$, $p\in \mathbb{N}$, 
$$\sum_{i,j=1}^p\langle K(w_i,w_j)\zeta_j,\zeta_i\rangle \geq 0.$$
Moreover, $K$ is hermitian, that is, $K(z,w) = {K(w,z)}^*$, $z,w\in \Omega$. Any hermitian function $K:\Omega\times \Omega \to {\mathcal M}_n(\mathbb C)$ possessing the non-negative definiteness property is said to be a non-negative definite kernel. Let $\mathcal{H}^0$ be the linear span of the vectors $\{K(\cdot, w)\zeta: w\in \Omega, \zeta\in\mathbb C^n\}$. Assuming that $K$ is non-negative definite, an inner product on $\mathcal{H}^0$ is defined by first setting $$\langle K(\cdot, w_1)\zeta, K(\cdot, w_2)\eta\rangle:= \langle{K(w_2,w_1)\zeta,\eta}\rangle_{{\mathbb C}^n}$$ and then extending it to the linear span $\mathcal{H}^0$ of these vectors. Completing $\mathcal{H}^0$ with respect to the norm induced by the inner product on $\mathcal{H}^0$, we obtain a functional Hilbert space in which $K$ serves as a reproducing kernel (Moore's Theorem, see \cite[Theorem 2.14]{PR}).  

Unless we mention otherwise, we will use scalar kernels, that is,  $K:\Omega\times\Omega \rightarrow \mathbb C$. In what follows, we then assume that the $m$-tuple $\boldsymbol{M}:= (M_1, \ldots , M_m)$ of multiplication by the coordinate functions (and hence multiplication by any polynomials $p \in \mathbb{C}[\boldsymbol{z}]$ in $m$-variables) is bounded on the reproducing kernel Hilbert space $(\mathcal{H}, K)$.
The reproducing kernel (scalar valued) Hilbert space $(\mathcal{H}, K)$ is a \textit{Hilbert module over the polynomial ring $\mathbb{ C } [ \boldsymbol{ z }] $}, equipped with the module map $\mathfrak{m}: \mathbb{ C } [ \boldsymbol{ z }] \times \mathcal{H} \to \mathcal{H}$ defined by $\mathfrak{m} (p, f) = p f$, where $p f $ denotes the point-wise multiplication. This is equivalent to saying that $p \mapsto \mathfrak{m}_p$, where $\mathfrak{m}_p: \mathcal{H} \to \mathcal{H}$ is defined by setting 
$\mathfrak{m}_p(f) = p f$, $p\in \mathbb{ C } [ \boldsymbol{ z } ] $, is a homomorphism into the algebra of bounded operators on $\mathcal{H}$. 
The study of Hilbert modules over a function algebra was introduced by R. G. Douglas and a detailed study appeared in \cite{RGDVIP}. The Hilbert modules considered here are modules over the polynomial ring $\mathbb{C}[\boldsymbol{z}]$. 

We also assume throughout the paper that the polynomial ring $\mathbb{C}[\boldsymbol{z}]$ is dense in $\mathcal{H}$. Adapting the terminology from \cite[page. 3]{ChenGuo}, any such reproducing kernel Hilbert space $(\mathcal{H}, K)$ is said to be an \textit{analytic} Hilbert module over the polynomial ring $\mathbb{C}[\boldsymbol{z}]$. We let $\operatorname{HM}_a(\Omega)$ be the set of all analytic Hilbert modules $(\mathcal{H}, K)$ over $ \mathbb{ C } [ \boldsymbol{ z } ]$ consisting of holomorphic functions defined on $\Omega$. 

A very useful consequence of the polynomial density is that 
\[\dim \big (\bigcap_{i=1}^m \ker (M_i - w_i)^*\big ) = 1,\]
see \cite[Remark, p. 285]{RGDGM}. Suppose that $(\mathcal{H},K)$ is an analytic Hilbert module, and $K$ is the reproducing kernel for $\mathcal{H}$ defined on $\Omega$. The map $\gamma:\Omega^* \to \mathcal{H}$ defined by setting $\gamma(\overbar{w}) = K(\cdot, w)$ is holomorphic on $\Omega^*$ by definition. Thus, 
\[\Omega^*\ni \overbar{w} \stackrel{\Gamma}{\longmapsto} \bigcap_{i=1}^m \ker(M_i-w_i)^*,\,\, \Omega^*=\{\overbar{w} \mid w\in \Omega\}\]
is a holomorphic  map into the projective space $P_1(\mathcal{H})$ of the Hilbert space $\mathcal{H}$: 
\[P_1(\mathcal{H}):= \{[f]\mid f\in \mathcal{H}, f \sim g \text{~if~} f = \lambda g \text{ ~for some~ } 0 \ne \lambda\in \mathbb{C}\}.\] 
The holomorphic frame $\gamma$ defines a  hermitian holomorphic line bundle $E$ on $\Omega^*$: The fibre at $\overbar{w}$ of $E$ is $\cap_{i=1}^m \ker (M_i - w_i)^*$ and it embeds into $\mathcal{H}$ as $\Gamma(\overbar{w})$. Declaring this embedding to be an isometry, we get a hermitian structure on $E$ defined by $K(w,w)$. 
Any two holomorphic frames of the bundle $E$ differ by a holomorphic multiplier, that is, if $K(\cdot, w)$ and $K^\prime(\cdot, w)$ both serve as holomorphic frames of $E$, then there is a  non-vanishing holomorphic function $\varphi$ defined on some open subset of $\Omega$ such that 
\begin{equation}\label{changeofscale}
K(z,w) = \varphi(z) K^\prime(z,w) \overline{\varphi(w)},\,\, z,w\in \Omega.\end{equation}
This is called a change of scale in \cite{CS} and is a particular instance of the second equality in Equation \eqref{changeV}.
Then it follows that two analytic Hilbert modules are (unitarily) isomorphic if and only if there is a non-vanishing change of scale. 
It is clear that to construct an invariant for the analytic Hilbert module $(\mathcal{H}, K)$, one must get rid of the ambiguity introduced by the change of scale. Indeed, Cowen and Douglas show that the curvature $\mathcal{K}$ of the line bundle $E$ (sometimes, we also refer to $\mathcal{K}$ as the curvature of the analytic Hilbert module $(\mathcal{H}, K)$), namely, 
\begin{equation} \label{curvaturedefn} \mathcal{K}(\bar{w}) = -\sum_{ i, j }^{ m }  \frac{ \partial^2 }{ \partial w_i \partial \overline{ w_j } } \log K (w, w)   dw_i \wedge d\overline{ w_j}, w \in \Omega,
\end{equation} 
will do the trick.  It follows that $K^\prime$ is obtained from $K$ by a change of scale if and only if their curvature $(1,1)$ forms are equal. 
Thus, the curvature $\mathcal K$ of the line bundle $E$ is a complete invariant for the analytic Hilbert module $(\mathcal{H}, K)$. Setting
\begin{equation}\label{eqn:1.4}
 \mathbb{K}(w):= \big(\!\! \big( \frac{ \partial^2 }{ \partial w_i \partial \overline{ w_j } } \log K ( { w }, { w } ) \big)\! \!\big),  w \in \Omega,
\end{equation}
we note that $\mathbb{K}(w)$ is an $n\times n$ hermitian matrix. Moreover, it is non-negative definite for each $w\in \Omega$, see \cite[Proposition 2.1]{BKM}. Thus, $\mathbb{K}$ defines a hermitian structure on the trivial holomorphic bundle of rank $n$ defined on $\Omega$. If $K$ is the Bergman kernel of $\Omega$, then $\mathbb{K}$ is defined to be the Bergman metric of the trivial holomorphic bundle of rank $n$ on $\Omega$.   

Before continuing further with these two invariants — the kernel modulo change of scale, and the curvature — we pause to rephrase $G$-homogeneous commuting tuples of operators in the language of Hilbert modules, since questions of this kind are often easier to describe that way.

\begin{definition}
An analytic Hilbert module $(\mathcal{H}, K)$ is said to be $G$-homogeneous if the commuting tuple $(M_1, \ldots , M_m)$ of operators induced by the module maps $\mathfrak{m}_{z_k}$, $k=1,\ldots , m$, is homogeneous under the group $G$. 
\end{definition}

Homogeneous analytic Hilbert modules have been studied in the past, assuming that the $G$-action is transitive \cite{BM, BM03, DH, H, K, KM11, M, MU, PV, Wil}. In this case, assigning a suitable value to $K(0,0)$, one determines $K(z,z)$ for any $z\in \Omega$. Examining this further, it is not hard to see that such a kernel must be a positive real power $\lambda$ of the Bergman kernel $B$. As long as $\lambda$ is picked so as to ensure $B^\lambda$ is non-negative definite, the description of all the $G$- homogeneous commuting tuples is complete. Homogeneity under non-transitive action might be traced back to the papers \cite{Gellar,AHHK}, where the group is taken to be the circle group. More recently, such questions have been investigated in the papers \cite{CY, GKP,GKMP}.  These questions (both in the case of transitive action as well as the case where it is not) overlap with the question of determining imprimitivity systems, see \cite{MNV} for details. 
 
To determine which analytic Hilbert module $\mathcal{H}$ is homogeneous, indeed, to find all of them, one may adopt one of the two methods outlined above. The first involving the kernel function says that an analytic Hilbert module $(\mathcal{H}, K)$ is homogeneous if and only if the kernel $K$ is quasi-invariant. A precise statement is given in Theorem \ref{thm:2.1-Kquasi}. The second approach is to use the curvature and conclude similarly that an analytic Hilbert module $(\mathcal{H}, K)$ is homogeneous if and only if its curvature $\mathcal{K}$ is quasi-invariant. Again, a precise statement is in Proposition \ref{curvature transformation}. In this paper, we combine both of these methods to obtain some interesting new results, for instance, Theorem \ref{cor:2.2} and Theorem \ref{thm:2.2}. In the first half of the paper, following the work of \cite{GM}, we provide a recipe of constructing new examples of homogeneous analytic Hilbert modules starting from known examples.  In the second half, we apply these results to study homogeneous analytic Hilbert modules on the symmetrized bidisc $\mathbb G_2$. 

Starting from a homogeneous analytic Hilbert module, we describe two algorithms to construct new ones. These depend crucially on the notion of a Wallach set that we recall now. Let $t$ be any positive real number. 
If $\Omega$ is simply connected, and $K$ is a non-negative definite kernel with $K(z,w) \ne 0$, $z,w\in \Omega$, then take the principal branch of the logarithm on $\Omega\times \Omega$ to ensure that $K^t$ is holomorphic in the first variable and anti-holomorphic in the second. However, it is not obvious that it is non-negative definite for any $t\not \in \mathbb{N}$. It is therefore natural to define the set 
\begin{equation} \label{Wallach set Omega K}
\mathcal{W}_\Omega(K):=\{ ( 0 < )\, t \mid K^t \text{ is non-negative definite}\}
\end{equation}  
known as the Wallach set. There are a few instances, where $\mathcal{W}_\Omega(K)$ is completely known, see for instance, \cite{FK}. In this paper, we consider a subset $\mathcal{W}_{\Omega,a}(K)$ of the Wallach set, namely, 
\begin{equation}\label{Wallach set Omega K analytic}
\mathcal{W}_{\Omega,a}(K):= \{( 0 < )\,t \mid (\mathcal{H}, K^t) \in \operatorname{HM}_a(\Omega)\}.
\end{equation}
Since the Schur product of two non-negative definite matrices is again non-negative definite, it follows that if $K$ is a non-negative definite kernel, then $K^n$, $n\in \mathbb{N}$, is also a non-negative definite kernel. Thus it is evident that $\mathbb{N} \subseteq \mathcal{W}_{\Omega}(K)$. In fact, it is not hard to verify that $\mathbb{N} \subseteq \mathcal{W}_{\Omega,a}(K)$ whenever $(\mathcal{H}, K)$ is in $\operatorname{HM}_a$.

Note that even if $\lambda$ is not in $\mathcal{W}_\Omega(K)$, the positive real analytic function $K(w,w)^\lambda$ defines a hermitian structure for the trivial holomorphic vector bundle on $\Omega$. In this paper, we assume $\lambda\in \mathcal{W}_{\Omega,a}(K)$, only to ensure the existence of an analytic Hilbert module $\mathcal{M}$, where $K^\lambda$ serves as its reproducing kernel. However, if we drop this assumption, the (appropriate and obviously modified) results below will be true of the trivial holomorphic hermitian vector bundle defined on $\Omega$ with hermitian structure $K^\lambda$. This is made explicit in Remark \ref{vb}.

Recall from \cite[Corollary 2.4]{GM} that if $\Omega$ is simply connected and $K$ is nowhere vanishing on $\Omega\times \Omega$, then 
\begin{equation} \label{curvaturematrix} \mathbb{K}^{[0]}(z,w) := K({z}, {w})^{2} \big(\!\! \big( \frac{ \partial^2 }{ \partial z_i \partial \overline{ w_j } } \log K ( { z }, { w } ) \big)\! \!\big), ~~ z, w \in \Omega,\end{equation} 
defines a non-negative definite kernel on $ \Omega\times \Omega $. 

So, for $ \lambda \in \mathcal{ W }_{ \Omega} ( K ) $, the function $\mathbb{K}^{[\lambda]}:= K^{\lambda} {\mathbb K^{[0]}} (=K^{\lambda+2} \mathbb{K})$, is a non-negative definite kernel on $\Omega$. 
What is more is that the Hilbert space $\mathbb{H}^{[\lambda]}$ obtained by taking the completion of the linear span of the holomorphic functions 
$\big \{\mathbb{K}^{[\lambda]}(\cdot, w ) \zeta \mid \zeta \in \mathbb{C}^m,\,w\in \Omega \big \}$
is a Hilbert space of holomorphic functions defined on $\Omega$ taking values in $\mathbb{C}^m$.  Moreover, the polynomial ring $\mathbb{C}[\boldsymbol{z}] \otimes \mathbb{C}^m$ is dense in $\mathbb{H}^{[\lambda]}$, provided $(\mathcal{H}, K^\lambda) \in \operatorname{HM}_a(\Omega)$. This has two consequences. First,  if $ \lambda \in \mathcal{ W }_{ \Omega, a } ( K ) $, then the Hilbert module $(\mathbb{H}^{[\lambda]}, \mathbb K^{[\lambda]})$ is homogeneous. This is proved in Proposition \ref{homog of K^[lambda]}.  Second, $(\mathbb{H}^{[\lambda]}_{\det}, \det(\mathbb{K}^{[\lambda]}))$ is an analytic Hilbert module and is homogeneous whenever $(\mathcal{H}, K)$ is homogeneous, see Theorem \ref{det theorem}. The same argument would show that $(\mathbb{H}^{[0]}_{\det}, \det(\mathbb{K}^{[0]}))$ is a homogeneous analytic Hilbert module. 

The constructions above of $K^\lambda$, $\det(\mathbb{K}^{[\lambda]})$ are worth carrying out as it leads to new examples in contrast to the transitive case. For example, if $\Omega$ is a Cartan domain (it is a bounded symmetric domain realized as a subset of $\mathbb{C}^m$ for some $m > 0$), then the bi-holomorphic automorphism group of $\Omega$ is transitive and the Bergman kernel $B_\Omega$, indeed, any $B_\Omega^\lambda$, $\lambda\in \mathcal{W}(B_\Omega)$, is quasi-invariant.  The function $\det(\mathbb{B}_\Omega)$ is a scalar multiple of the Bergman kernel $B_\Omega$, see Corollary \ref{multipleBerg}. Consequently, the kernel $\det(\mathbb{B}_\Omega^{[\lambda]})$, $\lambda\in \mathcal{W}(B_\Omega)$ is a power of the Bergman kernel $B_\Omega$. So in case of transitive action, the construction above does not give us any new example apart from the weighted Bergman space itself. A non-trivial consequence of this property is the following: 
\begin{align*}
\operatorname{tr}(\overbar{\partial}\big ( (\mathbb{B}_\Omega^{[\lambda]})^{-1} \partial \mathbb{B}_\Omega^{[\lambda]} \big ) &=\overbar{\partial}\partial \log \det \mathbb{B}_\Omega^{[\lambda]}
= \alpha \overbar{\partial}\partial \log B_\Omega,
\end{align*}
where  $\alpha$ is some positive constant and $\overbar{\partial}\partial \log B_\Omega$ is the curvature $(1,1)$ form of $B_\Omega$ as defined in Equation \eqref{curvaturedefn}. Thus, as is well-known (see \cite[Chapter 4, Proposition 3]{Mok}), these are K\"{a}hler-Einstein metrics on $\Omega$. A natural question is whether the kernel $\det(\mathbb{B}_\Omega^{[\lambda]})$ remains a power of $B_\Omega$ even when $\Omega$ is not a Cartan domain.

In the second half of the paper, applying the previous considerations to $\Omega:=\mathbb{G}_2$, the symmetrized bidisc and taking $G$ to be the automorphism group of $\mathbb{G}_2$, we show that for the weighted Bergman kernels defined on $\mathbb{G}_2$, this is not so.  
 As a result, in this case, starting from $\det(\mathbb{K}^{[\lambda]})$ we produce many more new examples of $\operatorname{Aut}(\mathbb{G}_2)$- homogeneous analytic Hilbert modules. 
As an application of our main theorem, we also show that none of the weighted Bergman kernels is K\"{a}hler-Einstein (see Corollary \ref{ke}). Any non-negative definite kernel $K$ on $\mathbb{G}_2$ with an affirmative answer to this question would provide an example of a  K\"{a}hler-Einstein metric on $\mathbb{G}_2$. 
Moreover, in this case,  Theorem \ref{thm_characterization} provides an explicit criterion in terms of the curvature invariant for detecting when a Hilbert module $\mathcal{H}$ in $\operatorname{HM}_a(\mathbb{G}_2)$ is $G$-homogeneous. 

Next, we recall that the weighted Bergman modules $(\mathbb{A}^{(\lambda)}(\mathbb{G}_2), B^{(\lambda)})$ are homogeneous under $G$.  The curvature of the Bergman kernel $B$ of $\mathbb{G}_2$ was computed earlier; see \cite{CY}. Here, we compute the curvature for the weighted Bergman kernels $B^{(\lambda)}$ defined on $\mathbb{G}_2$. Since the curvature serves as a complete unitary invariant, we use it as one of our main tools in comparing the unitary equivalence class of the weighted Bergman modules. 

Theorem \ref{nonvanishing} shows that $B^{(\lambda)}$ is zero-free on $\mathbb G_2\times\mathbb G_2$ precisely when $0<\lambda\leq2$. Since $\mathbb G_2$ is simply connected, a normalized global logarithmic branch of $B^{(\lambda)}$ may then be chosen, so that $(B^{(\lambda)})^\nu$ is well defined for every $\nu > 0$. In the Hardy-kernel case $\lambda=1$, every such positive power is a non-negative definite kernel, see Proposition \ref{inf div of Hardy ker}. Moreover, for $ \lambda \geq 1 $, $ ( \mathbb{A}^{(1, \nu)}, ( B^{(1)})^{\nu})$ is shown to be an analytic Hilbert module. For $0<\lambda\leq2$, $\lambda\neq 1$, however, the problem of determining exactly which powers $(B^{(\lambda)})^\nu$ are non-negative definite remains open.

If $\nu \in \mathcal{W}_{\mathbb{G}_2, a}(B^{(\lambda)})$, then $(\mathbb{A}^{(\lambda, \nu)}, (B^{(\lambda)})^\nu)$ is also a  homogeneous analytic Hilbert module. For different $(\lambda, \nu)$'s, these are not unitarily equivalent; see Theorem \ref{inequivofweightedbergker}.  To find homogeneous analytic Hilbert modules on $\mathbb G_2$ other than the ones  that have been already listed, we consider what one might call the ``weighted determinantal Bergman modules'' $(\mathbb{H}^{(\lambda, \nu)}, \det (B^{(\lambda)})^{\nu+2} \mathbb{B}^{(\lambda)})$, where  
$\mathbb{B}^{(\lambda)}$ is the $2\times 2$ curvature matrix of the weighted Bergman kernel $B^{(\lambda)}$ of $\mathbb{G}_2$. It is shown that these are homogeneous analytic Hilbert modules and are inequivalent among themselves. One of the main results of this second half of the paper, namely Theorem \ref{inequivofdetker}, says that in contrast to the transitive case, none of the Hilbert modules $(\mathbb{H}^{(\lambda, \nu)}_{\det}, \det((B^{(\lambda)})^{\nu+2} \mathbb{B}^{(\lambda)}))$ are equivalent to any of the Hilbert modules $(\mathbb{A}^{(\mu, \eta)}, (B^{(\mu)})^\eta)$. 

The computations of this paper described above extend to a much larger class of Hilbert modules. We intend to investigate these in future.  
Indeed, by starting with a pair of quasi-invariant kernels $K_1, K_2$, we get a plethora of examples by repeating two operations - taking powers, and multiplying by the determinant of the curvature - which give $K_1^\nu(\mathbb{K}_2^{[\lambda]})^\mu$ for any admissible choices of $\lambda,\mu,\nu$. For instance, 
the direct summands $\mathbb{A}_{\rm sym}^{(\mu)}(\mathbb{D}^2)$ consisting of the symmetric functions and $\mathbb{A}_{\rm anti}^{(\mu)}(\mathbb{D}^2)$ consisting of the anti-symmetric functions in the weighted Bergman space $\mathbb{A}^{(\mu)}(\mathbb{D}^2)$ provide such a pair of kernels.  

\subsection{Open Questions} 
We highlight two intriguing open questions below. The first concerns the source of homogeneous kernels in general; the second asks when the pointwise powers of a single kernel - the weighted Bergman kernel of $\mathbb{G}_2$ - gives rise to an analytic Hilbert module. 

Theorem \ref{cor:2.2} provides a recipe to construct all the homogeneous kernels from the projective cocycle $J$ (see Remark \ref{projco}, in particular Equation \eqref{cocycle} ) satisfying
 \[J ( g_0, z )^{ - 1 }  \overline{ J ( g_0, z )^{ - 1 }} = 1,\,\, g_0 \in \mathrm{stab}(z), z\in\Lambda ,\]
and a positive function $K_{\Lambda}$ on $\Lambda$. In all the examples described above, the factor $J$ appearing in Equation \eqref{eqn:1.1} is some function of the Jacobian of the derivative of an automorphism $g$ of $\mathbb{G}_2$ at $\boldsymbol{u} \in \mathbb{G}_2$.  These are easily verified to be cocycles, i.e., they satisfy Equation \eqref{cocycle}. This prompts the question of finding examples of cocycles on $\mathbb{G}_2$ that are not necessarily of this form. A related question is the following. Starting with an assignment of a non-negative definite kernel defined on the fundamental set, say $K_\Lambda(z,z)$, $z\in \Lambda$, if we define $K$ on all of $\mathbb{G}_2 \times \mathbb{G}_2$ by forcing quasi-invariance as in Equation \eqref{eqn:quasi-inv-funda}, then for what choices of $K_\Lambda$, the function $K$ is non-negative definite? 

For the other open problem, recall from the discussion above that $(B^{(\lambda)})^\nu$ is well defined as a sesqui-holomorphic function for every $\nu>0$ when $0<\lambda\leq 2$, and that determining exactly which of these powers are non-negative definite - the Wallach set $\mathcal{W}_{\mathbb{G}_2}(B^{(\lambda)})$ — is known only for $\lambda=1$, where $\mathcal{W}_{\mathbb{G}_2}(B^{(1)})=(0,\infty)$. For $\lambda>2$, however, $B^{(\lambda)}$ has simple zeros, so a global noninteger power cannot be defined; positive integer powers remain well defined and are automatically non-negative definite by the Schur product theorem, giving $\mathcal{W}_{\mathbb{G}_2}(B^{(\lambda)})=\mathbb{N}$ in this range. What remains open, then, is to determine $\mathcal{W}_{\mathbb{G}_2}(B^{(\lambda)})$ for the range $0<\lambda\leq2$, $\lambda\neq1$, and, whenever $\nu\in\mathcal{W}_{\mathbb{G}_2}(B^{(\lambda)})$, whether the reproducing kernel Hilbert space associated with $(B^{(\lambda)})^\nu$ is an analytic Hilbert module. The latter question also remains to be settled for $\lambda=1$ and $0<\nu<1$.
These problems amount to determining the generalized Wallach sets defined in \eqref{Wallach set Omega K} and \eqref{Wallach set Omega K analytic} associated with the weighted Bergman kernels on $\mathbb G_2$.

\section{Curvature criterion for homogeneous modules}

With all our assumptions on $(\mathcal H, K)$ and following the discussion of Section 5 in \cite{AKGM}, it is easy to determine a condition on the reproducing kernel ensuring the homogeneity of the commuting $d$-tuple $\boldsymbol{M}$. We first note that $\boldsymbol{M}_g$ on $\mathcal H$ is unitarily equivalent to the tuple of multiplication operators $ \widetilde{ \boldsymbol{ M } } $ by coordinate functions on $\mathcal H_g:= \{f\circ  g^{-1}: f\in\mathcal H\}$. The inner product on $\mathcal H_g$ is borrowed from $\mathcal H$ by making the map $U^1_g:f\mapsto f\circ g^{-1}$ from $\mathcal H$ to $\mathcal H_g$ unitary that leads to the equivalence. By definition, $\mathcal H_g$ is a reproducing kernel Hilbert space with 
$K_g:\Omega\times \Omega \rightarrow \mathbb C,$
given by the formula: 
\begin{equation} \label{kg}
K_g(z,w) := K(g^{-1}\cdot z, g^{-1}\cdot w) = K(g^{-1}(z), g^{-1}(w)), 
\end{equation}
as the reproducing kernel. An equivalent formulation would be to start with the non-negative definite kernel $K_g$ as defined in \eqref{kg} and to observe that the associated Hilbert space $\mathscr{H}$ obtained using  \cite[Equation (3.1) and (3.2)]{BM} coincides with $\mathcal H_g$. Note that $U^1_g$ is the extension of the mapping $K(\cdot, w)\mapsto K_g(\cdot, g(w))$, and $U^1_g \boldsymbol{M}_h {U^1_g}^* = \boldsymbol{\widetilde M}_{h\circ g^{-1}}$ for $h\in G$.

Fix $g\in G$. Taking $h=g$, we conclude that the coordinate multiplication tuple $\widetilde{\boldsymbol M}$ is bounded on $\mathcal H_g$. Moreover, the homogeneity of $\boldsymbol M$ is equivalent to the unitary equivalence of $\boldsymbol M$ on $\mathcal H$ and $\widetilde{\boldsymbol M}$ on $\mathcal H_g$, for every $g\in G$.

Since $1\in\mathcal H$ and $\boldsymbol M_g$ is bounded, we have 
$p\circ g=p(\boldsymbol M_g)1\in\mathcal H$
for every $p\in\mathbb C[\boldsymbol z]$. By the definition of $\mathcal H_g$, it follows that $\mathbb C[\boldsymbol z]\subseteq\mathcal H_g$. In particular, $\mathcal H_g$ contains a nonzero constant function, and hence
\[
K_g(\lambda,\lambda)>0,\quad \lambda\in\Omega.
\]
Furthermore, since $\mathcal H_g$ contains all polynomials, polynomial interpolation at finitely many distinct points shows that $K_g$ is strictly positive.

The unitary $U_g^1$ intertwines $\widetilde{\boldsymbol M}$ with $\boldsymbol M_g$. Since $g$ is invertible, it follows that
\begin{align*}
\dim\bigcap_{i=1}^n\ker(\widetilde M_i-w_iI)^*
&=
\dim\bigcap_{i=1}^n\ker(M_{g_i}-w_iI)^*\\
&=
\dim\bigcap_{i=1}^n
\ker\!\left(M_i-(g^{-1}(w))_iI\right)^*\\
&=1, \,\, w\in\Omega.
\end{align*}
Thus, the hypotheses of \cite[Theorem 3.7]{CS} are satisfied for $K$ and $K_g$. Consequently, for this $g\in G$, $\widetilde{\boldsymbol M}$ is unitarily equivalent to $\boldsymbol M$ if and only if
\begin{equation}\label{eqn:1.1}
K(g^{-1}(z), g^{-1}(w))= J(g^{-1}, z)^{-1} K(z, w) \overline{J(g^{-1}, w)^{-1}}
\end{equation}
for some non-vanishing function $J:G \times \Omega\rightarrow\mathbb C$ which is holomorphic on $\Omega$. 
Let  $U^2_g: \mathcal H_g \to \mathcal H$ be this unitary given  by: $f\mapsto J(g^{-1}, \cdot) f$, $f\in \mathcal H_g$. 
Thus, $\boldsymbol M$ is homogeneous if and only if Equation \eqref{eqn:1.1} holds for every $g\in G$.  
Moreover, the unitary $U_g=U_g^2 U_g^1$ intertwining $\boldsymbol M_g$ and $\boldsymbol M$ on $\mathcal H$ is  given by the formula \[(U_g f)(z) := (U^2_gU^1_g f) (z) = J (g^{-1}, z )  (f \circ g^{-1}) (z), \, f\in (\mathcal H, K),\, z\in \Omega.\]

\begin{theorem} \label{thm:2.1-Kquasi}
Let $\Omega\subseteq \mathbb{C}^m$ be a $G$-space, where $G$ is a subgroup of the bi-holomorphic automorphism group of $\Omega$. Let $(\mathcal{H}, K)$
be an analytic Hilbert module over the polynomial ring $\mathbb{C}[\boldsymbol{z}]$ consisting of holomorphic functions defined on $\Omega$. The commuting tuple $\boldsymbol{M}$ of multiplication by the coordinate functions $z_i$, $1\leq i \leq m$, on $(\mathcal{H}, K)$ is homogeneous under the group $G$ if and only if 
\[K(g (z), g (w))= {J(g, z)}^{-1} K(z, w) \overline{J(g, w)^{-1}},\, z,w\in \Omega,\]
for all $g\in G$.
\end{theorem}

\begin{remark}\label{projco}
Suppose that $G$ is a locally compact second countable group. Then a Borel function $m:G\times G \to \mathbb{T}$ is called a multiplier if 
\[m(g, 1) = m(1, g) = 1,\,\,m(g_1, g_2)m(g_1 g_2, g_3) = m(g_1, g_2g_3) m(g_2, g_3)\] 
for all $g, g_1, g_2, g_3 \in G$.

Let $X$ be a standard Borel space and let the $G$ action $G \times X \to X$ be Borel. Suppose that $J:G\times X \to \mathbb{C}\setminus \{0\}$ is jointly Borel. Then $J$ is said to be a {projective cocycle} if it satisfies   Wallach
\begin{equation} \label{cocycle}
J( g_1 g_2, x) = m(g_1,g_2) J(g_1, g_2 \cdot x) J(g_2, x), ~~~ g_1, g_2\in G, \, x \in X.
\end{equation}
If $J$ satisfies Equation \eqref{cocycle} with $m(g_1,g_2) = 1$ for all $g_1, g_2$ in $G$, then it is called a cocycle.  

Let $\mathcal{F}$ be the space of all complex valued Borel functions defined on $X$ and let $\mathcal{L}(\mathcal{F})$ be the algebra of all linear transformations on $\mathcal{F}$. If $J$ is a projective cocycle, then the map $U:G \to  \mathcal{L}(\mathcal{F})$ given by the formula
\begin{equation}\label{eqn:proj rep}
(U_g f) (x) = J(g^{-1}, x) f(g^{-1}(x)),\,\, g\in G,x\in X,     
\end{equation}
where $U_g:=U(g)$, is a projective homomorphism, that is, 
\[U_{g_1g_2} = m(g_2^{-1}, g_1^{-1}) U_{g_1} U_{g_2},\,\, g_1,g_2 \in G.\]

In particular, if $\mathcal{H}$ is a Hilbert space consisting of functions on $X$, $U_g$ defined by Equation \eqref{eqn:proj rep} leaves $\mathcal{H}$ invariant, and $U_g$ is unitary, then $U$ is said to be a projective multiplier representation.
\end{remark} 

Any commuting tuple of operators in the Cowen-Douglas class of rank $1$ is irreducible \cite[Proposition 1.18]{CD}. We show that analytic Hilbert modules possess this property as well.  
\begin{lemma}\label{irreducibility}
Let $(\mathcal{H}, K)$ be an analytic Hilbert module. Then the commuting tuple $\boldsymbol{M} = (M_1, \ldots, M_n)$ of the multiplication by the coordinate functions is irreducible.    
\end{lemma}

\begin{proof}
Since the tuple $\boldsymbol{M} = (M_1, \ldots, M_n)$ is irreducible if and only if $\boldsymbol{M}^* = (M_1^*, \ldots, M_n^*)$ is irreducible, it is enough to prove that $\boldsymbol{M}^*$ is irreducible. On the contrary, assume that $\boldsymbol{M}^*$ is not irreducible. Let $\mathcal{M}$ be a proper non-trivial joint reducing subspace of each $M_i^*$, $1 \leq i \leq n$. Suppose $K_{\mathcal{M}}$ and $K_{\mathcal{M}^\perp}$ are the reproducing kernels of $\mathcal{M}$ and $\mathcal{M}^\perp$, respectively.  

Since $w\longmapsto K_{\mathcal M}(\cdot,w)$  and  $w\longmapsto K_{\mathcal M^\perp}(\cdot,w)$
are anti-holomorphic $\mathcal H$-valued functions, and neither is identically zero, their respective zero sets have empty interior. Hence, there exists $w\in\Omega$ such that both  $K_{\mathcal{M}}(\cdot, w)$ and $K_{\mathcal{M^\perp}}(\cdot, w)$ are non-zero.
For such a $w$, $M_i^*K_{\mathcal M}(\cdot,w)
=\overline{w_i},K_{\mathcal M}(\cdot,w)$, and 
$M_i^*K_{\mathcal M^\perp}(\cdot,w)
=\overline{w_i},K_{\mathcal M^\perp}(\cdot,w)$, 
for each $i$. This implies that 
$$\dim \big (\bigcap_{i=1}^m \ker (M_i - w_i)^*\big ) \geq 2,$$
which is a contradiction.
\end{proof}
In general, if a homogeneous commuting tuple $\boldsymbol{M}$ is irreducible, then the unitary, say $U_g$, inducing the unitary equivalence between $\boldsymbol{M}$ and $\boldsymbol{M}_g$ can be chosen so that the map $g\mapsto U_g$ is a homomorphism. If $(\mathcal{H}, K)$ is a homogeneous analytic Hilbert module, then we prove below that this homomorphism must be a projective multiplier representation. 
\begin{prop}\label{existence of rep}
Suppose that $(\mathcal{H}, K)$ is an analytic Hilbert module,  $\boldsymbol{M}$ is the commuting tuple of multiplication by the coordinate functions and $G$ is a locally compact second countable group. Then $(\mathcal{H}, K)$ is $G$-homogeneous if and only if there exists a projective multiplier representation $U : G \to \mathcal{U}(\mathcal{H})$ such that $\boldsymbol{M}_g = U(g)^* \boldsymbol{M}U(g)$, $g \in G$        
\end{prop}
\begin{proof}
Since $(\mathcal{H}, K)$ is an analytic Hilbert module, from Lemma \ref{irreducibility} it follows that  $\boldsymbol{M}$ is irreducible. If $(\mathcal{H}, K)$ is $G$-homogeneous, then the existence of a projective unitary representation $U : G \to \mathcal{U}(\mathcal{H})$ satisfying $\boldsymbol{M}_g = U(g)^* \boldsymbol{M}U(g)$, $g \in G$ follows from \cite[Theorem 1.6]{MS}. The discussion preceding Theorem \ref{thm:2.1-Kquasi} shows that $U$ must be a multiplier representation.   
\end{proof}

For $ w \in \Omega $, there exist unique point $ z \in \Lambda $ and a biholomorphism $ g_{ z, w } \in G $ depending on $ z $ and $ w $, such that $ g_{ z, w } ( z ) = w $. Consequently, if $K$ is quasi-invariant under $G$, then it follows from Equation \eqref{eqn:1.1} that
\begin{equation} \label{eqn:quasi-inv-funda}
    K ( w, w ) = J ( g_{ z, w }, z )^{ - 1 } K ( z , z ) \overline{ J ( g_{ z, w }, z )^{ - 1 } } .
\end{equation}  
Also, note that if $z \in \Lambda$ and $g_0$ is in the stabilizer subgroup stab$(z)$ fixing $z$, then it follows from Equation \eqref{eqn:1.1} that
\begin{equation*}
    K ( z, z ) = J ( g_0, z )^{ - 1 } K ( z , z ) \overline{ J ( g_{ 0 }, z )^{ - 1 } } .
\end{equation*}
Moreover,  the restriction of $ K $ and $ J $ to $ \Lambda $ are enough to determine the value $K( w, w ) $, $w\in \Omega$, as shown below.

Assume that $K_\Lambda : \Lambda \to \mathbb C$ is any positive function and $J : G \times \Omega \to \mathbb C$ is a projective cocycle such that
\begin{equation}\label{stab relation}
    J ( g_0, z )^{ - 1 } \overline{ J ( g_0, z )^{ - 1 } } = 1,
\end{equation}
for every $z \in \Lambda$ and $g_0 \in \mbox{stab}(z)$. Suppose that $K_{\Omega} : \Omega \to \mathbb C$ is determined by setting    
\begin{equation}\label{def of K_Omega}
    K_\Omega ( w ) = J ( g_{ z, w }, z )^{ - 1 } K_\Lambda ( z ) \overline{ J ( g_{ z, w}, z )^{ - 1 } }, \,\,w  \in \Omega,\,\,z \in \Lambda,
\end{equation}
where $g_{ z, w } \in G$ such that $g_{z,w} (z) = w.$ We claim that  $K_{\Omega}$ given by Equation \eqref{def of K_Omega} is well-defined. In fact, if $\tilde{g} \in G$ is another element of $G$ mapping $z$ to $w$, then $\tilde{g}^{-1} \circ g_{z, w} \in \mbox{stab} (z)$ and therefore, 
\begin{align*}
J ( g_{ z, w }, z )^{ - 1 } K_\Lambda ( z ) \overline{ J ( g_{ z, w}, z )^{ - 1 } }
&= J ( \tilde{g} \circ \tilde{g}^{-1} \circ g_{ z, w }, z )^{ - 1 } K_\Lambda ( z ) \overline{ J ( \tilde{g} \circ \tilde{g}^{-1} \circ g_{ z, w}, z )^{ - 1 } }\\
&= J ( \tilde{g} , z )^{ - 1 } J ( \tilde{g}^{-1} \circ g_{ z, w }, z )^{ - 1 } K_\Lambda ( z ) \overline{ J (\tilde{g}^{-1} \circ g_{ z, w}, z )^{ - 1 }} \overline{ J ( \tilde{g}, z )^{ - 1 } }\\
&= J ( \tilde{g} , z )^{ - 1 }  K_\Lambda ( z )  \overline{ J ( \tilde{g}, z )^{ - 1 } }
\end{align*}
Here, the second equality holds because $J$ is a projective cocycle and the third equality holds because of Equation \eqref{stab relation}. Let us assume $K_\Omega$ to be a real analytic function and there exists a non-negative definite kernel  $K : \Omega \times \Omega \to \mathbb C$, holomorphic in the first variable and anti-holomorphic in the second variable,  such that $K(w,w) := K_{\Omega}(w),\,\,w \in \Omega$. Note that if $g \in G$ and $w \in \Omega$, then
\begin{flalign*}
K(w, w) &= J ( g_{ z, w }, z )^{ - 1 } K ( z, z ) \overline{ J ( g_{ z, w}, z )^{ - 1 } }\\
&= J ( g \circ g^{-1} \circ g_{ z, w }, z )^{ - 1 } K ( z, z ) \overline{ J ( g \circ g^{-1} \circ g_{ z, w}, z )^{ - 1 } }\\
&= J ( g , g^{-1}(w) )^{ - 1 } J ( g^{-1} \circ g_{ z, w }, z )^{ - 1 } K ( z, z ) \overline{ J (g^{-1} \circ g_{ z, w}, z )^{ - 1 }} \overline{ J ( g, g^{-1}(w) )^{ - 1 } }\\
&= J ( g , g^{-1}(w) )^{ - 1 }  K ( g^{-1}(w), g^{-1}(w) )  \overline{ J ( g, g^{-1}(w) )^{ - 1 } }.
\end{flalign*}
 Since $K_\Omega$ is well-defined, the validity of the fourth equality is evident. Let 
 $$\phi(z, w) = K(z, w) - J ( g , g^{-1}(z) )^{ - 1 }  K ( g^{-1}(z), g^{-1}(w) )  \overline{ J ( g, g^{-1}(w) )^{ - 1 } }.$$
Since $\phi$ is holomorphic in the first variable and anti-holomorphic in the second variable and $\phi(w, w) = 0$, from \cite[Proposition 1, page 10]{dang}, it follows that $\phi(z, w) = 0$. This shows that $K$ is quasi-invariant function with respect to $J$. We summarize all these observations in the theorem stated below. 

\begin{theorem} \label{cor:2.2}
Suppose that $(\mathcal{H}, K)$ is an analytic Hilbert module. Then $(\mathcal{H},K)$ is $G$-homogeneous if and only if
\[ K ( w, w ) = J ( g_{ z, w }, z )^{ - 1 } K ( z , z ) \overline{ J ( g_{ z, w }, z )^{ - 1 } }, z\in \Lambda,\, w\in \Omega, \]
where $g_{ z, w } \in G$ such that $w = g_{ z, w }(z)$, for  a projective cocycle $J : G \times \Omega \to \mathbb C$ that is holomorphic for any fixed $g\in G$ and  
\[J ( g_0, z )^{ - 1 }  \overline{ J ( g_0, z )^{ - 1 }} = 1,\,\, g_0 \in \mathrm{stab}(z) .\]
\end{theorem}

If $(\mathcal{H}, K)$ is an analytic Hilbert module, then it turns out that 
$\dim \cap_{ i = 1 }^m ( M_{ z_i }^* - \overline{ w_i } I ) = 1 $, $w\in \Omega$. Thus, the holomorphic function $ \bar{ w } \mapsto K ( \cdot, w ) $ gives rise to a structure of holomorphic vector bundle on 
$$ \{ ( \bar{ w }, f ) : f \in \cap_{ i = 1 }^m \ker ( M_{ z_i }^* - \overline{ w_i }\text{I} ) \} \subset \Omega^* \times \mathcal{ H }, $$
making it a hermitian holomorphic line bundle -- denoted by $ \mathsf{ L }_\mathcal{ H } $ -- over $ \Omega^* : = \{ w : \bar{ w } \in \Omega \} $ with the hermitian structure induced from $ \mathcal{ H } $. 

We recall below the fundamental theorem of Cowen and Douglas in the language of Hilbert modules over the polynomial ring $\mathbb{C}[\boldsymbol{z}]$. The original theorem was for an operator $T$ in the Cowen-Douglas class $B_1(\Omega)$. When $\Omega\subset \mathbb{C}^m$, $m>1$, the module action is induced by the $m$-tuple $\boldsymbol{M}$ of multiplication by coordinate functions. The relationship between commuting $m$-tuple $\boldsymbol{M}$ and the curvature of the associated line bundle is now more complicated, but the result is still the same, see \cite{Bolyi}. The theorem stated below is a variant of a theorem due to Chen and Douglas that appeared in  \cite{ChenDouglas}. A proof for this version can be easily made up by adapting the original proof in \cite{ChenDouglas}. Hence, it is omitted. However, we point out that here, the Hilbert modules are assumed to be analytic, while in the paper \cite{ChenDouglas}, the Hilbert modules are assumed to be locally free. Nonetheless, these notions are closely related and do not affect the result below.  

\begin{theorem}[Cowen-Douglas] \label{module equiv}
Two analytic Hilbert modules $ \mathcal{ H }_1 $ and $ \mathcal{ H }_2 $ are unitarily equivalent if and only if the respective hermitian holomorphic line bundles $ \mathsf{ L }_{\mathcal{ H }_1} $ and $ \mathsf{ L }_{\mathcal{ H }_2 } $ are isomorphic as hermitian holomorphic line bundles. Moreover, the equivalence of the line bundles is completely determined by the curvature matrix $\mathbb{K}$.
\end{theorem}

Here $\mathbb{K}$ is the $m \times  m$ matrix of real analytic functions defined from the coefficients of the curvature $\mathcal{K}$ of the line bundle $ \mathsf{ L }_{\mathcal{H}}$, see  \eqref{curvaturedefn} and \eqref{eqn:1.4}.

We now formulate a criterion for $ G $-homogeneity of an analytic Hilbert module $\mathcal{H}$ involving the curvature matrix $\mathbb{K}$ associated with the corresponding line bundle $\sf L_{\mathcal{H}}$. 

\begin{prop}\label{curvature transformation}
Let $\mathcal{H}$ be an analytic Hilbert module and $ g : \Omega \rightarrow \Omega $ be a biholomorphism. Then $\boldsymbol{M}_g$ is unitarily equivalent to $\boldsymbol{M}$ if and only if 
\begin{equation} \label{eqn:2.1}
\mathbb{K}(w)=\left(Dg(g^{-1}(w))^{\operatorname{tr}}\right)^{-1}\mathbb{K}(g^{-1}(w))\left(\overline{Dg( g^{-1}(w))}\right)^{-1},\end{equation}
where $ D g $ denotes the derivative of $ g $.
\end{prop}

\begin{proof}
The Hilbert space $\mathcal{H}_g$ with the reproducing kernel $ K_g : \Omega \times \Omega \to \mathbb{ C } $, given by Equation \eqref{kg}, is an analytic Hilbert module. Suppose $\mathsf L_{\mathcal{H}}$ and $\mathsf L_{\mathcal{H}_g}$ are the corresponding hermitian holomorphic line bundles with the holomorphic frames $\{K(\cdot, \bar w) : w \in \Omega^*\}$ and $\{K_g(\cdot, \bar w) : w \in \Omega^*\}$, respectively. If $\mathbb K_g$ is the curvature matrix of the line bundle $\mathsf L_{g}$, using the chain rule, we have
\begin{eqnarray}
\nonumber {\mathbb {K}}_{g}(w)&=& \big(\!\!\big( -\frac{\partial^2}{\partial w_i\partial \overline{w}_j}\log K_g(w, w) \big) \!\!\big)_{i,j=1}^m\\
\nonumber &=& \big(\!\!\big( -\frac{\partial^2}{\partial w_i\partial \overline{w}_j}\log K(g^{-1}(w), g^{-1}(w)) \big) \!\!\big)_{i,j=1}^m\\
\nonumber &=& \big(Dg( g^{-1}(w))^{\operatorname{tr}}\big)^{-1}\mathbb{K}(g^{-1}(w))\left(\overline{Dg( g^{-1}(w))}\right)^{-1},
\end{eqnarray}
where $\mathbb K$ is the curvature matrix of the line bundle $\mathsf L_{\mathcal{H}}$.

Note that the line bundles $\mathsf L_{\mathcal{H}}$ and $\mathsf L_{\mathcal{H}_g}$ are isomorphic if and only if $\mathbb K(w) = \mathbb K_g(w)$ for every $w \in \Omega$.  Thus, by Theorem \ref{module equiv}, the Hilbert modules $\mathcal{H}$ and $\mathcal{H}_g$ are unitarily equivalent if and only if  $\mathbb K(w) = \mathbb K_g(w)$ for $w \in \Omega$. Consequently, $\boldsymbol{M}_g$ and $\boldsymbol{M}$ are unitarily equivalent if and only if  Equation \eqref{eqn:2.1} holds.
\end{proof}

Now taking the determinant in both sides of Equation \eqref{eqn:2.1},   we obtain the following corollary.
\begin{cor} \label{det transformation formula}
Let $\mathcal{H}$ be a $G$-homogeneous analytic Hilbert module with the reproducing kernel $ K : \Omega \times \Omega \rightarrow \mathbb C  $. Then, we have  
    $$ \det \mathbb{ K } ( z ) = \det D g ( z ) \det \mathbb{ K } ( g ( z ) ) \overline{ \det D g ( z ) }, ~ z \in \Omega, g \in G, $$
    where $ \mathbb{ K } $ is the curvature matrix of the corresponding hermitian holomorphic line bundle $\mathsf L_{\mathcal{H}}$.
\end{cor}

Recall that the closed subspace of $ L^2 ( \Omega, dV ) $ consisting of holomorphic functions on $ \Omega $, where $ dV $ denotes the Lebesgue measure on $ \Omega $, is a reproducing kernel Hilbert space, known as the Bergman space over $ \Omega $. The corresponding reproducing kernel is called the Bergman kernel of $ \Omega $. It is well known that the Bergman kernel of $ \Omega $ is a quasi-invariant kernel with respect to Aut$ ( \Omega ) $ (cf. \cite[p. 191]{GMIJPAM}), namely, it satisfies 
Equation \eqref{eqn:1.1} with the cocycle $ J ( g, z ) := \det Dg ( z ) $, $ g \in \text{ Aut} ( \Omega ) $ and $ z \in \Omega $. Thus, the corollary below is an immediate consequence of Corollary \ref{det transformation formula}.
\begin{cor}  \label{multipleBerg}
Let $\Omega$ be a transitive $G$-space. Assume that  $(\mathcal{H}, K)$ is a $G$-homogeneous analytic Hilbert module.  Then $\det \mathbb K(z)$ is a positive scalar multiple of the Bergman kernel on $\Omega$. 
\end{cor} 

As a consequence of Proposition \ref{curvature transformation}, we conclude that an analytic Hilbert module $ \mathcal{ H } $ is $ G $-homogeneous if and only if the associated curvature matrix satisfies  
Equation \eqref{eqn:2.1} for every $ g \in G $. However, since $ G \cdot \Lambda = \Omega $, where $ \Lambda $ is a fundamental set,
Equation \eqref{eqn:2.4} provides a weaker criterion for $ G $-homogeneity of $ \mathcal{ H } $. Moreover,  Equation \eqref{eqn:2.4} given below can be used as a formula to compute the curvature matrix associated with such module from its restriction to the fundamental set satisfying Equation \eqref{eqn:2.5}.

\begin{theorem} \label{thm:2.2}
    Suppose $\mathcal{H}$ is an analytic Hilbert module. Then $\mathcal{H}$ is $G$- homogeneous if and only if 
\begin{equation}\label{eqn:2.4}
\mathbb {K}(w) = {\big (D g(z)^{\operatorname{tr}}\big )}^{-1}  \mathbb {K}(z) \big (\overbar{D g(z)}\big )^{-1},  z\in \Lambda,\,w\in \Omega \setminus \Lambda, 
\end{equation}
for $g \in G$ such that $g(z) = w$. Moreover, if we also have $h(z) = w$ for $h\in G$, $h\ne g$,  then 
\begin{equation} \label{eqn:2.5}
\mathbb K(z) = D(g^{-1}\circ h) (z)^{\operatorname{tr}} \mathbb K(z)  \overline{D(g^{-1}\circ h)(z)}, ~~~ z\in \Lambda. 
\end{equation}
\end{theorem}

\begin{proof}
If $ \mathcal{H} $ is $ G $-homogeneous, by Proposition \ref{curvature transformation}, the curvature $ \mathbb{ K } $ of the holomorphic hermitian line bundle $\mathsf L_{\mathcal{H}} $ associated to $ \mathcal{H} $ satisfies Equation \eqref{eqn:2.1}. In particular, any $ z \in \Lambda $ and $ w \in \Omega \setminus \Lambda $, satisfying $ g ( z ) = w $ for some $ g \in G $, yields Equation \eqref{eqn:2.4}. 

Conversely, for any $w\in \Omega$, there exists a $z \in \Lambda$ and $g_0 \in G$ such that $g_0(z) = w$. (Note that $g_0$ is not necessarily unique.) Let $g$ be a fixed but arbitrary element of $G$. Thus $g \circ g^{-1} \circ g_0(z) = w$ and $ g^{-1} \circ g_0 ( z ) = g^{-1} (w)$. We are given:
\begin{align*}
\mathbb K(w) &= \big ( D(g\circ g^{-1} \circ g_0 ) ( z )^{\operatorname{tr}} \big )^{-1} \mathbb K(z) \big ( \overline{D(g\circ g^{-1} \circ g_0 ) (z) } \big )^{- 1}\\
&= \big ( [ D g ( g^{ - 1 } \circ g_0 ( z ) ) D ( g^{ - 1 } \circ g_0 ) ( z ) ]^{ \operatorname{tr} }\big )^{ - 1 } \mathbb{ K } ( z ) \big ( \overline{ D g ( g^{ - 1 } \circ g_0 ( z ) ) } \overline{ D ( g^{ - 1 } \circ g_0 ) ( z ) } \big )^{ - 1 }\\
&= \big ( ( D g ( g^{ - 1 } \circ g_0 ( z ) ) )^{ \operatorname{tr}} \big )^{ - 1 } \mathbb{ K } ( g^{-1} ( w ) ) \big ( \overline{  D g ( g^{ - 1 } \circ g_0 ( z ) ) } \big )^{ - 1 }\\
&= \big ( D g ( g^{ - 1 } ( w ) )^{ \operatorname{tr} } \big )^{ - 1 } \mathbb{ K } ( g^{ - 1 } ( w ) ) \big ( \overline{ D g ( g^{ - 1 } ( w ) ) } \big )^{ - 1 },
\end{align*}
where to verify the third equality, we note that $g^{-1} \circ g_0 ( z ) = g^{-1} (w)$ and consequently, substituting $g\circ g_0$ in \eqref{eqn:2.4}, we have 
\[ \mathbb{ K } ( g^{-1} ( w ) ) = \big (  ( D ( g^{ - 1 } \circ g_0 ) ( z ) )^{ \operatorname{tr} } \big )^{ - 1 } \mathbb{ K } ( z ) \big ( \overline{ D ( g^{ - 1 } \circ g_0 ) ( z ) } \big )^{ - 1 } . \]
Substituting the value of $\mathbb K(z)$ from this equality completes the verification of the third equality in the string of equalities above. 
 Now applying Proposition \ref{curvature transformation}, we conclude that $\boldsymbol{M}_g$ is unitarily equivalent to $\boldsymbol{M}$ completing the first half of the proof. 

Finally, equating the values of $ \mathbb{ K } ( w ) $ obtained by choosing $ g $ and $h$ with $g(z)=w=h(z)$, $z\in \Lambda$, in \eqref{eqn:2.4}, we verify the equality in \eqref{eqn:2.5}.
\end{proof}

\begin{remark}
From Theorem \ref{thm:2.2}, we observe that the values of $\mathbb{K}$ on $\Omega$  can be determined from its values on $\Lambda$, provided $ \mathbb{ K } ( z ) $ at each $ z \in \Lambda $ satisfies the identity given in Equation \eqref{eqn:2.5} for every pair $g, h$ in $G$ with $g(z) = w = h(z)$. Therefore, the values of $\mathbb{K}$ on $\Lambda$  determine the unitary equivalence class of the associated homogeneous analytic Hilbert module.
\end{remark}

Let $ \mathcal{ H } $ be a $ G $-homogeneous analytic Hilbert module with the reproducing kernel $ K : \Omega \times \Omega \to \mathbb{ C } $. Thus, $K$ is a holomorphic function in both the variables when thought of as a function on $\Omega\times \Omega^*$.
Therefore,  $\log K$ is a well-defined function on $\Omega\times \Omega$, holomorphic in the first variable and anti-holomorphic in the second variable if $\Omega$ is simply connected and $K$ is non-vanishing on $\Omega$. We make these assumptions in what follows. Now, the function 
$$ \mathbb{ K }^{ [ \lambda ] } ( z, w ) := K ( z, w )^{ \lambda + 2 } \left( \!\! \left( \frac{ \partial^2 }{ \partial z_i \partial \overline{ w_j } } \log K ( z, w ) \right) \!\! \right)_{ i, j = 1 }^m, ~~~ z, w \in \Omega, $$
is an well-defined non-negative definite function taking values in $\mathbb{C}^{m\times m}$.

In general, $(\mathcal{H}, K^\lambda)$ needs not be an analytic Hilbert module even if $\lambda \in \mathcal{W}_{\Omega}(K)$. Let $\mathcal{W}_{\Omega, a}(K)$ be the subset of $\mathcal{W}_{\Omega}(K)$ such that  $(\mathcal{H},  K^{ \lambda})$ is an analytic Hilbert module for each $\lambda \in\mathcal{W}_{\Omega, a}$. Note that $ \mathbb{ N } \subset \mathcal{ W }_{ \Omega, a } ( K ) $. 
We show that each such $ \mathbb{ K }^{ [ \lambda ] } $ gives rise to a $ G $-homogeneous analytic Hilbert module, provided $\lambda \in \mathcal{W}_{\Omega, a}(K)$.

\begin{prop}\label{homog of K^[lambda]}
Suppose that $(\mathcal{ H }, K)$ is a $G$-homogeneous analytic Hilbert module.  If $ \lambda \in \mathcal{ W }_{ \Omega, a} ( K ) $, then the Hilbert module $( \mathbb{H}^{[\lambda]}, \mathbb K^{[\lambda]} ) $  equipped with the natural action of $\mathbb{C}[\boldsymbol{z}]$ is also homogeneous and contains $\mathbb C[\boldsymbol{z}]\otimes \mathbb C^m$ as a dense subset.
\end{prop}

\begin{proof}
Boundedness of the multiplication operators by the coordinate functions on $ \mathbb{H}^{[\lambda]} $ follows from the fact that the multiplication operators by the coordinate functions on the Hilbert space associated to the non-negative definite kernels $ K^\lambda $ are bounded.

Now, since $ \mathcal{ H } $ is $ G $-homogeneous, there exists a $ J : G \times \Omega \to \mathbb{ C } \setminus \{ 0 \} $ which, for each fixed $g$, is holomorphic on $ \Omega $ such that $K$ is quasi-invariant with respect to $J$: 
\begin{equation} \label{eqn:2.7} 
K ( z, w )^{ ( \lambda + 2 ) } = J ( g, z )^{ ( \lambda + 2 ) } K ( g ( z ), g ( w ) )^{ ( \lambda + 2 ) } (\overline{ J ( g, w ) } )^{ ( \lambda + 2 ) }, ~ g \in G. 
\end{equation}
Consequently, using Equation \eqref{eqn:2.1} we have 
\begin{eqnarray} \label{eqn:2.7.1}
\mathbb{ K }^{ [ \lambda ] } ( z, w ) 
& = & \big ( J ( g, z )^{ ( \lambda + 2 ) }  D g ( z ) \big )  \mathbb{ K }^{ [ \lambda ] } ( g ( z ), g ( w ) ) \overline{ \big ( J ( g, w )^{ ( \lambda + 2 ) }  D g ( w ) \big ) }
\end{eqnarray}
verifying that $ \mathbb{ K }^{ [ \lambda ] } $ is quasi-invariant and hence $ \mathbb{H}^{[\lambda]}$ is homogeneous.

Next, we note that the function 
$$ \mathbb{ K }^{ [0] } ( z, w ) := K ( z, w )^{ 2 } \left( \!\! \left( \frac{ \partial^2 }{ \partial z_i \partial \overline{ w_j } } \log K ( z, w ) \right) \!\! \right)_{ i, j = 1 }^m, ~~~ z, w \in \Omega, $$
is well-defined and is non-negative definite \cite[Corollary 2.4]{GM}. So if $(\mathcal{ H }, K)$ is analytic, the Hilbert space $\mathbb{H}^{[0]}$ determined by it contains the set of polynomials $\mathbb C[\boldsymbol{z}]\otimes \mathbb C^m$ as a dense subset \cite[Proposition 3.4]{GM}. 
Since $\lambda \in \mathcal{ W }_{ \Omega, a } ( K )$, the Hilbert space determined by the non-negative definite kernel $K^\lambda$  contains $\mathbb C[z]$ as a dense subset and therefore, $\mathbb C[z]\otimes \mathbb C^m$ is dense in $\mathbb{H}^{[\lambda]}$. This completes the proof.
\end{proof}

We now prove that $\det \mathbb K^{[\lambda]}$ gives rise to a $G$-homogeneous analytic Hilbert module whenever $\lambda \in \mathcal{ W }_{ \Omega, a } ( K )$. First, we need the following lemmas to accomplish this. In all of the lemmas below, we take $ \Omega $ to be a bounded domain in $ \mathbb{ C }^m $, $ K : \Omega \times \Omega \to \mathcal{ M }_r ( \mathbb{C} ) $ a reproducing kernel with the associated reproducing kernel Hilbert space $ \mathsf{ H } $ of holomorphic functions on $ \Omega $ and taking values in $ \mathbb{ C }^r $.

\begin{lemma} \label{lem:2.5}
    $ \det K : \Omega \times \Omega \rightarrow \mathbb{ C } $ is a non-negative definite kernel.
\end{lemma}

\begin{proof}
    Since $ K $ is a non-negative definite kernel taking values in $ \mathcal{M}_r ( \mathbb{C} ) $, by Kolmogorov's decomposition, there exists a function $ H : \Omega \rightarrow \mathcal{ B } ( \mathcal{ X }, \mathbb{ C }^r ) $ 
    such that
    $$ K ( z, w ) = H ( z ) H ( w )^*, ~~~ z, w \in \Omega  $$
 for some Hilbert space $ \mathcal{ X } $. Now  write $ H ( z ) = ( H_1 ( z ), \hdots, H_r ( z ) )^{\operatorname{tr}} $ where $ H_1 ( z ), \hdots, H_r ( z ) \in \mathcal{ X }^* $. Define, for each $ z \in \Omega $, the linear operator $ F ( z ) : \Omega \rightarrow \bigwedge^r \mathcal{ X } $ by 
 $$ F ( z ) : = H_1 ( z )^* \wedge \cdots \wedge H_r ( z )^* . $$
 Setting $ E ( z ) = F ( z )^* $ for $ z \in \Omega $, we have that
 $$ \det K ( z, w ) = E ( z ) E ( w )^*, ~~~ z, w \in \Omega, $$
 verifying that the function $ E : \Omega \rightarrow \mathcal{ B } ( \bigwedge^r \mathcal{ X }^*, \mathbb{ C } ) $ yields the Kolmogorov decomposition for $ \det K $. Consequently, $ \det K $ is a non-negative definite kernel.
\end{proof}

\begin{lemma}\label{lem:2.6}
If the set of all polynomials $\mathbb C[z] \otimes \mathbb C^r$ is dense in $\mathsf{H}$, then so is the set of all polynomials $\mathbb C[z]$ in the Hilbert space $\mathsf{H}_{\det}$ where $\mathsf{H}_{\det}$ is the reproducing kernel Hilbert space with the reproducing kernel $\det K$. 
\end{lemma}

\begin{proof}
Since the similar line of arguments as in the case of $ r = 2 $ works for any $ r \in \mathbb{ N } $, here, for simplicity, we only present the proof of the statement for $ r = 2 $. 

Consider the closed linear span of the set $\{K(\cdot, w)e_1 \wedge K(\cdot, w)e_2 : w \in \Omega\}$ in $\mathsf{H} \otimes \mathsf{H}$ and denote it by $\bigwedge^2 \mathsf{H}$. First, we claim that $f \wedge g \in \bigwedge^2 \mathsf{H}$ for any $f, g \in \mathsf{H}$. Indeed, if $f, g \in \mathsf{H}$, there exist sequences $\{x_n\}$, $\{y_n\}$ in the subspace generated by the set $\{K(\cdot, w)e_i : w \in \Omega, i = 1, 2\}$ such that $x_n \to f$ and $y_n \to g$. It is then easily verified that $x_n \wedge y_n \in \bigwedge^2 \mathsf{H}$ and $x_n \wedge y_n \to f \wedge g$ verifying that $f \wedge g \in \bigwedge^2 \mathsf{H}$. Thus, $\{f \wedge g : f, g \in \mathsf{H}\} \subset \bigwedge^2 \mathsf{H}$. In fact, it is dense in $\bigwedge^2 \mathsf{H}$. This, in particular, implies that $\{p \wedge q : p, q \in \mathbb C[z] \otimes \mathbb C^2\}$ is densely contained in $\bigwedge^2 \mathsf{H}$ since the set $\mathbb C[z] \otimes \mathbb C^2$ is dense in $\mathsf{H}$.

Now recalling the definition of the canonical hermitian product on $ \bigwedge^2 \mathsf{ H } $ induced from $ \mathsf{ H } $, it can be seen, for $f, g \in \mathsf{H}$, that  
$$\left\langle f \wedge g, \frac{1}{2} K(\cdot, w)e_1 \wedge K(\cdot, w)e_2 \right\rangle = f(w) \wedge g(w), \hspace{0.1in} w \in \Omega .$$
Since $\{f \wedge g : f, g \in \mathsf{H}\}$ is dense in $ \bigwedge^2 \mathsf{ H } $, it shows that the family $\{\frac 1 2 K(\cdot, w)e_1 \wedge K(\cdot, w)e_2 : w \in \Omega\}$ has the reproducing property. Consequently, the reproducing kernel of $\bigwedge^2 \mathsf{H}$ can be computed as follows.
\begin{flalign}\label{eqn repro det}
 \nonumber\left\langle \frac{1}{2} K(\cdot, w)e_1 \wedge K(\cdot, w)e_2, \frac{1}{2} K(\cdot, z)e_1 \wedge K(\cdot, z)e_2 \right\rangle
 \nonumber &= \frac{1}{2} K(z, w)e_1 \wedge K(z, w)e_2\\
  &= \frac{1}{2} \det K(z, w).
\end{flalign}
This proves that the set of all polynomials $\mathbb C[z]$ is dense in the Hilbert space $\mathsf{H}_{\det}$.
\end{proof}

\begin{lemma}\label{lem:2.7}
If the multiplication operators by the coordinate functions are bounded on $\mathsf{H}$, then so are the multiplication operators by the coordinate functions on the Hilbert space $\mathsf{H}_{\det}$, where $\mathsf{H}_{\det}$ is the reproducing kernel Hilbert space with the reproducing kernel $\det K$. 
\end{lemma}

\begin{proof}
Let $\bigwedge^r \mathsf{H}$ be the closed linear span of $ \{ K ( \cdot, w ) e_1 \wedge \cdots \wedge K ( \cdot, w ) e_r : w \in \Omega \} $ in $\mathsf{H} \otimes \cdots \otimes \mathsf{H}$ ($ r $ times). For each $1 \leq i \leq m$, suppose $M_{i}$ is the multiplication by the coordinate function $z_i$ on the Hilbert space $\mathsf{H}$. A direct computation yields 
 \begin{equation}\label{bdd 1}
  \left(M_{i}^* \otimes I \right)  \left(K ( \cdot, w ) e_1 \wedge \cdots \wedge K ( \cdot, w ) e_r \right) = \bar w_i \left(K ( \cdot, w ) e_1 \wedge \cdots \wedge K ( \cdot, w ) e_r \right)
 \end{equation}
 for every $w = (w_1, \cdots, w_m)$ in $\Omega$ and $1 \leq i \leq m$. This, in particular, implies that $\bigwedge^r \mathsf{H}$ is an invariant subspace of $M_{i}^* \otimes I$ for each $i$.
 
 Now it follows from Equation \eqref{eqn repro det} that the map $U : \bigwedge^r \mathcal{H} \to \mathcal{H}_{\det}$ given by 
 $$U\left(K ( \cdot, w ) e_1 \wedge \cdots \wedge K ( \cdot, w ) e_r\right) = \sqrt{r!}\det K(\cdot, w),\,\,w \in \Omega,$$
defines a unitary operator. Moreover, it can be seen that 
\begin{equation}\label{bdd 2}
 U \left(M_{i}^* \otimes I \right) U^* \det K(\cdot, w) = \bar w_i \det K(\cdot, w), 1 \leq i \leq m .
 \end{equation}
 This identity, however, implies that $U \left(M_{i}^* \otimes I \right) U^*$ is the adjoint of the multiplication operator by the coordinate function $z_i$ on the Hilbert space $\mathsf{H}_{\det}$. In particular, multiplication operators by the coordinate functions are bounded on $\mathsf{H}_{\det}$.
\end{proof}

Given a $ G $-homogeneous analytic Hilbert module with the reproducing kernel $ K $ and $ \lambda \in \mathcal{ W }_{ \Omega, a } ( K ) $, the analytic Hilbert module associated to the kernel $ \det \mathbb{ K }^{ [ \lambda ] } $ turns out to be $ G $-homogeneous as shown in the following theorem. 

\begin{theorem}\label{det theorem}
If $\mathcal{ H }$ is a $G$-homogeneous analytic Hilbert module with the reproducing kernel $ K : \Omega \times \Omega \to \mathbb C$ and $ \lambda \in \mathcal{ W }_{ \Omega, a } ( K ) $, then the Hilbert space $ \mathbb{H}^{[\lambda]}_{\det} $ determined by the kernel $ \det \mathbb K^{[\lambda]}$ is also an analytic Hilbert module that is $ G $-homogeneous. 
\end{theorem}

\begin{proof}
Lemma \ref{lem:2.5} and Lemma \ref{lem:2.6} yield the proof of the fact that $ \det \mathbb{ K }^{ [ \lambda ] } $ is a non-negative definite kernel on $ \Omega $ and $ \mathbb C [\boldsymbol{z}] $ is dense in $\mathbb{H}^{[\lambda]}_{\det}$, respectively. 

The proof of the fact that the multiplication operators by the coordinate functions on $ \mathbb{H}^{[\lambda]}_{\det} $ are bounded follows form Lemma \ref{lem:2.7}.

Finally, the homogeneity of $ \mathbb{H}^{[\lambda]}_{\det} $ follows from Proposition \ref{homog of K^[lambda]}. Indeed, taking determinant in both sides of Equation \eqref{eqn:2.7.1}, we obtain
\begin{eqnarray*}
\det \mathbb{ K }^{ [ \lambda ] } ( z, w ) & = & J ( g, z )^{ m ( \lambda + 2 ) } K ( g ( z ), g ( w ) )^{ m ( \lambda + 2 ) } ( \overline{ J ( g, w ) } )^{ m ( \lambda + 2 ) } \times\\
&  & \phantom{symmetrizedsymsymm} \det D g ( z ) \det \mathbb{ K } ( g ( z ), g ( w ) ) \overline{ \det D g ( w ) } \\
& = & \big ( J ( g, z )^{ m ( \lambda + 2 ) } \det D g ( z ) \big ) \det \mathbb{ K }^{ [ \lambda ] } ( g ( z ), g ( w ) ) \overline{ \big ( J ( g, w )^{ m ( \lambda + 2 ) } \det D g ( w ) \big ) }
\end{eqnarray*}
verifying that $ \det \mathbb{ K }^{ [ \lambda ] } $ is quasi-invariant.
\end{proof}

\section{Homogeneous analytic Hilbert modules over symmetrized bidisc}
Let $\mathbb D\subset \mathbb C$ be the open unit disc and $\mathbb T$ be the unit circle.  We let M\"{o}b denote the bi-holomorphic automorphism group of $\mathbb D$. Thus, $\varphi\in \text{M\"{o}b}$ is taken to be of the form: $\varphi_{t, \alpha}(z) = t \tfrac{\alpha - z}{1-\overbar{\alpha} z}$, $t\in \mathbb T$ and $\alpha\in \mathbb D$. The unique involution interchanging $0$ and $\alpha$ is of the form $\varphi_{1, \alpha}$ and we set $\varphi_\alpha:=\varphi_{1, \alpha}$, $\alpha\in \mathbb D$. 
Also, $\{\varphi_{t,0}:t\in \mathbb T\}$ is the stabilizer subgroup of $0$ in M\"{o}b. For brevity and when there is no possibility of confusion, we set $\varphi_t:= \varphi_{t,0}$, $t\in \mathbb T$. 
 
Let $\mathbb G_2:=\{(z_1+z_2,z_1 z_2)\mid (z_1, z_2) \in \mathbb D^2\}$ be the symmetrized bidisc. Thus $\mathbb G_2$ is the image of the bidisc under the symmetrization map $\boldsymbol s: \mathbb C^2 \to \mathbb C^2$, $s(z_1,z_2) := (z_1 + z_2, z_1 z_2)$.  
Evidently, if $f:\mathbb D^2 \to \mathbb C$ is a holomorphic function which is symmetric, i.e., $f(z_1,z_2) = f(z_2,z_1)$, then it defines a  holomorphic function $\tilde f :\mathbb G_2\to \mathbb C$ by the formula: $\tilde f(\boldsymbol s(z_1, z_2)) = f(z_1,z_2)$, $z_1,z_2 \in \mathbb D$. It follows that any
holomorphic function $h:\mathbb D\to \mathbb D$ defines a new 
holomorphic function $\tilde{h}: \mathbb G_2 \to \mathbb G_2$ by setting $\tilde{h}\left(\boldsymbol s\left(z_{1},z_{2}\right)\right) = \boldsymbol s\left(h\left(z_{1}\right), h\left(z_{2}\right)\right)$.  Also, $\tilde{h}(\mathcal{R}) \subseteq \mathcal{R}$, where $\mathcal{R} =\{(2z, z^2)\mid z\in \mathbb D\}$ is the image under the symmetrization map $\boldsymbol s$ of the diagonal set $\triangle:=\{(z,z)\mid z\in \mathbb D\}\subset \mathbb D^2$. In particular, $\tilde{\varphi}$, $\varphi\in \mbox{M\"{o}b}$, defines a holomorphic function on $\mathbb{G}_2$. Moreover, $\tilde{\varphi}^{-1} = \widetilde{\varphi^{-1}}$ and  therefore, $\tilde{\varphi}$ is an automorphism of $\mathbb G_2$.
In the paper \cite{JP}, Jarnicki and Pflug prove that 
\begin{equation}
\text{Aut}(\mathbb{G}_{2}) =  \{\tilde{\varphi}:\varphi \in  \mbox{M\"{o}b}\}, 
\end{equation}
where $\text{Aut}(\mathbb{G}_{2})$ is the group of the bi-holomorphic automorphisms of the symmetrized bidisc $\mathbb G_2$. This set of maps acts transitively on $\mathcal{R}$. The action of the bi-holomorphic automorphism group $\text{Aut}(\mathbb{G}_{2})$ on $\mathbb G_2$, however, is not transitive.  It is not hard to check that the set $\Lambda = \{(r, 0) : 0 \leq r < 1\}$ is a fundamental set for this action, see \cite[Theorem 2.1]{BBM}.

Let $ \mathcal{ H } $ be an analytic Hilbert module of holomorphic functions on $ \mathbb{ G }_2 $ with the reproducing kernel $ K : \mathbb{ G }_2 \times \mathbb{ G }_2 \to \mathbb{ C } $. In this section, we find a necessary and sufficient condition for $ \mathcal{ H } $ to be Aut$(\mathbb G_2)$-homogeneous in terms of the curvature matrix of the line bundle $ \mathsf{ L }_{ \mathcal{ H } } $ associated to $ \mathcal{ H } $ using Theorem \ref{thm:2.2} with $ \Omega = \mathbb{ G }_2 $ and $ G = $ Aut$(\mathbb G_2)$. We begin by computing $\mbox{stab}(r, 0)$ for $0 < r < 1$. 

\begin{lemma}\label{fixing elements in aut}
For $0 < r < 1$, $\tilde{\varphi}_r \in$ Aut$(\mathbb G_2)$ is the only non-trivial element in $\rm{stab}(r, 0)$, where $\varphi_r$ is the involution in M\"ob mapping $r$ to $0$. 
\end{lemma}

\begin{proof}
A direct computation verifies that $\tilde{\varphi}_r(r, 0) = (r, 0)$. Now, suppose $\tilde{\varphi}$ is such that $\tilde{\varphi}(r,0) = (r, 0)$. Note that 
$$\tilde{\varphi}(r, 0) = (\varphi(r) + \varphi(0), \varphi(r) \varphi(0)).$$
By our assumption, $\varphi(r) + \varphi(0) = r$ and $\varphi(r)\varphi(0) = 0$. The condition $\varphi(r)\varphi(0) = 0$ implies that either $\varphi(r) = 0$ or $\varphi(0) = 0$.
\begin{description}
    \item[\rm Case (i)] In this case, assume that $\varphi(0) = 0$. Then, we must have $\varphi(w) = e^{i \theta} w$, $w \in \mathbb D$. This implies that $\varphi(r) + \varphi(0) = e^{i \theta} r$. On the other hand, we have $\varphi(r) + \varphi(0) = r$. This implies that $e^{i \theta} = 1$. This proves that $\varphi$ is identity.

    \item[\rm Case (ii)] In this case, assume that $\varphi(r) = 0$. Then, we must have $\varphi(w) = e^{i \theta} \frac{w - r}{1 - rw}$, $w \in \mathbb D$. Thus, $\varphi(r) + \varphi(0) = -e^{i\theta}r$. Since $\varphi(r) + \varphi(0) = r$, it follows that $e^{i \theta} = -1$. This proves that $\varphi = \varphi_r$.
\end{description}
Combining the two cases, we observe that $\tilde{\varphi}_r \in$ Aut$(\mathbb G_2)$ is the only element other than the identity element in $\mbox{stab}(r, 0)$.  
\end{proof}

\begin{lemma}\label{curvature condition}
 Let $\mathcal{H}$ be an analytic Hilbert module with the reproducing kernel $K : \mathbb G_2 \times \mathbb G_2 \to \mathbb C$. Suppose $\mathbb K = \left(\!\!\left( \mathbb K_{i\bar j}\right)\!\!\right)$, $\mathbb K_{i\bar j} := \partial_i \bar{\partial}_j \log K$, be the curvature matrix of the associated hermitian holomorphic line bundle $\mathsf L_{\mathcal{H}}$.    
 Then $\mathbb K$ satisfies the following Equation
\begin{equation}\label{eqn:curv_trans_funda}
    \mathbb K(r, 0) = D \tilde{\varphi}(r, 0)^{ \operatorname{ tr } } \mathbb K(r, 0) \overbar{D \tilde{\varphi}(r, 0)}
\end{equation}
for all $0 \leq r < 1$ and $\varphi \in$ M\"ob such that $\tilde{\varphi}(r, 0) = (r, 0)$ if and only if 
\begin{equation}\label{eqn:curv_on_fund}
  \mathbb K_{1\bar 2}(r, 0) = \mathbb K_{2 \bar 1} (r, 0)\,\,\mbox{and}\,\,r(r^2 - 2)\mathbb K_{1\bar1}(r, 0) = 2\mathbb K_{1 \bar2}(r, 0) + r\mathbb K_{2\bar 2}(r, 0), 
\end{equation}
for all $r \in [0, 1)$.
\end{lemma}

\begin{proof}
 Let $r \in (0, 1)$ be a fixed but arbitrary element. From Lemma \ref{fixing elements in aut}, we know that $\tilde{\varphi}_r$ is the only element other than the identity in Aut$(\mathbb G_2)$ such that $\tilde{\varphi}_r(r, 0) = (r, 0)$. A direct computation gives  
 $$D\tilde{\varphi}_r(r, 0) = \frac{1}{r^2 - 1} \begin{pmatrix}
     1 & r(r^2 - 2)\\
     r & -1
 \end{pmatrix}\,\,\mbox{and}\,\,D\tilde{\varphi}_r(r, 0)^{-1} = \frac{1}{r^2 - 1} \begin{pmatrix}
     1 & r(r^2 - 2)\\
     r & -1
 \end{pmatrix}.$$
Since $D\tilde{\varphi}_r(r, 0) = \overbar{D\tilde{\varphi}_r(r, 0)}$, Equation \eqref{eqn:curv_trans_funda} is transformed to the following Equation
$$\mathbb K(r, 0) D \tilde{\varphi}(r, 0)^{-1} = D \tilde{\varphi}(r, 0)^{ \operatorname{ tr } } \mathbb K(r, 0).$$
Now substituting the value of $D \tilde{\varphi}(r, 0)$ and $D \tilde{\varphi}(r, 0)^{-1}$ in the previous Equation and then equating the $(1, 1)$-entry and $(1, 2)$-entry from both sides, we obtain $\mathbb K_{1\bar 2}(r, 0) = \mathbb K_{2 \bar 1} (r, 0)$ and $r(r^2 - 2)\mathbb K_{1\bar1}(r, 0) = 2\mathbb K_{1 \bar2}(r, 0) + r\mathbb K_{2\bar 2}(r, 0)$, respectively. 

Taking $r \to 0$ to the both sides of $r(r^2 - 2)\mathbb K_{1\bar1}(r, 0) = 2\mathbb K_{1 \bar2}(r, 0) + r\mathbb K_{2\bar 2}(r, 0)$, we obtain $\mathbb K_{1 \bar 2}(0,0) = 0$. Moreover, since $\mathbb K_{1\bar 2}(r, 0) = \mathbb K_{2 \bar 1} (r, 0)$ for all $0 < r < 1$,  $\mathbb K_{2 \bar 1}(0,0) = 0$. The conditions $\mathbb K_{1 \bar 2}(0,0) = 0$ and $\mathbb K_{2 \bar 1}(0,0) = 0$ are equivalent to the equation 
\begin{equation*}
    \mathbb K(0, 0) = D \tilde{\varphi}(0, 0)^{ \operatorname{ tr } } \mathbb K(0, 0) \overbar{D \tilde{\varphi}(0, 0)}
\end{equation*}
whenever $\tilde{\varphi} \in \mbox{Aut}(\mathbb G_2)$ is an element in stab$(0,0)$. This proves that Equation \eqref{eqn:curv_trans_funda} is satisfied for all $0 \leq r < 1$ if and only if Equation \eqref{eqn:curv_on_fund} is satisfied for all $0 \leq r < 1$.  
\end{proof}

We now prove the main result in this subsection. For every $\alpha \in \mathbb D$, let $\varphi_\alpha \in$ M\"ob be the unique involution which maps $\alpha$ to $0$ and for every $t \in \mathbb T$, let $\varphi_t \in$ M\"ob be defined by $\varphi_t(z) = tz$, $z \in \mathbb D$. 

\begin{theorem}\label{thm_characterization}
Let $\mathcal{H}$ be an analytic Hilbert module with the reproducing kernel $K : \mathbb G_2 \times \mathbb G_2 \to \mathbb C$. Suppose $\mathbb K = \left(\!\!\left( \mathbb K_{i\bar j}\right)\!\!\right)$, $\mathbb K_{i\bar j} := \partial_i \bar{\partial}_j \log K$, is the curvature matrix of the associated hermitian holomorphic line bundle $\mathsf L_{\mathcal{H}}$.   
For any $\alpha, \beta \in \mathbb D$, define $\theta: \mathbb D \times \mathbb D \to \mathbb T$ by setting 
\[\theta(\alpha, \beta) = \begin{cases} \tfrac{\varphi_\alpha(\beta)}{|\varphi_\alpha(\beta)|} & \text{~if~}  \alpha\ne \beta \\ 1 & \text{~if~} \alpha = \beta.\end{cases}\]

Then $\mathcal{H}$ is Aut$(\mathbb G_2)$-homogeneous if and only if $\mathbb K$ is given by 
\begin{multline}\label{eq:1}
\mathbb{K}(\alpha + \beta, \alpha \beta) = \left(D (\tilde\varphi_{\alpha} \circ \tilde \varphi_{\theta(\alpha, \beta)}) (|\varphi_{\alpha}(\beta)|, 0)^{ \operatorname{ tr } }\right)^{-1} \\ \mathbb{K}\left(|\varphi_{\alpha}(\beta)|, 0\right) \overbar{D (\tilde\varphi_{\alpha} \circ \tilde \varphi_{\theta(\alpha, \beta)}) (|\varphi_{\alpha}(\beta)|, 0)}^{-1},  
\end{multline}
for any $\alpha, \beta \in \mathbb D$ and $r(r^2 - 2)\mathbb K_{1\bar1}(r, 0) = 2\mathbb K_{1 \bar2}(r, 0) + r\mathbb K_{2\bar 2}(r, 0)$, $r \in [0, 1)$.
\end{theorem}

\begin{proof}
For  all $r \in [0, 1)$, we have the equality $r(r^2 - 2)\mathbb K_{1\bar1}(r, 0) = 2\mathbb K_{1 \bar2}(r, 0) + r\mathbb K_{2\bar 2}(r, 0)$. From Lemma \ref{curvature condition}, it follows that the curvature matrix $\mathbb K$ satisfies Equation \eqref{eqn:curv_trans_funda} for all $0 \leq r < 1$ and $\varphi \in$ M\"ob such that $\tilde{\varphi}(r, 0) = (r, 0)$.
Let $\alpha, \beta$ be two arbitrary but fixed elements in $\mathbb D$. Note that 
$$(\tilde\varphi_{\alpha} \circ \tilde \varphi_{\theta(\alpha, \beta)}) (|\varphi_{\alpha}(\beta)|, 0) = (\alpha + \beta, \alpha \beta).$$
Now the proof is completed by appealing to Theorem \ref{thm:2.2}. 
\end{proof}

Let $\mathbb A^{(\lambda)} (\mathbb D^2)$, $\lambda > 1$, be the reproducing kernel Hilbert space consisting of complex valued holomorphic functions defined on $\mathbb D^2$, square integrable  with respect to the measure $dV^{(\lambda)}:=\tfrac{\lambda-1}{\pi} (\rho(\boldsymbol r))^{\lambda - 2}d\boldsymbol r d\boldsymbol \theta$, where 
\[\rho(\boldsymbol r) = (1-r_1^2)(1-r_2^2),\,\, d\boldsymbol r=  r_1 r_2 dr_1 dr_2\text{ and } d \boldsymbol \theta=  d \theta_1 d\theta_2.\]  
The reproducing kernel of $\mathbb A^{(\lambda)} (\mathbb D^2)$ is then known to be the function $$ \frac{1}{(1-z_1\overbar{w}_1)^\lambda(1-z_2\overbar{w}_2)^\lambda} . $$

\subsection{Weighted Bergman modules \texorpdfstring{$\mathbb A^{(\lambda)} (\mathbb G^2)$}{LG}} Let $d \hat{V}_{\boldsymbol s}^{(\lambda)}$ be the measure on the symmetrized bidisc $\mathbb{G}_2$ obtained by the change of variable formula:
$$
\int_{\mathbb{G}_2} f d \hat{V}_{\boldsymbol s}^{(\lambda)}=\int_{\mathbb{D}^2}(f \circ \boldsymbol s)\left|J_{\boldsymbol s}\right|^2 d V^{(\lambda)}, \lambda>1,
$$
where $J_{\boldsymbol s}(\boldsymbol{z})$ is the complex Jacobian of the symmetrization map $\boldsymbol s$. For $\lambda > 1$, we let $dV_{\boldsymbol s}^{(\lambda)}$ be the normalized measure $\left\|J_{\boldsymbol s}\right\|_\lambda^{-2} d V_{\boldsymbol s}^{(\lambda)}$, where $ \norm{ J_{ \boldsymbol s } }^{ - 2 }_{ \lambda } $ is the norm of $ J_{ \boldsymbol s } $ in the Hilbert space $ L^2 ( \mathbb D^2, dV^{ ( \lambda ) } ) $.
 The weighted Bergman space $\mathbb{A}^{(\lambda)}\left(\mathbb{G}_2\right), \lambda>1$, on the symmetrized bidisc $\mathbb{G}_2$ is the subspace of the Hilbert space $L^2\left(\mathbb{G}_2, d V_{\boldsymbol s}^{(\lambda)}\right)$ consisting of holomorphic functions. The map $\Gamma: \mathbb{A}^{(\lambda)}\left(\mathbb{G}_2\right) \longrightarrow \mathbb{A}^{(\lambda)}\left(\mathbb{D}^2\right)$ defined by the rule:
$$
(\Gamma f)(\boldsymbol{z})=\left\|J_{\boldsymbol s}\right\|_\lambda^{-1} J_{\boldsymbol s}(\boldsymbol{z})(f \circ \boldsymbol s)(\boldsymbol{z}), \quad f \in \mathbb{A}^{(\lambda)}\left(\mathbb{G}_2\right), \boldsymbol{z} \in \mathbb{D}^2
$$
is an isometry. The range of $\Gamma$ coincides with the subspace $\mathbb A^{(\lambda)}_{\text{anti}}(\mathbb D^2)\subset \mathbb A^{(\lambda)}(\mathbb D^2)$ consisting anti-symmetric functions in $\mathbb A^{(\lambda)}(\mathbb D^2)$. Thus $\Gamma$ is a unitary isomorphism between $\mathbb{A}^{(\lambda)}\left(\mathbb{G}_2\right)$ and $\mathbb A^{(\lambda)}_{\text{anti}}(\mathbb D^2)$. 

The reproducing kernel of $\mathbb A^{(\lambda)}_{\text{anti}}(\mathbb D^2)$ is computed explicitly in \cite{MRZ}. Therefore, the reproducing kernel $B^{(\lambda)}$ -- known as the weighted Bergman kernel with weight $\lambda$ -- of the Hilbert space $\mathbb{A}^{(\lambda)}\left(\mathbb{G}_2\right)$ can be computed explicitly using the reproducing kernel of  $\mathbb A^{(\lambda)}_{\text{anti}}(\mathbb D^2)$ and the unitary isomorphism $\Gamma$ as shown in \cite[Theorem 2.3]{MRZ}. For $\lambda > 0$, the weighted Bergman kernel $B^{(\lambda)}$ on the symmetrized bidisc is given  below (cf. \cite[Corollary 2.4]{MRZ}). 
For $\boldsymbol{z} = (z_1,z_2)$ and $\boldsymbol{w}=(w_1, w_2)$ in $\mathbb{D}^2$, let $B_{\boldsymbol{s}}^{(\lambda)}$ be the kernel 
\begin{flalign}\label{defofwbergker}
 &B^{(\lambda)}_{\boldsymbol{s}}( \boldsymbol{ z } ,  \boldsymbol{ w } )\\
\nonumber &\phantom{gada}= \frac{1}{2(z_1 - z_2)(\bar{w}_1 - \bar{w}_2)} \times
\left[ \frac{1}{(1 - z_1\bar{w}_1)^\lambda (1 - z_2\bar{w}_2)^\lambda} - \frac{1}{(1 - z_1 \bar{w}_2)^{\lambda} (1 - z_2 \bar{w}_1)^{\lambda}}\right],  
\end{flalign}
defined on the bidisc $\mathbb{D}^2$. We note that the right hand side of Equation \eqref{defofwbergker} is symmetric in $z_1,z_2$ and $w_1, w_2$. Consequently, it is a function of $\boldsymbol s( \boldsymbol{ z } ), \boldsymbol s( \boldsymbol{ w })$. Therefore, the kernel $B_{\boldsymbol{s}}^{(\lambda)}$ defines a kernel $B^{(\lambda)}$ on $\mathbb{G}_2$ by setting 
\begin{equation}
B^{(\lambda)}(\boldsymbol s( \boldsymbol{ z } ), \boldsymbol s( \boldsymbol{ w }) )= B^{(\lambda)}_{\boldsymbol{s}}( \boldsymbol{ z }, \boldsymbol{ w }),\,\, \boldsymbol{z}, \boldsymbol{w}\in \mathbb{D}^2.
\end{equation} 

 If we set $\Phi:\mathbb{D}^2 \to \mathbb{D}^2$ to be the map $(z_1,z_2) \mapsto (\varphi(z_1), \varphi(z_2)) $, $\varphi\in \mbox{M\"{o}b}$, then the holomorphic function determined by imposing the equality $\tilde{\varphi}\circ \boldsymbol{s}
 =\boldsymbol{s} \circ \Phi$ is an automorphism of $\mathbb{G}_2$. 
 The automorphism group of $\mathbb{G}_2$ consists of exactly the functions $\tilde{\varphi}$, $\varphi$ in M\"{o}b. Applying the chain rule, we have that 
\[
    \det(D\tilde{\varphi})(\boldsymbol{s}(\boldsymbol{z})) = \frac{\varphi(z_2) - \varphi(z_1)}{z_2 - z_1} \varphi^\prime(z_1) \varphi^\prime(z_2),\,\, z_1, z_2 \in \mathbb{D}, 
\]
where $D\tilde{\varphi}$ is the derivative of $\tilde{\varphi}$.
It is easily checked that 
\begin{equation}
\varphi(z_1) - \varphi(z_2) =(z_1 - z_2) \big (\varphi^\prime (z_1)\varphi^\prime(z_2)\big )^{1/2},\,\, z_1, z_2 \in \mathbb{D}. 
\end{equation}
Evidently, $\varphi^\prime (z_1)\varphi^\prime(z_2)$ is a symmetric function of $z_1, z_2 \in \mathbb{D}$, and hence, it defines a function on $\mathbb{G}_2$, which we denote by $\hat{\phi}$. 
Therefore, we have 
\begin{equation}\label{eqn:4.8}
\det(D\tilde{\varphi})(\boldsymbol{s}(\boldsymbol{z})) = (\varphi^\prime (z_1)\varphi^\prime(z_2)\big )^{3/2} =  \big (\hat{\phi}(\boldsymbol{s}(\boldsymbol{z}) )\big )^{3/2},\,\, z_1, z_2 \in \mathbb{D}.
\end{equation}
Rewriting Equation \eqref{eqn:4.8} in the form 
\[J(\tilde{\varphi}, \boldsymbol{s}(\boldsymbol{z})): = \hat{\phi} (\boldsymbol{s}(\boldsymbol{z}) ) =  \det(D\tilde{\varphi})(\boldsymbol{s}(\boldsymbol{z}))^{2/3}
\] 
and then applying the chain rule, we see that $J:\operatorname{Aut}(\mathbb{G}_2) \times \mathbb{G}_2 \to \mathbb{C}\setminus \{0\}$ defines a cocycle. 

A fundamental set $\Lambda$ for the symmetrized bidisc $\mathbb{G}_2$ is $\{(r,0)\mid 0\leq r < 1\}$ (see Equation \eqref{cocycle}). Explicitly, for 
$ ( z_1 + z_2, z_1 z_2 ) \in \mathbb{ G }_2 $, consider the involution $\varphi_{z_1}(w) = - \frac{w - z_1}{1 - \bar{z}_1 w}$, $w \in \mathbb D$, in M\"ob. Write $\varphi_{z_1}(z_2) = r e^{i\theta}$ with $r = \left| \varphi_{z_1}(z_2)\right|$ and $\theta \in [0, 2 \pi)$. Now denote the automorphism $ w \mapsto e^{ i \theta } w $ of $ \mathbb{ D } $ by $ \varphi_{ \theta } $. Since $\varphi_{z_1}$ is an involution,  it follows that 
$$\left(\varphi_{z_1} \circ \varphi_\theta\right)(r) = \varphi_{z_1}(r e^{i\theta}) = \varphi_{z_1} \left(\varphi_{z_1}(z_2)\right) = z_2.$$
Finally, letting $\varphi = \varphi_{z_1} \circ \varphi_{\theta}$ and recalling the definition of $ \tilde{ \varphi } $, we have that $\tilde{\varphi}\left(s(| \varphi_{z_1}(z_2) | , 0) \right) = (z_1 + z_2, z_1z_2)$.

If $B$ is the Bergman kernel of a domain $\Omega \subset \mathbb{C}^m$ and $g:\Omega\to \Omega$ is a bi-holomorphic map of $\Omega$, then the transformation rule for $B$ is the equality: 
\begin{equation*}
     B(\boldsymbol{z},\boldsymbol{w}) = \det (Dg(\boldsymbol{z})) B(g(\boldsymbol{z}), g (\boldsymbol{w}) )\, \overline{\det (Dg(\boldsymbol{w}))},\,\, \boldsymbol{z}, \boldsymbol{w}\in \Omega.  
\end{equation*}
Specializing to the case of $\Omega=\mathbb{G}_2$, $\boldsymbol{z} = \boldsymbol{w}$ and taking $g=\tilde{\varphi}$, we have the formula 
\begin{equation}\label{eqn:4.8a}
B^{(2)} (\boldsymbol{s}(\boldsymbol{z}), \boldsymbol{s}(\boldsymbol{z}) ) = |\det (D\tilde{\varphi}(\boldsymbol{s}(\boldsymbol{z})))|^2 B^{(2)}(\tilde{\varphi}(\boldsymbol{s}(\boldsymbol{z})),\tilde{\varphi}(\boldsymbol{s}(\boldsymbol{z}))) 
\end{equation}
where $B^{(2)}$ is the Bergman kernel of the symmetrized bidisc $\mathbb{G}_2$. Setting $\tilde{\varphi}:= \widetilde{\varphi_{z_1} \circ \varphi_\theta}$, and evaluating Equation \eqref{eqn:4.8a} at $(| \varphi_{z_1}(z_2) | , 0)$, we finally obtain the factorization: 
\begin{align*}
    B^{(2)}(\boldsymbol{s}(\boldsymbol{z}), \boldsymbol{s}(\boldsymbol{z})) &= |\det (D\tilde{\varphi}(\boldsymbol{s}(| \varphi_{z_1}(z_2) | , 0)))^{-2} B^{(2)}((| \varphi_{z_1}(z_2) | , 0), (| \varphi_{z_1}(z_2) | , 0))\\
    &= \frac{1}{|1-z_1\overbar{z_2}|^6} \frac{1}{2|\varphi_{z_1}(z_2)|^2} \left ( \frac{1}{( 1- |\varphi_{z_1}(z_2)|^2)^2} -1 \right )\\
    &= \frac{1}{|1-z_1\overbar{z_2}|^6} \frac{|1-z_1\overbar{z_2}|^2}{2|z_1-z_2|^2} \frac{|z_1 - z_2|^2(|1-z_1\overline{z_2}|^2 + (1-|z_1|^2)(1-|z_2|^2))}{(1-|z_1|^2)^2(1-|z_2|^2)^2}\\
    &= \frac{1}{2|1-z_1\overbar{z_2}|^4} \Big ( \frac{|1-z_1\overline{z_2}|^2}{(1-|z_1|^2)^2(1-|z_2|^2)^2} + \frac{1}{(1-|z_1|^2)(1-|z_2|^2)}\Big )
\end{align*}
where we have used the easily verified formula: 
\[\det (D\tilde{\varphi}(\boldsymbol{s}(| \varphi_{z_1}(z_2) | , 0)) = \varphi_{z_1}^\prime(|\varphi_{z_1}(z_2)|) |\varphi_{z_1}^\prime(0)| = |1-z_1 \overline{z_2}|^3.\]
In this case, the factorization of $B^{(2)}$ obtained from general principles shows that it is quasi-invariant (relative to $\operatorname{Aut}(\mathbb{G}_2)$) with the multiplier 
    $\det (D\tilde{\varphi}(\boldsymbol{s}(| \varphi_{z_1}(z_2) | , 0))^{-2}$. 
Unlike the case where the automorphism group acts transitively and therefore, any power of the Bergman kernel $B$ is automatically quasi-invariant, here $B^{(\lambda)}$ is not a power of $B^{(2)}$. The computation given above for $B^{(2)}$ shows that if $(B^{(2)})^\nu$ is non-negative definite for some $\nu >0$, then it must be quasi-invariant. 
Furthermore, if the Hilbert module determined by the kernel $(B^{(2)})^\nu$ is analytic, it must be homogeneous under $\operatorname{Aut}(\mathbb{G}_2)$. However, we have to proceed differently to establish that $B^{(\lambda)}$ is quasi-invariant. Then we investigate if and when  $(B^{(\lambda)})^\nu$ is quasi-invariant for a pair $\lambda, \nu$ of positive real numbers in the following subsection.

\begin{prop}\label{homogofweightedbergker}
For every $\lambda > 0$, the reproducing kernel $B^{(\lambda)}$ of the  Hilbert space $\mathbb A^{(\lambda)}(\mathbb G_2)$ is quasi-invariant with respect to $\operatorname{Aut}(\mathbb{G}_2)$. Moreover, $\mathbb A^{(\lambda)}(\mathbb G_2)$ is an analytic $\operatorname{Aut}(\mathbb G_2)$-homogeneous  Hilbert module. 
\end{prop}
\begin{proof} Since the space of anti-symmetric polynomials is dense in $\mathbb A^{(\lambda)}_{\text{anti}}(\mathbb D^2)$, it follows that the $\mathbb{C}[\boldsymbol{z}]$ is dense in $\mathbb A^{(\lambda)}(\mathbb G_2)$.
Note that the pair $(M_1, M_2)$ acting on $\mathbb A^{(\lambda)}(\mathbb G_2)$ is unitarily equivalent to $(M_1+M_2, M_1 M_2)$ acting on $\mathbb A^{(\lambda)}_{\text{anti}}(\mathbb D^2)$. 
Consequently, the pair $(M_1, M_2)$ acting on $\mathbb A^{(\lambda)}(\mathbb G_2)$ is also bounded. Therefore, $\mathbb A^{(\lambda)}(\mathbb G_2)$ is an analytic Hilbert module.  Consequently, it is enough to prove only that $\mathbb A^{(\lambda)}(\mathbb G_2)$ is Aut$(\mathbb G_2)$-homogeneous. Since $\mathbb A^{(\lambda)}(\mathbb G_2)$ is an analytic Hilbert module, Aut$(\mathbb G_2)$-homogeneity of $\mathbb A^{(\lambda)}(\mathbb G_2)$ is equivalent to quasi-invariance of the reproducing  kernel $B^{(\lambda)}$.    

Let $\varphi$ be an arbitrary element of M\"{o}b and $\boldsymbol{z} = (z_1, z_2)$, $\boldsymbol{w} = (w_1, w_2)$ be two arbitrary point of $\mathbb D^2$. From Equation \eqref{defofwbergker}, we have
\begin{flalign*}
 &B^{(\lambda)}(\tilde{\varphi}(\boldsymbol s( \boldsymbol{ z }) ), \tilde{\varphi}(\boldsymbol s( \boldsymbol{ w }) ))
 = B^{(\lambda)}(\boldsymbol s( \varphi(\boldsymbol{ z }) ), \boldsymbol s( \varphi(\boldsymbol{ w })) )\\
 &= \frac{1}{2\left(\varphi(z_1) - \varphi(z_2)\right)\left(\overbar{\varphi({w}_1)} - \overbar{\varphi({w}_2)}\right)}\\  &\phantom{Din} \times
\left[ \frac{1}{\left(1 - \varphi(z_1)\overbar{\varphi({w}_1)}\right)^\lambda \left(1 - \varphi(z_2)\overbar{\varphi({w}_2)}\right)^\lambda} - \frac{1}{\left(1 - \varphi(z_1)\overbar{\varphi({w}_2)}\right)^\lambda \left(1 - \varphi(z_2)\overbar{\varphi({w}_1)}\right)^\lambda}\right].  
\end{flalign*}

A direct computation verifies the following identities
\begin{align*}
& \left(\varphi(z_1) - \varphi(z_2)\right)(\overbar{\varphi({w}_1)} - \overbar{\varphi({w}_2)}) = \\ &\phantom{gadadharmisraprahllad}\left(\varphi'(z_1)\right)^{\frac{1}{2}}\left(\varphi'(z_2)\right)^{\frac{1}{2}} \overbar{\left(\varphi'(w_1)\right)}^\frac{1}{2} \overbar{\left(\varphi'(w_2)\right)}^\frac{1}{2}(z_1 - z_2)(\overbar{w_1} - \overbar{w_2}),\\
& \left(1 - \varphi(z_1)\overbar{\varphi(w_1)}\right)^\lambda = \left(\varphi'(z_1)\right)^{\frac{\lambda}{2}}\overbar{\left(\varphi'(w_1)\right)}^\frac{\lambda}{2}(1 - z_1 \overbar{w_1})^\lambda,\\
& \left(1 - \varphi(z_2)\overbar{\varphi(w_2)}\right)^\lambda = \left(\varphi'(z_2)\right)^{\frac{\lambda}{2}}\overbar{\left(\varphi'(w_2)\right)}^\frac{\lambda}{2}(1 - z_2 \overbar{w_2})^\lambda,\text{ and }\\
& \left(1 - \varphi(z_1)\overbar{\varphi({w}_2)}\right)^\lambda \left(1 - \varphi(z_2)\overbar{\varphi({w}_1)}\right)^\lambda =\\ &\phantom{gadadharmisraprahlladdeb}\left(\varphi'(z_1)\right)^{\frac{\lambda}{2}}\overbar{\left(\varphi'(w_1)\right)}^\frac{\lambda}{2} \left(\varphi'(z_2)\right)^{\frac{\lambda}{2}}\overbar{\left(\varphi'(w_2)\right)}^\frac{\lambda}{2} (1 - z_1 \overbar{w_2})^{\lambda}(1 - z_2\overbar{w_1})^\lambda.
\end{align*}
Substituting these four identities in the expression of $B^{(\lambda)}(\tilde{\varphi}(\boldsymbol s( \boldsymbol{ z }) ), \tilde{\varphi}(\boldsymbol s( \boldsymbol{ w }) ))$, we obtain
\begin{equation}\label{quasiinvproofbergker}
 B^{(\lambda)}(\boldsymbol s( \boldsymbol{ z }) , \boldsymbol s( \boldsymbol{ w }) ) = \left(\varphi^\prime(z_1) \varphi^\prime(z_2) \right)^{\frac{\lambda+1}{2}}
 B^{(\lambda)}(\tilde{\varphi}(\boldsymbol s( \boldsymbol{ z }) ), \tilde{\varphi}(\boldsymbol s( \boldsymbol{ w }) ))
 \overbar{\left(\varphi^\prime(w_1) \varphi^\prime(w_2) \right)}^{\frac{\lambda+1}{2}} 
\end{equation}
Since $\varphi^\prime(z_1) \varphi^\prime(z_2)$ is a symmetric holomorphic function on $\mathbb D^2$, there exists a holomorphic function $\hat{\phi}$ on $\mathbb G_2$ such that $\hat{\phi}(\boldsymbol{s}( \boldsymbol{ z })) = \varphi^\prime(z_1) \varphi^\prime(z_2)$. Therefore, we can rewrite Equation \eqref{quasiinvproofbergker} in the form 
\begin{equation}\label{quasiinvproofbergker1}
    B^{(\lambda)}(\boldsymbol s( \boldsymbol{ z }) , \boldsymbol s( \boldsymbol{ w }) ) =  \left(\hat{\phi}(\boldsymbol s( \boldsymbol{ z }))\right)^{\frac{\lambda + 1}{2}}
 B^{(\lambda)}(\tilde{\varphi}(\boldsymbol s( \boldsymbol{ z }) ), \tilde{\varphi}(\boldsymbol s( \boldsymbol{ w }) ))
 \overbar{\left(\hat{\phi}(\boldsymbol s( \boldsymbol{ w }))\right)}^{\frac{\lambda + 1}{2}}.  
\end{equation}
proving that $B^{(\lambda)}$ is quasi-invariant. Thus, from Theorem \ref{thm:2.1-Kquasi}, it follows that $\mathbb A^{(\lambda)}(\mathbb G_2)$ is an Aut$(\mathbb G_2)$-homogeneous analytic Hilbert module. 
\end{proof}
\begin{cor} \label{cor:4.5} The weighted Bergman kernel $B^{(\lambda)}$ of the symmetrized bidisc $\mathbb{G}_2$ is quasi-invariant. Explicitly, we have 
\[B^{(\lambda)}(\boldsymbol{u}, \boldsymbol{v}) = (\hat{\phi}(\boldsymbol{u}))^{\tfrac{\lambda+1}{2}} B^{(\lambda)}(\tilde{\varphi}(\boldsymbol{u}), \tilde{\varphi}(\boldsymbol{v})) \overline{(\hat{\phi}(\boldsymbol{u}))}^{\tfrac{\lambda+1}{2}},\,\,\boldsymbol{u}, \boldsymbol{v} \in \mathbb{G}_2.\]
\end{cor}

\begin{remark} 
It is possible to give a different proof of  Proposition \ref{homogofweightedbergker}. Consider the Hilbert module $\mathbb{A}^{(\lambda)}(\mathbb{D}^2)$. The pair of multiplication operators $(M_1,M_2)$ on $\mathbb{A}^{(\lambda)}(\mathbb{D}^2)$ is homogeneous under the automorphism group of $\mathbb{D}^2$. 
The intertwining unitary in this case is of the form: 
\[(U^{(\lambda)}_{\phi^{-1}} f)(\boldsymbol{z}) = (\det D\phi)^{\lambda/2}(\boldsymbol{z}) f(\phi(\boldsymbol{z})),\,\,  f\in \mathbb{A}^{(\lambda)}(\mathbb{D}^2),\] 
where $\phi=(\varphi_1,\varphi_2)$, $\varphi_1, \varphi_2$ are in M\"{o}b.
The subspace $\mathbb{A}_{\operatorname{anti}}^{(\lambda)}(\mathbb{D}^2)$ is invariant under the unitaries $U_{\phi^{-1}}^{(\lambda)}$ and the pair of operators $(M_1+ M_2, M_1 M_2)$, when $\phi=(\varphi, \varphi)$, $\varphi \in \mbox{M\"{o}b}$. Therefore, 
\[U_{\phi^{-1}}^{(\lambda)} (M_1 + M_2, M_1 M_2) {U_{\phi^{-1}}^{(\lambda)*}}= (\varphi(M_1) + \varphi(M_2), \varphi(M_1) \varphi(M_2) ).\]
However, transporting the unitaries $U_{\phi^{-1}}^{(\lambda)}$ to the weighted Bergman space $\mathbb{A}^{(\lambda)}(\mathbb{G}_2)$ using the unitary $\Gamma$ defined in \cite[p. 2361]{MRZ}, we conclude that 
it must intertwine $(M_1, M_2)$ and $\tilde{\varphi}(M_1, M_2)$ on $\mathbb{A}^{(\lambda)}(\mathbb{G}_2)$. This proves that for $\lambda > 0$, the pair $(M_1, M_2)$ defined on $\mathbb{A}^{(\lambda)}(\mathbb{G}_2)$ is homogeneous under the automorphism group of $\mathbb{G}_2$. An explicit form of this unitary is given in Remark \ref{rem:4.11}. 
\end{remark}
\begin{remark}
 From Equation \eqref{quasiinvproofbergker1}, we can obtain a factorization  recovering $B^{(\lambda)}(\boldsymbol{z}, \boldsymbol{z})$ from its values on the fundamental set, namely, $\Lambda=\{(r,0) : r=  |\varphi_{z_1} (z_2)|, \varphi \in \mbox{\rm M\"{o}b}\}$. Indeed, we have  
 \begin{equation}\label{factorquasiinvbergker}
  B^{(\lambda)}(\boldsymbol s( \boldsymbol{ z }) , \boldsymbol s( \boldsymbol{ z }) ) = (1-z_1\overbar{z_2})^{\frac{\lambda + 1}{2}}
 \frac{1}{2|\varphi_{z_1}(z_2)|^2} \big ( \frac{1}{( 1- |\varphi_{z_1}(z_2)|^2)^\lambda} -1 \big )
 \overbar{(1-z_1\overbar{z_2})}^{\frac{\lambda + 1}{2}}.  
\end{equation}
\end{remark}

\subsection{Hilbert modules \texorpdfstring{$\mathbb A^{(\lambda, \nu)}$}{TT} over \texorpdfstring{$ \mathbb{ G }_2 $}{AA}} 
 In this subsection, we investigate positive real powers of the weighted Bergman kernels $ B^{ ( \lambda ) } $ on the symmetrized bidisc $ \mathbb{ G }_2 $. 
We note that if $ ( B^{ ( \lambda ) } )^{ \nu } $ is non-negative definite for some $\nu >0$, then there is a (uniquely determined) Hilbert module $\mathbb A^{(\lambda, \nu)}$ possessing $ ( B^{ ( \lambda ) } )^{ \nu } $ as its reproducing kernel. For this, we first check if $ B^{ ( \lambda ) } $ is a non-vanishing function on $\mathbb G_2$ using the lemma below.

\begin{lemma} \label{weightedbergkerexpansion}
For every $\boldsymbol{ z } = (z_1, z_2) \in \mathbb D^2$, there exists $\epsilon_{\boldsymbol{ z }} > 0$ such that the function $ B^{ ( \lambda ) } $ on $ \mathbb{ G }_2 $ is expressed as
    \begin{equation} \label{eqncurv1bergker}
        B^{ ( \lambda ) } (\boldsymbol s( \boldsymbol{ z } ), \boldsymbol s (\boldsymbol{ w }) ) = \frac{ 1 }{ 2 (1 - z_1\bar{w}_2)^\lambda(1 -  z_2 \bar{ w }_1)^{  \lambda } }      \sum_{ n \geq 0 }  { \binom{\lambda}{n + 1 }} \frac{ ( z_1 - z_2 )^{ n } (\bar w_1 - \bar w_2)^n}{ ( 1 - z_1 \bar w_1 )^{ n + 1 } ( 1 - z_2 \bar w_2 )^{ n + 1 } },
    \end{equation}
for all $\boldsymbol{ w } = (w_1, w_2) \in \mathbb D^2$ with $|w_1 - w_2| < \epsilon_{\boldsymbol{ z }}$.    
In particular, for $ z \in \mathbb{ D } $,
     \begin{equation} \label{resBerg}
      B^{ ( \lambda ) } (\boldsymbol{s} (z,z), \boldsymbol{s} (z,z) ) = \frac{ \lambda }{ 2 } \times \frac{ 1 }{ ( 1 - | z |^2 )^{ 2 \lambda + 2 } }.
    \end{equation}  
\end{lemma}
\begin{proof}
Let $\boldsymbol{z} = (z_1, z_2)$ and $\boldsymbol{w} = (w_1, w_2)$ be arbitrary but fixed points in $\mathbb D^2$. We have 
\begin{flalign*}
&B^{(\lambda)}(\boldsymbol s(\boldsymbol{z}), \boldsymbol s(\boldsymbol{w}))\\ & \phantom{xxxxx} = \frac{1}{2(z_1 - z_2)(\bar{w}_1 - \bar{w}_2)} \times
\left[ \frac{1}{(1 - z_1\bar{w}_1)^\lambda (1 - z_2\bar{w}_2)^\lambda} - \frac{1}{(1 - z_1\bar{w}_2)^\lambda (1 - z_2\bar{w}_1)^\lambda}\right]\\
&\phantom{xxxxx} = \frac{1}{2(z_1 - z_2)(\bar w_1 - \bar w_2) (1 - z_1 \bar w_2)^{\lambda} (1 - z_2 \bar w_1)^\lambda} \left[\frac{(1 - z_1 \bar w_2)^{\lambda} (1 - z_2 \bar w_1)^\lambda}{(1 - z_1 \bar w_1)^\lambda (1 - z_2 \bar w_2)^\lambda} - 1\right]\\
&\phantom{xxxxx} = \frac{1}{2(z_1 - z_2)(\bar w_1 - \bar w_2) (1 - z_1 \bar w_2)^{\lambda} (1 - z_2 \bar w_1)^\lambda} \\
&\phantom{xxxxxxxxxxxxxxx.}\phantom{gadadhar} \times \left[\displaystyle \sum_{n \geq 0} (-1)^n {\binom{\lambda}{n}} \left( 1 - \frac{(1 - z_1 \bar w_2)(1 - z_2 \bar w_1)}{(1 - z_1 \bar w_1)(1 - z_2 \bar w_2)}\right)^n - 1 \right]\\
&\phantom{xxxxx} = \frac{ 1 }{ 2 (1 - z_1 \bar w_2)^{\lambda} (1 - z_2 \bar w_1)^\lambda } \sum_{ n \geq 0 } { \binom{\lambda}{n + 1} } \frac{ (z_1 - z_2)^{ n } (\bar w_1 - \bar w_2)^n }{ ( 1 - z_1 \bar w_1 )^{ n + 1 } ( 1 - z_2 \bar w_2 )^{ n + 1 } }
\end{flalign*}
completing the proof of \eqref{eqncurv1bergker}. 
Finally, Equation \eqref{resBerg} follows by substituting $z_1 = z_2 = w_1 = w_2 = z$ in Equation \eqref{eqncurv1bergker}.
\end{proof}

A proof of the following theorem for $\lambda=2$, along a slightly different route, appears in \cite[Proposition 11]{EZ}.
\begin{theorem} \label{nonvanishing}
$ B^{ ( \lambda ) } $ does not vanish on $ \mathbb{ G }_2 \times \mathbb G_2$ if and only if $0 < \lambda \leq 2$.
\end{theorem}

\begin{proof}
Let $B^{(\lambda)}_s : \mathbb D^2 \times \mathbb D^2 \to \mathbb C$ be the function defined by $B^{(\lambda)}_s (\boldsymbol{z}, \boldsymbol{w}) = B^{(\lambda)} (\boldsymbol{s} (\boldsymbol{z}), \boldsymbol{s}(\boldsymbol{w}))$, $\boldsymbol{z}, \boldsymbol{w} \in \mathbb D^2$. The symmetrization map
$\boldsymbol s:\mathbb D^2\longrightarrow\mathbb G_2$ is surjective. Therefore,
$B^{(\lambda)}$ is zero-free on $\mathbb G_2\times\mathbb G_2$
if and only if $B_s^{(\lambda)}$ is zero-free on $\mathbb D^2\times\mathbb D^2$.
It is therefore enough to prove that $B^{(\lambda)}_s$ does not vanish on $\mathbb D^2 \times \mathbb D^2$ if and only if $0 < \lambda \leq 2$.

First, assume that $\lambda > 2$. Take $t = \tan \frac{\pi}{2\lambda}$. Since $\lambda > 2$, we have that $0 < \frac{\pi}{2\lambda} < \frac{\pi}{4}$ and therefore, $0 < t < 1$. We have that 
\begin{align*}
B_s^{(\lambda)}((i \sqrt{t},-i \sqrt{t}),(\sqrt{t},-\sqrt{t})) 
&= \frac{(1-i t)^{-2 \lambda}-(1+i t)^{-2 \lambda}}{8 i t} \\
&=\frac{\sin (2 \lambda \arctan t)}{4 t\left(1+t^2\right)^\lambda} .
\end{align*}
Since $t=\tan \left(\frac{\pi}{2 \lambda}\right)$ and $0<\pi /(2 \lambda)<\pi / 4$, it follows that $\arctan t=\frac{\pi}{2 \lambda}$.
Consequently, 
$\sin (2 \lambda \arctan t)=\sin \pi=0$,
and hence 
\[B^{(\lambda)}_s \left((\sqrt{t}i, -\sqrt{t}i), (\sqrt{t}, -\sqrt{t})\right) = 0.\]

Let $\lambda$ be a positive real number with $\lambda \leq 2$. We claim that $B^{(\lambda)}_s\left((z_1,z_2), (w, 0)\right)$ does not vanish for any $z_1, z_2, w \in \mathbb D$.

From Equation \eqref{eqncurv1bergker}, it follows that 
\[ B^{ ( \lambda ) }_s ( ( z, z ), (w, 0) ) =  \frac{ \lambda }{ 2 ( 1 - z \overline{ w } )^{\lambda + 1} },\]
for any $z, w \in \mathbb D$. 

If $w=0$, then the second bidisc variable is $(0,0)\in\Delta$. Taking limit as $w\to 0$, we get 
\[B_s^{(\lambda)}\bigl((z_1,z_2),(0,0)\bigr) = \frac{\lambda}{2}\neq 0.\]
Therefore, it remains only to consider $z_1\neq z_2$ and $w\neq0$.

Thus to prove our claim, it is enough to assume $z_1 \neq z_2$ and $w \neq 0$. Suppose, to the contrary, that there exist $z_1, z_2, w \in \mathbb D$ with $z_1 \neq z_2$, $w \neq 0$ and $B^{(\lambda)}_s\left((z_1,z_2), (w, 0)\right) = 0$. For any $\lambda > 0$, here as before powers are defined by using the principal branch of the logarithm. From Equation \eqref{defofwbergker}, it follows that
\begin{equation}\label{eqn:non-vanishing}
    (1 - z_1 \bar w)^\lambda = (1 - z_2 \bar w)^\lambda.
\end{equation}
Now put $u_i=1-z_i\overline w$, $i=1,2$. Then 
\[|u_i-1|=|z_i\overline w|<1,\]
so $\operatorname{Re}u_i > 0$ and hence $-\frac{\pi}{2}<\arg u_i<\frac{\pi}{2}$.
For $i=1,2$, writing 
\[
u_i=r_i e^{i\theta_i}, \quad r_i>0,\quad -\frac{\pi}{2}<\theta_i<\frac{\pi}{2},
\]
we have  
\[
r_1^\lambda e^{i\lambda\theta_1} = r_2^\lambda e^{i\lambda\theta_2},
\]
from Equation \eqref{eqn:non-vanishing}. Hence 
$r_1^\lambda=r_2^\lambda$, and $\lambda(\theta_1-\theta_2)\in2\pi\mathbb Z$.
On the other hand, because $0<\lambda\le 2$ and
$|\theta_1-\theta_2|<\pi$,
\[
|\lambda(\theta_1-\theta_2)|<2\pi.
\]
Therefore $\theta_1=\theta_2$, and hence $u_1=u_2$. Since $w\neq0$, this forces $z_1=z_2$, a contradiction.
This proves that
\[
B_s^{(\lambda)}
\bigl((z_1,z_2),(w,0)\bigr)\neq 0.
\]

Finally, assume that $z_1, z_2, w_1$ and $w_2$ are arbitrary elements of $\mathbb D$. Let $\varphi$ be an element of M\"ob such that $\varphi(w_2) = 0$. By the preceding argument, $B^{(\lambda)}_s\left((z_1,z_2), (w, 0)\right)$ does not vanish for any $z_1, z_2, w \in \mathbb D$, it follows that $B_s^{(\lambda)}
\bigl(\varphi(\boldsymbol z),\varphi(\boldsymbol w)\bigr) \neq 0.$  
By the quasi-invariance formula \eqref{quasiinvproofbergker}, the kernel values
\[
B_s^{(\lambda)}(\boldsymbol z,\boldsymbol w)
\quad\text{and}\quad
B_s^{(\lambda)}
\bigl(\varphi(\boldsymbol z),\varphi(\boldsymbol w)\bigr)
\]
differ by a product of powers of
\[
\varphi^\prime(z_1),\ \varphi^\prime(z_2),\
\overline{\varphi^\prime(w_1)},\
\overline{\varphi^\prime(w_2)}.
\]
Since $\varphi^\prime$ has no zeros in $\mathbb D$, this multiplier is nonzero. Consequently,
$B_s^{(\lambda)}(\boldsymbol z,\boldsymbol w)\neq 0$.
\end{proof}

Since $\mathbb G_2$ is simply connected \cite[Theorem 2.3]{AY}, the nonvanishing of $B^{(\lambda)}$, $0<\lambda\le 2$, guarantees the existence of a global logarithmic $B^{(\lambda)}$ on $\mathbb G_2\times\mathbb G_2$. After requiring its value at $(0,0)$ to be real, we select a unique global branch as follows 
\[
\bigl(B^{(\lambda)}\bigr)^{\nu} : = \exp\bigl(\nu\log B^{(\lambda)}\bigr),
\quad \nu > 0.
\]
Moreover, we note that the function $ ( B^{ ( \lambda ) } )^{ \nu } $ is a non-negative kernel for every $ \nu > 0 $ if $\lambda=1$. We don't know if this is true for any other $\lambda\in (0,2]$ 

\begin{prop} \label{inf div of Hardy ker}
    (i) For $ \nu > 0 $, $ ( B^{ ( 1 ) } )^{ \nu } $ is a non-negative definite kernel on $ \mathbb{ G }_2 $. 

    (ii) For $\nu \geq 1$, the Hilbert space $\mathbb A^{(1, \nu)}$ corresponding to the non-negative definite kernel $ ( B^{ ( 1 ) } )^{ \nu } $ is an analytic Hilbert module.
\end{prop}

\begin{proof}
    (i) From the previous proposition it follows that $ ( B^{ ( 1 ) } )^{ \nu } $ is an well defined function on $ \mathbb{ G }_2 \times \mathbb{ G }_2 $. So, it remains to check if $ ( B^{ ( 1 ) } )^{ \nu } $, for $ \nu > 0 $, defines a non-negative definite kernel function on $ \mathbb{ G }_2 $. Recall that $ B^{ ( 1 )}_{\boldsymbol{s}} : \mathbb{ D }^2 \times \mathbb{ D }^2 \to \mathbb{ C } $ is the function defined by $ B^{ ( 1 ) }_{ \boldsymbol{ s } } ( \boldsymbol{ z }, \boldsymbol{ w } ) = B^{ ( 1 ) } ( \boldsymbol{ s } ( \boldsymbol{ z } ), \boldsymbol{ s } ( \boldsymbol{ w } ) ) $, $ \boldsymbol{ z }, \boldsymbol{ w } \in \mathbb{ D }^2 $. Consequently, it suffices to verify the non-negative definiteness of the function $ \left(B^{ ( 1 )}_s \right)^\nu$.
    
    Since $ \Log ( 1 - z_j \bar{ w }_{ \ell } ) $, for $ z_j, w_{ \ell } \in \mathbb{ D } $, $ 1 \leq j, \ell \leq 2 $, are well defined, we consider the function $ f : \mathbb{ D }^2 \times \mathbb{ D }^2 \to \mathbb{ C } $ defined by 
    $$ f ( \boldsymbol{ z }, \boldsymbol{ w } ) : = -\sum_{ j, \ell = 1 }^2 \Log ( 1 - z_{ j } \bar{ w }_{ \ell } ) , ~~ \boldsymbol{ z } = ( z_1, z_2 ), \boldsymbol{ w } = ( w_1, w_2 ) .$$
    Note that $ f $ is holomorphic in the first variable $ \boldsymbol{ z } $ and anti-holomorphic in the second variable $ \boldsymbol{ w } $.
    
 For $ \boldsymbol{ z }, \boldsymbol{ w } \in \mathbb{ D }^2 $, expanding $ f $ in the power series, we have
\begin{align*}
f(\boldsymbol{z}, \boldsymbol{w}) & =-\sum_{j, \ell=1}^2 \Log \left(1-z_j \overline{w_{\ell}}\right) \\
& =\sum_{n=1}^{\infty} \frac{\left(z_1^n+z_2^n\right) \overline{\left(w_1^n+w_2^n\right)}}{n} .
\end{align*}
Each kernel
$$
(\boldsymbol{z}, \boldsymbol{w}) \longmapsto \frac{\left(z_1^n+z_2^n\right) \overline{\left(w_1^n+w_2^n\right)}}{n}
$$
is a rank-one non-negative definite kernel. Hence $f$, being their pointwise limit, is non-negative definite. It then follows from the Schur product theorem that $f^m$ is non-negative definite for every $m \geq 0$. Consequently,
$$
\exp (\nu f)=\sum_{m=0}^{\infty} \frac{\nu^m}{m!} f^m
$$
is non-negative definite. Therefore, $
\left(B_{ \boldsymbol{ s } }^{(1)}\right)^\nu=2^{-\nu} \exp (\nu f)$
is non-negative definite.
    
       (ii) For $\nu=1$, the assertion follows from the fact that $\mathbb A^{(1)}(\mathbb G_2)$ is an analytic Hilbert module. We may therefore assume that $\nu>1$. Let $(M_{u_1}, M_{u_2})$ denote the pair of multiplication by the coordinate functions on $\mathbb A^{(1, \nu)}$. Since the multiplication operators by the coordinate functions are bounded on $\mathbb A^{(1)}(\mathbb G_2)$, it follows that there exist $ c_i > 0 $ such that $(c_i - u_i \bar v_i) B^{(1)}(\boldsymbol{u}, \boldsymbol{v}) $
    is a non-negative definite kernel on $ \mathbb{ G }_2 $ for each $i = 1, 2$. Hence, for $\nu > 1$ and $i = 1, 2$,
    $$(c_i - u_i \bar v_i) \left(B^{(1)}(\boldsymbol{u}, \boldsymbol{v})\right)^{\nu} = (c_i - u_i \bar v_i) B^{(1)}(\boldsymbol{u}, \boldsymbol{v}) \left(B^{(1)}(\boldsymbol{u}, \boldsymbol{v})\right)^{\nu - 1},\,\,\boldsymbol{u}, \boldsymbol{v} \in \mathbb G_2,$$
    is a non-negative definite kernel, being Schur product of two non-negative definite kernels, verifying the boundedness of $M_{u_1}$ and $M_{u_2}$ on $\mathbb A^{(1, \nu)}$.

Since $(B^{(1)})^\nu(\cdot,(0,0))\equiv2^{-\nu}$,
the constant function $1$ belongs to $\mathbb A^{(1,\nu)}$. The boundedness of the coordinate multiplication operators then implies that
$\mathbb C[\boldsymbol u]\subseteq\mathbb A^{(1,\nu)}$. Moreover, being a reproducing kernel Hilbert space of holomorphic functions, $ \mathbb{ A }^{ ( 1, \nu ) } $ possesses the dense set 
    $$ \Big\{\bar{\partial}^p_{v_1} \bar{\partial}^q_{v_2} \left(B^{(1)}(\cdot, (v_1, v_2))\right)^{\nu}|_{(v_1, v_2) = (0,0)} : p, q \geq 0 \Big\} . $$
    Thus, proving that $\mathbb A^{(1, \nu)}$ is an analytic Hilbert module reduces to verifying that the vectors $\bar{\partial}^p_{v_1} \bar{\partial}^q_{v_2} \left(B^{(1)}(\cdot, (v_1, v_2))\right)^{\nu}|_{(v_1, v_2) = (0,0)}$ are polynomials for $p , q \geq 0$. This can be seen by expressing the kernel $ (B^{(1)})^{\nu} $ in the coordinates of $ \mathbb{ G }_2 $.
        
    Recall from Equation \eqref{defofwbergker}, for $ \lambda = 1 $ and $\boldsymbol{z} = (z_1, z_2), \boldsymbol{w} = (w_1, w_2) \in \mathbb D^2$, that  
    $$\left(B^{(1)} (s(\boldsymbol{z}), s(\boldsymbol{w}))\right)^\nu = \frac{1}{2^{\nu}[(1 - z_1\bar{w}_1)(1 - z_2 \bar{w}_2)(1 - z_1 \bar{w}_2)(1 - z_2 \bar{w}_1)]^\nu} . $$
    Substituting $z_1 + z_2 = u_1$, $z_1 z_2 = u_2$, $w_1 + w_2 = v_1$ and $w_1 w_2 = v_2$ in this equation, we obtain
    $$\left(B^{(1)} ((u_1, u_2), (v_1, v_2))\right)^\nu = \frac{1}{2^{\nu}[1 - \left(u_1 \bar v_1 + 2 u_2 \bar v_2 - u_1^2 \bar v_2 - u_2 \bar v_1^2 + u_1 \bar v_1^2 + u_1 u_2 \bar v_1 \bar v_2\right)]^\nu}.$$
    Since $u_1 \bar v_1 + 2 u_2 \bar v_2 - u_1^2 \bar v_2 - u_2 \bar v_1^2 + u_1 \bar v_1^2 + u_1 u_2 \bar v_1 \bar v_2 \to 0$ as $u_1, u_2, v_1, v_2 \to 0$, we can find an open set $U$ containing $(0,0)$ in $\mathbb G_2$ such that 
    $$|u_1 \bar v_1 + 2 u_2 \bar v_2 - u_1^2 \bar v_2 - u_2 \bar v_1^2 + u_1 \bar v_1^2 + u_1 u_2 \bar v_1 \bar v_2| < 1, ~~  ( u_1, u_2 ), ( v_1, v_2 ) \in U .$$
    Therefore, for $(u_1, u_2), (v_1, v_2) \in U$, we have
    \begin{flalign*}
       \left(B^{(1)} ((u_1, u_2), (v_1, v_2))\right)^\nu = \frac{1}{2^\nu}\sum_{n \geq 0} \frac{(\nu)_n}{n!} (u_1 \bar v_1 + 2 u_2 \bar v_2 - u_1^2 \bar v_2 - u_2 \bar v_1^2 + u_1 \bar v_1^2 + u_1 u_2 \bar v_1 \bar v_2)^n, 
    \end{flalign*}
    where $(\nu)_n = \nu (\nu +1) \cdots (\nu + n - 1)$ is the Pochhammer symbol. A direct computation then verifies that 
    $$\bar{\partial}_{v_1}^p \bar{\partial}_{v_2}^q (u_1 \bar v_1 + 2 u_2 \bar v_2 - u_1^2 \bar v_2 - u_2 \bar v_1^2 + u_1 \bar v_1^2 + u_1 u_2 \bar v_1 \bar v_2)^n|_{(v_1, v_2) = (0,0)} = 0$$
    for all $n > p + q$ and $ p, q \in \mathbb{ N } \cup \{ 0 \} $. This proves that $\bar{\partial}^p_{v_1} \bar{\partial}^q_{v_2} \left(B^{(1)}(\cdot, (v_1, v_2))\right)^{\nu}|_{(v_1, v_2) = (0,0)}$ is a polynomial for all $ p, q \in \mathbb{ N } \cup \{ 0 \} $. Since both sides are holomorphic in $\boldsymbol u$, and $\mathbb G_2$ is connected, the identity theorem shows that this equality holds on all of $\mathbb G_2$.
    \end{proof}

However, for $\lambda > 2$, 
$B^{(\lambda)}$ has zeros in $\mathbb{G}_2$ and therefore, any power other than the ones given by natural numbers is not continuously defined. So, for such $\lambda$'s the only non-negative kernel is when $\nu \in \mathbb{N}$. 

For $\lambda > 0$ and $\nu \in \mathcal{W}_{\mathbb G_2, a}(B^{(\lambda)})$, the Hilbert space $\mathbb{A}^{(\lambda, \nu)}$ determined by the reproducing kernel $(B^{(\lambda)})^\nu$ is an analytic Hilbert module. We now show, for $\nu \in \mathcal{W}_{\mathbb G_2, a}(B^{(\lambda)})$, these are  Aut$(\mathbb G_2)$-homogeneous analytic Hilbert modules $\mathbb{A}^{(\lambda, \nu)}(\mathbb{G}_2)$.

\begin{prop}\label{powerwtdberg}
For $\lambda > 0$ and $\nu \in \mathcal{W}_{\mathbb G_2, a}(B^{(\lambda)})$,
the analytic Hilbert module  $\mathbb{A}^{(\lambda, \nu)}$ is Aut$(\mathbb G_2)$-homogeneous.  
\end{prop}
\begin{proof}
    By hypothesis, $(B^{(\lambda)})^\nu$ is non-negative definite and the associated Hilbert space $\mathbb{A}^{(\lambda, \nu)}$ is analytic as $\nu \in \mathcal{W}_{\mathbb G_2, a}(B^{(\lambda)})$. 
    Since $B^{(\lambda)}$, $\lambda >0$, is a quasi-invariant non-negative definite kernel, it follows that $(B^{(\lambda)})^\nu$ is also quasi-invariant. Therefore, $\mathbb{A}^{(\lambda, \nu)}$ is homogeneous. 
\end{proof}

\begin{remark} \label{rem:4.11} Recalling \cite[Proposition 2.1]{AKGM}, we see from Equation \eqref{quasiinvproofbergker} that the map $U_{\tilde{\varphi}} : \mathbb{A}^{(\lambda, \nu)} \to \mathbb{A}^{(\lambda, \nu)}$, defined by
\[\left(U_{\tilde{\varphi}} f\right) (\boldsymbol{s}(\boldsymbol{z})) :=
\left(\hat{\phi}(\boldsymbol s( \boldsymbol{ z }))\right)^{\frac{\lambda + 1}{2}} 
f\left(\tilde{\varphi}(\boldsymbol{s}(\boldsymbol{z}))\right),\,\,f \in \mathbb{A}^{(\lambda, \nu)},\,\,\boldsymbol{z} = (z_1, z_2) \in \mathbb D^2\]
is unitary for every $\tilde{\varphi} \in \rm{Aut}(\mathbb G_2)$. We have pointed out earlier that the map $J : \rm{Aut}(\mathbb G_2) \times \mathbb G_2 \to \mathbb C$, defined by
\[J(\tilde{\varphi}, \boldsymbol{s}(\boldsymbol{z})) := 
\hat{\phi}(\boldsymbol s( \boldsymbol{ z })),\,\,\tilde{\varphi}\in\operatorname{Aut}(\mathbb{G}_2), \,\, \boldsymbol{s}(\boldsymbol{z}) \in \mathbb{G}_2,\]
is a cocycle and in consequence, so is $J^{\frac{\lambda+1}{2}}$. Therefore, $\tilde{\varphi} \mapsto U_{\tilde{\varphi}}$ defines a projective multiplier representation of $\operatorname{Aut}(\mathbb G_2)$. By the remark immediately following \cite[Proposition 2.1]{AKGM}, we infer that 
\[U_{\tilde{\varphi}}\tilde{\varphi}\big (\boldsymbol{M}^{(\lambda, \nu)}\big ) = \boldsymbol{M}^{(\lambda, \nu)} U_{\tilde{\varphi}},\,\, \tilde{\varphi} \in \operatorname{Aut}(\mathbb G_2),\] 
where $\boldsymbol{M}^{(\lambda, \nu)}$ is the pair of multiplication operators by the coordinate functions on the Hilbert space $\mathbb A^{(\lambda, \nu)}$. 
A factorization of the kernel $(B^{(\lambda)})^\nu$, analogous to the one given in Equation \eqref{quasiinvproofbergker1} (or, Corollary \ref{cor:4.5}) for $B^{(\lambda)}$, is now evident.

 We give an interesting decomposition of the unitary representation $U$ into irreducible components. Let $C_{\boldsymbol{s}} : \mathbb A^{(\lambda, \nu)} \to \mbox{\rm Hol}(\mathbb D^2, \mathbb{C})$ be given by $(C_{\boldsymbol{s}}f)(\boldsymbol{z}) = f(\boldsymbol{s}(\boldsymbol{z}))$, $f \in \mathbb A^{(\lambda, \nu)}$, $\boldsymbol{z} \in \mathbb D^2$. Declare $C_{\boldsymbol{s}}$ to be a unitary map onto its image $\mathcal{H}^{(\lambda, \nu)} := \mbox{\rm Im}\,\, C_{\boldsymbol{s}}$. Then $\tilde{U} = C_{\boldsymbol{s}} U C_{\boldsymbol{s}}^*$ defines a projective multiplier representation of M\"ob on $\mathcal{H}^{(\lambda,\nu)}$. Note that the Hilbert space $\mathcal{H}^{(\lambda,\nu)}$ consists of symmetric functions over $\mathbb{D}^2$ and contains all the symmetric polynomials. Therefore, \cite[Theorem 2.8(ii)]{DAH} ensures the existence of $\alpha > 0$ such that $\tilde{U}$ is unitarily equivalent to $\oplus_{m = 0}^\infty D^+_{\alpha + 4m}$.
\end{remark}   

\subsection{The Hilbert modules \texorpdfstring{$\mathbb H^{(\lambda, \nu)}$}{AB} over \texorpdfstring{$ \mathbb{ G }_2$}{BB}}
Let $\mathbb B^{(\lambda)}$ denote the curvature matrix of the weighted Bergman kernel $B^{(\lambda)}$ of $\mathbb G_2$, namely,   
\[\mathbb B^{(\lambda)}(w) = \left( \!\!\! \left (
\frac{~\partial^2}{\partial{w_i}
\partial{\overline w_j}}\log B^{(\lambda)}(w,w)\right )\!\!\!\right)_{i,j=1}^{2}, w\in \mathbb G_2.\] In what follows, we let $\mathbb B^{(\lambda, \nu)}$ denote the $2\times 2$ matrix $(B^{(\lambda)})^{\nu+2} \mathbb B^{( \lambda) }$. If $\nu \in \mathcal W_{\mathbb G_2}(B^{(\lambda)}) \cup \{ 0 \} $, then $\mathbb B^{(\lambda, \nu)}$, being a product of two non-negative definite kernels $(B^{(\lambda)})^\nu$ and $(B^{(\lambda)})^2\mathbb B^{(\lambda)}$, is a non-negative definite kernel. Let $\mathbb H^{(\lambda, \nu)}$ be the Hilbert space with the reproducing kernel $\det \mathbb B^{(\lambda, \nu)}$.

\begin{theorem}\label{detkernelofberg}
For $\lambda > 0$ and $\nu \in \mathcal{W}_{\mathbb G_2, a}(B^{(\lambda)}) \cup \{ 0 \} $, the Hilbert space $\mathbb H^{(\lambda, \nu)}$ is an \rm{Aut}$(\mathbb G_2)$-homogeneous analytic Hilbert module.   
\end{theorem}

\begin{proof}
The fact that the Hilbert space $\mathbb A^{(\lambda)}(\mathbb G_2)$ with the reproducing kernel $B^{(\lambda)}$ is an Aut$(\mathbb G_2)$-homogeneous analytic Hilbert module follows form Proposition \ref{homogofweightedbergker}. Consequently, the proof of the theorem follows from Theorem \ref{det theorem}.    
\end{proof}
\begin{remark}
Again, a factorization formula for $\det \mathbb B^{(\lambda, \nu)}$, analogous to the one in Corollary \ref{cor:4.5} is the following:
\begin{flalign*}
    \det \mathbb B^{(\lambda, \nu)} (\boldsymbol{u}, \boldsymbol{v}) &= (\hat{\phi}(\boldsymbol{u}))^{(\lambda + 1)(\nu + 2)} \det \left(D \tilde{\varphi}(\boldsymbol{u})\right)\\ 
    &\phantom{aaaaaaaa}\det \mathbb B^{(\lambda, \nu)} (\tilde{\varphi}(\boldsymbol{u}), \tilde{\varphi}(\boldsymbol{v}) ) \overline{(\hat{\phi}(\boldsymbol{v}))^{(\lambda + 1)(\nu + 2)} \det \left(D \tilde{\varphi}(\boldsymbol{v})\right)},\,\,\boldsymbol{u}, \boldsymbol{v} \in \mathbb{G}_2,
\end{flalign*}
establishing that $\mathbb{H}^{(\lambda, \nu)}$ is quasi-invariant. 
\end{remark}

\section{Inequivalence}

In the preceding section, we have introduced two  natural classes of Aut$ ( \mathbb{ G }_2 ) $-homogeneous analytic Hilbert modules,  namely, $\{\mathbb{A}^{(\lambda, \nu)} : \lambda > 0\,\,\mbox{and}\,\,\nu \in \mathcal{W}_{\mathbb G_2, a}(B^{(\lambda)})\}$ and $\{\mathbb H^{(\mu, \eta)} : \mu > 0\,\,\mbox{and}\,\,\eta \in \mathcal{W}_{\mathbb G_2, a}(B^{(\mu)}) \cup \{ 0 \}\}$. We now prove that the Hilbert modules in each of these two classes are unitarily inequivalent among themselves. What is more is that none of $ \mathbb{A}^{(\lambda, \nu)} $ is unitarily equivalent to any of $ \mathbb H^{(\mu, \eta)} $. The main technique used to prove these results is to compare the curvatures of the corresponding hermitian holomorphic line bundles. First, we must compute the curvature. This computation for the unweighted Bergman kernel is in \cite{CY}.

For the curvature computation involving the weighted Bergman kernel, it will be useful to set up the notation
$B^{(\lambda)}_{\boldsymbol{s}} : \mathbb D^2 \times \mathbb D^2 \to \mathbb C$, defined by $B^{(\lambda)}_{\boldsymbol s}(\boldsymbol{z}, \boldsymbol{w}) = B^{(\lambda)}(\boldsymbol s(\boldsymbol{z}), \boldsymbol s(\boldsymbol{w})),$ $\boldsymbol{z}, \boldsymbol{w} \in \mathbb D^2$. Note that $B^{(\lambda)}_{\boldsymbol s}$ is non-negative definite. Similarly, $\mathbb B^{(\lambda)}_{\boldsymbol{s}}$ denotes the curvature matrix corresponding to $B^{(\lambda)}_{\boldsymbol s}$. 

\begin{prop}\label{curv}
 The  entries $\mathbb B_{i \bar j}^{(\lambda)}$, $1 \leq i, j\leq 2$, of the curvature matrix $\mathbb B^{(\lambda)}$ are given by 
\begin{align*}
& \mathbb B^{(\lambda)}_{1\bar{1}}(z,0) = b_{1\bar{1}}^{(\lambda)}(z,0), \hspace{0.2in} \mathbb B^{(\lambda)}_{2\bar{2}}(z,0) = \frac{1}{|z|^2} \left[ b_{1\bar{1}}^{(\lambda)}(z,0) - 2 b^{ ( \lambda ) }_{ 1 \bar{2}}( z, 0 ) + b_{2\bar{2}}^{(\lambda)}(z,0) \right] \\
& \mathbb B^{(\lambda)}_{1\bar{2}}(z,0) = \overbar{\mathbb B^{(\lambda)}_{2\bar{1}}(z,0)} = \frac{1}{z} \left[ b^{ ( \lambda ) }_{ 1 \bar{2}}( z, 0 ) -  b^{ ( \lambda ) }_{ 1 \bar{1}}( z, 0 ) \right] ,\,\,z \in \mathbb D,
\end{align*} 
where $b_{i \bar j}^{(\lambda)}$, $1 \leq i, j \leq 2$, are the entries of the curvature matrix $\mathbb B^{(\lambda)}_{\boldsymbol{s}}$ given by 
\begin{align*}
& b_{1\bar{1}}^{(\lambda)}(z,0) = \frac{\lambda\left[1 - (1-|z|^2)^\lambda ( 1 + \lambda | z |^2 ) \right]}{(1 - |z|^2)^2 \left[1 - (1 - |z|^2)^\lambda \right]^2}, \hspace{0.2in}  b_{2\bar{2}}^{(\lambda)}(z,0) = ( 1 - | z |^2 )^2 b_{1\bar{1}}^{(\lambda)}(z,0) \\
& b_{1\bar{2}}^{(\lambda)}(z,0) = b_{2\bar{1}}^{(\lambda)}(z,0) = \frac{\lambda ( 1 - |z|^2 )^{ \lambda }\left[\lambda |z|^2 - (1-|z|^2) + (1 - |z|^2)^{ \lambda + 1 } \right]}{(1 - |z|^2) \left[1 - (1 - |z|^2)^\lambda \right]^2},\,\,z \in \mathbb D.
\end{align*}
\end{prop}

\begin{proof}
For $1 \leq i, j \leq 2$ and $\boldsymbol{z} \in \mathbb D^2$, note that 
\begin{flalign}\label{curvcom1}
   \nonumber b_{i\bar{j}}^{(\lambda)} (\boldsymbol{z}) &= \frac{\partial^2}{\partial z_i \partial \bar{z}_j} \log B^{(\lambda)}_s(\boldsymbol{z}, \boldsymbol{z})\\
    &= \frac{B^{(\lambda)}_s(\boldsymbol{z}, \boldsymbol{z}) \frac{\partial^2}{\partial z_i \partial \bar{z}_j} B^{(\lambda)}_s(\boldsymbol{z}, \boldsymbol{z}) - \frac{\partial}{\partial z_i} B^{(\lambda)}_s(\boldsymbol{z}, \boldsymbol{z}) \frac{\partial}{\partial \bar{z}_j} B^{(\lambda)}_s(\boldsymbol{z}, \boldsymbol{z}) }{\left(B^{(\lambda)}_s(\boldsymbol{z}, \boldsymbol{z})\right)^2}.
\end{flalign}

Replacing $\boldsymbol{z}$ and $\boldsymbol{w}$ by $(z, 0)$, $z \in \mathbb D$, in Equation \eqref{defofwbergker}, we obtain
$$B^{(\lambda)}_s((z,0), (z, 0)) = \frac{1 - (1 - |z|^2)^\lambda}{2|z|^2(1 - |z|^2)^\lambda}.$$

Let $z$ be an arbitrary element of $\mathbb D$. From a direct computation, we obtain the followings:
\begin{align*}
& \partial_1 B^{(\lambda)}_{\boldsymbol s} ((z, 0), (z, 0)) = \frac{1}{2|z|^2} \left[ \frac{\lambda \bar z}{(1 - |z|^2)^{\lambda + 1}} - \frac{1}{z(1 - |z|^2)^\lambda} + \frac{1}{z} \right],\\
& \bar{\partial}_1 B^{(\lambda)}_{\boldsymbol s} ((z, 0), (z, 0)) = \frac{1}{2|z|^2} \left[ \frac{\lambda z}{(1 - |z|^2)^{\lambda + 1}} - \frac{1}{\bar{z}(1 - |z|^2)^\lambda} + \frac{1}{\bar{z}} \right],\\
& \partial_2 B^{(\lambda)}_{\boldsymbol s} ((z, 0), (z, 0)) = \frac{1}{2|z|^2} \left[ \frac{1}{z(1 - |z|^2)^{\lambda}} - \frac{1}{z} - \lambda \bar z\right],\\
& \bar{\partial}_2 B^{(\lambda)}_{\boldsymbol s} ((z, 0), (z, 0)) = \frac{1}{2|z|^2} \left[ \frac{1}{\bar{z}(1 - |z|^2)^{\lambda}} - \frac{1}{\bar{z}} - \lambda z\right],\\
& \partial_1  \bar{\partial}_1 B^{(\lambda)}_{\boldsymbol s} ((z, 0), (z, 0)) = \frac{1}{2|z|^2} \left[ \frac{\lambda |z|^2(\lambda |z|^2 - 1) + (|z|^2 - 1)^2 + 2 \lambda |z|^4}{|z|^2(1 - |z|^2)^{\lambda + 2}} - \frac{1}{|z|^2} \right],\\
& \partial_1  \bar{\partial}_2 B^{(\lambda)}_{\boldsymbol s} ((z, 0), (z, 0)) = \frac{1}{2|z|^2} \left[ \frac{1}{|z|^2} - \frac{1}{|z|^2(1 - |z|^2)^\lambda} + \frac{\lambda}{(1 - |z|^2)^{\lambda + 1}} \right]\\
&\hspace{3.88cm} = \partial_2 \bar{\partial}_1 B^{(\lambda)}_{\boldsymbol s} ((z, 0), (z, 0)),\\
& \partial_2  \bar{\partial}_2 B^{(\lambda)}_{\boldsymbol s} ((z, 0), (z, 0)) = \frac{1}{2|z|^2} \left[\frac{1}{|z|^2(1 - |z|^2)^\lambda} + \frac{\lambda}{(1 - |z|^2)^{\lambda}} - \lambda^2|z|^2 - 2\lambda\right].
\end{align*}
Substituting the values of $B^{(\lambda)}_{\boldsymbol s}((z, 0), (z, 0))$, $\partial_i B^{(\lambda)}_{\boldsymbol s}((z, 0), (z, 0))$, $\bar \partial_j B^{(\lambda)}_{\boldsymbol s}((z, 0), (z, 0))$ and $\partial_i \bar \partial_j B^{(\lambda)}_{\boldsymbol s}((z, 0), (z, 0))$ in Equation \eqref{curvcom1}, we obtain the expression of $b_{i \bar j}^{(\lambda)}(z, 0)$, for $1 \leq i, j \leq 2$, as it is stated in the statement of this proposition.

Since $B^{(\lambda)}_{\boldsymbol s}(\boldsymbol{z}, \boldsymbol{w}) = B^{(\lambda)} (\boldsymbol s(\boldsymbol{z}), \boldsymbol s(\boldsymbol{w}))$, $\boldsymbol{z}, \boldsymbol{w} \in \mathbb D^2$, an application of the chain rule yields that 
\begin{equation}\label{curvcom2}
 \mathbb B^{(\lambda)}_{\boldsymbol s} (\boldsymbol{z}) = D \boldsymbol s(\boldsymbol{z})^t \mathbb B^{(\lambda)}(\boldsymbol s(\boldsymbol{z})) \overbar{D \boldsymbol s(\boldsymbol{z})},\,\,\boldsymbol{z} \in \mathbb D^2.   \end{equation}
Let $\boldsymbol{z} = (z_1, z_2)$ be an arbitrary element of $\mathbb D^2$. Note that 
$$D\boldsymbol s(\boldsymbol{z}) = \begin{pmatrix}
1 & 1\\
z_2 & z_1
\end{pmatrix}.$$
If $z_1 \neq z_2$, then $D\boldsymbol s(z)$ is invertible. Thus, if $0 \neq z \in \mathbb D$, substituting the values of $D\boldsymbol s(z, 0)$ and $\mathbb B^{(\lambda)}_{\boldsymbol s}$ in Equation \eqref{curvcom2}, we obtain $\mathbb B^{(\lambda)} (z, 0)$ which, in particular, gives us the expressions of $\mathbb B^{(\lambda)}_{i \bar j}(z, 0)$, $1 \leq i, j \leq 2$, as stated in the statement of this proposition. Since $\mathbb B^{(\lambda)}$ is a continuous function on $\mathbb G_2$, it follows that the expressions of $\mathbb B^{(\lambda)}_{i \bar j}(z, 0)$, $1 \leq i, j \leq 2$, as stated in the statement of this proposition, holds for every $z \in \mathbb D$.
\end{proof}
 We briefly explain the idea of the proof of the theorem below. Two other theorems, namely, Theorem \ref{thminequiv} and Theorem \ref{inequivofdetker} are proved based on the same idea. As we have pointed out earlier two analytic Hilbert modules $(\mathcal {H} , K)$ and $(\mathcal {H}^\prime, K^\prime)$ in $\operatorname{HM}_a(\Omega)$ are equivalent if and only if there is a change of scale $\Psi$, that is, if and only if there is a nowhere vanishing holomorphic function $\Psi$ on $\Omega$ such that $K^\prime=\Psi K \overbar{\Psi}$. Thus, if $\mathcal{Z} \subset \Omega$ and $\mathcal{Z}$ is open in  $\mathbb{C}^k$ for some $k,\, k \leq n$, then $\Psi_{|\operatorname{res}\,\mathcal{Z}}$ is holomorphic. 
The restriction of $K$ and $K^\prime$ are holomorphic in the first variable and anti-holomorphic in the second on $\mathcal{Z}$.
Therefore, if $K^\prime$ is obtained from $K$ by a change of scale then restriction of $K^\prime$ to $\mathcal{Z}\times \mathcal{Z}$ is also obtained from $K$ restricted to $\mathcal{Z}\times \mathcal{Z}$ by a change of scale, namely, the restriction of $\Psi$ to $\mathcal{Z}$.  
If $(\mathcal{H}, K)$ is in $\operatorname{HM}_a(\Omega)$, then $(\mathcal{H}_{|\operatorname{res}\,\mathcal{Z}}, K_{| \operatorname{res}\,{\mathcal{Z \times\mathcal}Z}})$ is also in $\operatorname{HM}_a(\mathcal{Z})$, where 
$\mathcal{H}_{|\operatorname{res} \mathcal{Z}} = \{g\mid g= f_{|\operatorname{res} \, \mathcal{Z}}, f \in \mathcal{H}\}$. If we let $\mathcal{H}_0\subset \mathcal{H}$ be the submodule of functions vanishing on $\mathcal{Z}$, then $\mathcal{H}_{|\operatorname{res} \mathcal{Z}}$ is the quotient module $\mathcal{Q}:= \mathcal{H}\ominus \mathcal{H}_0$. Therefore, in this situation, if two analytic Hilbert modules $(\mathcal{H}, K)$ and $(\mathcal{H}^\prime, K^\prime)$ are equivalent, then so are the quotient modules $\mathcal{Q}= \mathcal{H} \ominus \mathcal{H}_0$ and $\mathcal{Q}^\prime=\mathcal{H} \ominus \mathcal{H}_0$. 
This is what we exploit in the proofs below where $\Omega=\mathbb{D}^2$ and $\mathcal{Z} \subset \mathbb{D}^2$ is the open set $\{(z,z)\mid z\in \mathbb{D}\}$ in $\mathbb{C}$.

\begin{theorem}\label{inequivofweightedbergker}
The \rm{Aut}$(\mathbb G_2)$-homogeneous analytic Hilbert modules $\mathbb A^{(\lambda_1, \nu_1)}$ and $\mathbb A^{(\lambda_2, \nu_2)}$ are unitarily equivalent if and only if $\lambda_1 = \lambda_2$ and $\nu_1 = \nu_2$.   
\end{theorem}

\begin{proof}
Assume that the Hilbert modules $\mathbb A^{(\lambda_1, \nu_1)}$ and $\mathbb A^{(\lambda_2, \nu_2)}$ are unitarily equivalent. Note that both $\mathbb A^{(\lambda_1, \nu_1)}$ and $\mathbb A^{(\lambda_2, \nu_2)}$ are analytic Hilbert modules. Therefore, there exists a non-vanishing holomorphic functions $F : \mathbb G_2 \to \mathbb C$ such that
\begin{equation}\label{sebsec2:eqn1}
\left(B^{(\lambda_1)}(\boldsymbol{s}( \boldsymbol{z}), \boldsymbol{s}(\boldsymbol{w}) )\right)^{\nu_1} = F(\boldsymbol{s}(\boldsymbol{z}))  \left(B^{(\lambda_2)}(\boldsymbol{s}( \boldsymbol{z}), \boldsymbol{s}(\boldsymbol{w}) )\right)^{\nu_2} \overbar{F(\boldsymbol{s}(\boldsymbol{w}))},\,\,\boldsymbol{z}, \boldsymbol{w} \in \mathbb D^2.  
\end{equation}
Substituting $ \boldsymbol{ z } $ and $ \boldsymbol{ w } $ by $ ( z, z ) $, $z \in \mathbb D$, in the equation above, Lemma \ref{weightedbergkerexpansion} yields that  
\begin{align*}
\left( \frac{ \lambda_1 }{ 2 } \times \frac{ 1 }{ ( 1 - | z |^2 )^{ 2 \lambda_1 + 2 } } \right)^{ \nu_1 } 
= | F ( s ( z, z ) ) |^2 \left( \frac{ \lambda_2 }{ 2 } \times \frac{ 1 }{ ( 1 - | z |^2 )^{ 2 \lambda_2 + 2 } } \right)^{ \nu_2 } .
\end{align*}
Taking logarithm and then applying $ \partial \overline{ \partial } $ to both sides of the equation above, we get
\begin{equation} \label{kerequiv3webergker}
      \nu_1 ( \lambda_1 + 1 ) = \nu_2(\lambda_2 + 1).
\end{equation}
Letting $F_{\boldsymbol{s}} = F \circ \boldsymbol{s}$ observe that Equation \eqref{sebsec2:eqn1} can be paraphrased as
\begin{equation*}
\left(B_{\boldsymbol{s}}^{(\lambda_1)}(\boldsymbol{z}, \boldsymbol{w} )\right)^{\nu_1} = F_{\boldsymbol{s}}(\boldsymbol{z})  \left(B_{\boldsymbol{s}}^{(\lambda_2)}(\boldsymbol{z}, \boldsymbol{w})\right)^{\nu_2} \overbar{F_{\boldsymbol{s}}(\boldsymbol{w})},\,\,\boldsymbol{z}, \boldsymbol{w} \in \mathbb D^2.  
\end{equation*}
Now substituting $\boldsymbol{z} = \boldsymbol{w}$ and applying the operator $\partial_2 \overbar{\partial_2} \log$ to both sides of this equation evaluated at $(z, 0) \in \mathbb D \times \{ 0 \}$, we obtain
\begin{equation} \label{eqn:4.12}
\nu_1 b_{2 \bar 2}^{(\lambda_1)}(z, 0) = \nu_2 b_{2 \bar 2}^{(\lambda_2)}(z, 0).
\end{equation}
Replacing the values of $b_{2 \bar 2}^{(\lambda_1)}(z, 0)$ and $b_{2 \bar 2}^{(\lambda_2)}(z, 0)$ from Proposition \ref{curv} and finally, considering the limit $z \to 1$ to both sides of \eqref{eqn:4.12}, we obtain 
\begin{equation} \label{kerequiv3webergker1}
      \nu_1  \lambda_1  = \nu_2\lambda_2.
\end{equation}
The equations \eqref{kerequiv3webergker} and \eqref{kerequiv3webergker1} together imply that $\lambda_1 = \lambda_2$ and $\nu_1 = \nu_2$.
\end{proof}

From Theorem \ref{detkernelofberg}, we see that $\{\mathbb H^{(\lambda, \nu)} : \lambda > 0\,\,\mbox{and}\,\,\nu \in \mathcal{W}_{\mathbb G_2, a}(B^{(\lambda)}) \cup \{ 0 \} \}$ is a second family of Aut$(\mathbb G_2)$-homogeneous analytic Hilbert modules. We now show that the Hilbert modules $\mathbb H^{(\lambda, \nu)}$ and $\mathbb A^{(\mu, \eta)}$ are unitarily inequivalent for any $\lambda, \mu > 0$ and $\nu \in \mathcal{W}_{\mathbb G_2, a}(B^{(\lambda)}) \cup \{ 0 \} $, $\eta \in \mathcal{W}_{\mathbb G_2, a}(B^{(\mu)})$.  To this end, we prove a couple of lemmas leading to computation of $ \det \mathbb B^{ ( \lambda, \nu ) }$.

\begin{lemma} \label{det-curv}
The determinant $ \det \mathbb{B}^{( \lambda ) } $ of the curvature matrix $ \mathbb{ B }^{ ( \lambda ) } $ is given by the formula 
\begin{align*}
\det \mathbb{ B }^{(\lambda)}(\boldsymbol{s} (\boldsymbol{z}), \boldsymbol{s} (\boldsymbol{z})) =& \frac{ \lambda^2 B^{(\lambda+2)}(\boldsymbol{s}(\boldsymbol{z}), \boldsymbol{s}(\boldsymbol{z}))}{2|1 - \bar z_1z_2|^{2\lambda + 2}(1 - |z_1|^2)^{\lambda+1}(1 - |z_2|^2)^{\lambda + 1} \left(B^{(\lambda)}(\boldsymbol{s}(\boldsymbol{z}), \boldsymbol{s}(\boldsymbol{z}))\right)^3} \times \\
& \frac{B^{(\lambda)}(\boldsymbol{s}(\boldsymbol{z}), \boldsymbol{s}(\boldsymbol{z})) \left(|1 - \bar z_1z_2|^{2\lambda + 2} + (1 - |z_1|^2)^{\lambda+1} (1 - |z_2|^2)^{\lambda+1}\right) - \lambda}{|z_1 - z_2|^4}, 
\end{align*}
for $\boldsymbol{z} = (z_1, z_2) \in \mathbb D^2$.
\end{lemma}

\begin{proof}
It follows from Proposition \ref{curv} that
\begin{align*} 
& b^{( \lambda ) }_{ 1 \bar{ 1 } } ( z, 0 ) b^{( \lambda ) }_{ 2 \bar{ 2 } } ( z, 0 ) - b^{( \lambda ) }_{ 1 \bar{ 2 } } ( z, 0 ) b^{( \lambda ) }_{ 2 \bar{ 1 } } ( z, 0 ) =\\
&\phantom{gadadharmisrapra}\frac{ \lambda^2 }{ 2 | z |^2 B^{ ( \lambda ) }_0 ( z )^2 ( 1 - | z |^2 )^{ 1 + \lambda } } \left[ ( 1 - | z |^2 )^{ \lambda + 1 } B_0^{ ( \lambda ) } ( z ) \frac{ B_0^{ ( 2 \lambda + 2 ) } ( z ) }{ B_0^{ ( \lambda + 1 ) } ( z ) } - \lambda \right], ~~~ z \in \mathbb{D},  
\end{align*}
where we set, for $ z \in \mathbb{ D } $,
$$ B_0^{ ( \lambda ) } ( z ) : = B^{ ( \lambda ) } ( ( z, 0), ( z, 0 ) ) = \frac{ 1 }{ 2 | z |^2 } \frac{ \left( 1 - ( 1 - | z |^2 )^{ \lambda } \right) }{ ( 1 - | z |^2 )^{ \lambda } } .$$ 
Consequently, the determinant of the curvature matrix $ \mathbb{ B }^{ ( \lambda ) } $ at $ ( z, 0 ) \in \mathbb{ G }_2 $ is given by the formula \begin{equation} \label{eqn:3.6.1}
\det \mathbb{ B }^{ ( \lambda ) } ( (z, 0), (z, 0) ) = \frac{ \lambda^2 }{ 2 | z |^4 B^{ ( \lambda ) }_0 ( z )^2 ( 1 - | z |^2 )^{ \lambda + 1 } } \left[ ( 1 - | z |^2 )^{ \lambda + 1 } B_0^{ ( \lambda ) } ( z ) \frac{ B_0^{ ( 2 \lambda + 2 ) } ( z ) }{ B_0^{ ( \lambda + 1 ) } ( z ) } - \lambda \right] . 
\end{equation}
For $z_1, z_2\in \mathbb D$,  replacing $z$ by $\varphi_{z_1}(z_2)$ in Equation \eqref{eqn:3.6.1}, we have
\begin{align} \label{eqn:3.6.1.1}
\nonumber &\det \mathbb{ B }^{ ( \lambda ) } (( \varphi_{z_1}(z_2), 0 ), ( \varphi_{z_1}(z_2), 0 ))\\
\nonumber &\phantom{gadadhar} = \frac{\lambda^2 B^{(\lambda + 1)}(\boldsymbol{s}(\boldsymbol{z}), \boldsymbol{s}(\boldsymbol{z}))}{2 \left(B^{(\lambda)}(\boldsymbol{s}(\boldsymbol{z}), \boldsymbol{s}(\boldsymbol{z}))\right)^3 (1 - |z_1|^2)^{\lambda+1} (1 - |z_2|^2)^{\lambda+1} |1 - \bar z_1 z_2|^{2\lambda - 4}} \times\\
&\phantom{gadadhar} \frac{\left[ B^{(\lambda)}(\boldsymbol{s}(\boldsymbol{z}), \boldsymbol{s}(\boldsymbol{z})) \left(|1 - \bar z_1 z_2|^{2\lambda + 2} + (1 - |z_1|^2)^{\lambda+1} (1 - |z_2|^2)^{\lambda + 1}\right) - \lambda \right]}{|z_1 - z_2|^4}, 
\end{align}
where $ \varphi_{ z_1 } ( z_2 ) = \frac{ z_2 - z_1 }{ 1 - \bar{ z_1 } z_2 } $.

Since $\mathbb A^{(\lambda)}(\mathbb G_2)$ is an Aut$(\mathbb G_2)$-homogeneous analytic Hilbert module, it follows from Equation \eqref{eq:1} that the curvature matrix $ \mathbb{ B }^{ ( \lambda ) } $ of the correspoding hermitian holomorphic line bundle $\mathsf L_{\mathbb A^{(\lambda)}(\mathbb G_2)}$ satisfies the following identity
\begin{align} \label{eqn:3.6.2}
\nonumber &\mathbb{ B }^{ ( \lambda ) } ( \boldsymbol{ s } ( \boldsymbol{ z }), \boldsymbol{ s } ( \boldsymbol{ z }))\\
&\phantom{xxx} = \big ( D ( \widetilde{ \varphi_{ z_1 } } ) ( \varphi_{ z_1 } ( z_2 ), 0 )^{\operatorname{tr}} ) \big )^{ - 1 } \mathbb{ B }^{ ( \lambda ) }  (( \varphi_{ z_1 } ( z_2 ), 0 ), ( \varphi_{ z_1 } ( z_2 ), 0 ))  \overline{ \big ( D ( \widetilde{ \varphi_{ z_1 } } ) ( \varphi_{ z_1 } ( z_2 ), 0 ) \big ) }^{ - 1 }.
\end{align}
Observe that $\det D \widetilde{\varphi_{z_1}} \left(\varphi_{z_1}(z_2), 0\right) = -(1 - \bar z_1 z_2)^3$. Consequently, taking the determinant in both sides of Equation \eqref{eqn:3.6.2} and substituting $ \det \mathbb{ B }^{ ( \lambda ) } (( \varphi_{z_1}(z_2), 0 ), ( \varphi_{z_1}(z_2), 0 ))$ from Equation \eqref{eqn:3.6.1.1}, we obtain the desired identity.
\end{proof}

Recall that $\mathbb B^{(\lambda, \nu)} = (B^{(\lambda)})^{\nu+2} \mathbb B^{( \lambda) } $ is the reproducing kernel on $ \mathbb{ G }_2 $ and $ \mathbb{ B }^{(\lambda, \nu)}_{ \boldsymbol{ s } } $ is the corresponding pulled-back kernel on $ \mathbb{ D }^2 $ via the symmetrization map $ \boldsymbol{ s } : \mathbb{ D }^2 \rightarrow \mathbb{ G }_2 $. Therefore, the polarization of
\begin{eqnarray*} 
\det \mathbb B_{ \boldsymbol{ s } }^{ ( \lambda, \nu ) } ( \boldsymbol{ z }, \boldsymbol{z} ) &  = & \det \left[ (B^{(\lambda)} ( \boldsymbol{ s } ( \boldsymbol{ z } ), \boldsymbol{ s } ( \boldsymbol{ z } ) )^{\nu+2} \mathbb B^{( \lambda) } ( \boldsymbol{ s } ( \boldsymbol{ z } ) ) \right] \\
& = & \left( B^{( \lambda ) } ( \boldsymbol{s} ( \boldsymbol{ z } ), \boldsymbol{s} ( \boldsymbol{ z } ) ) \right)^{ 2 (\nu+2) } \det \mathbb B^{ ( \lambda ) } ( \boldsymbol{s} ( \boldsymbol{ z } ) ) \\
& = & \frac{ \lambda^2 \left( B^{ ( \lambda ) } ( \boldsymbol{s} (\boldsymbol{ z } ), s ( \boldsymbol{ z } ) ) \right)^{ 2 \nu+ 1 } B^{ ( \lambda + 2 ) } ( \boldsymbol{s} (\boldsymbol{ z } ), \boldsymbol{s} ( \boldsymbol{ z } ) ) }{ 2|1 - \bar z_1z_2|^{2\lambda + 2}(1 - |z_1|^2)^{\lambda+1}(1 - |z_2|^2)^{\lambda + 1 }} \times H^{ ( \lambda ) } ( \boldsymbol{s}(\boldsymbol{ z }), \boldsymbol{s}(\boldsymbol{z}) ) 
\end{eqnarray*}
is a non-negative definite kernel on $\mathbb D^2$, where 
\begin{equation} \label{H(z)}
H^{(\lambda)} ( \boldsymbol{s}(\boldsymbol{ z }), \boldsymbol{s}(\boldsymbol{z}) ) : = \frac{B^{(\lambda)}(\boldsymbol{s}(\boldsymbol{z}), \boldsymbol{s}(\boldsymbol{z})) \left(|1 - \bar z_1z_2|^{2\lambda + 2} + (1 - |z_1|^2)^{\lambda+1} (1 - |z_2|^2)^{\lambda+1}\right) - \lambda}{|z_1 - z_2|^4}. 
\end{equation}
In a similar way as before, define $H^{(\lambda)}_{ \boldsymbol{ s } }(\boldsymbol{z}, \boldsymbol{z} ) = H^{(\lambda)}( \boldsymbol{ s }(\boldsymbol{z}), \boldsymbol{ s }(\boldsymbol{ z }) )$, for  $\boldsymbol{z} \in \mathbb D^2$. 

\begin{lemma} \label{lemcurv2}
There exists an open subset $U$ of $\mathbb D^2$ containing $\Delta$ such that the function $ H^{(\lambda)}_{ \boldsymbol{ s } } $ defined
in Equation \eqref{H(z)} is of the form 
    \begin{equation} \label{H(z)1}
        H^{(\lambda)}_{  \boldsymbol{ s } } ( \boldsymbol{ z }, \boldsymbol{z} ) = \frac{ \lambda ( \lambda + 1 ) ( 2 \lambda + 1 ) }{ 12 } \frac{ ( 1 - | z_1 |^2 )^{ \lambda - 2 } ( 1 - | z_2 |^2 )^{ \lambda - 2 } }{ | 1 - \bar{ z }_1 z_2 |^{ 2 \lambda } } + \sum_{ p \geq 3 } h_p ( z_1, z_2 ) | z_1 - z_2 |^{ 2p - 4 },
    \end{equation}
where 
\begin{eqnarray*} 
h_p ( z_1, z_2 ) & = & \left[ - \lambda { \binom{\lambda}{p} } + \frac{ \lambda }{ 2 } { \binom{\lambda + 1}{p} } + { \binom{\lambda}{p + 1} } + \frac{1}{2}\sum_{ \substack{ m + n = p \\ m, n \geq 1 }  } { \binom{\lambda}{m + 1} } { \binom{\lambda + 1}{n} } \right] \times\\
& & \phantom{space}\frac{ ( 1 - | z_1 |^2 )^{ \lambda - p } ( 1 - | z_2 |^2 )^{ \lambda - p } }{ | 1 - \bar{ z }_1 z_2 |^{ 2 \lambda } } , 
\end{eqnarray*}
and $ \boldsymbol{ z } = ( z_1, z_2 ) \in U $. In particular, for $ z \in \mathbb{ D } $, $ \lambda > 0 $ and $ \nu \geq 0 $,       
    \begin{equation} \label{resdetcurv}
      \det \mathbb B_s^{ ( \lambda, \nu ) } ( (z, z), (z,z) ) = \left( \frac{ \lambda }{ 2 } \right)^{ 2 (\nu+2) } \times \frac{ ( \lambda + 1 ) ( \lambda + 2 ) ( 2 \lambda + 1 ) }{ 6 } \times \frac{ 1 }{ ( 1 - | z |^2 )^{ 4 [ (\nu+2) ( \lambda + 1 ) + 2 ] } }.
    \end{equation} 
\end{lemma}
\begin{proof}
Substituting the expression of $B^{(\lambda)}_s(\boldsymbol{ z }, \boldsymbol{ z })$ from Lemma \ref{weightedbergkerexpansion} into Equation \eqref{H(z)}, we obtain  
\begin{flalign*}
H^{(\lambda)}_{  \boldsymbol{ s } }(\boldsymbol{z}, \boldsymbol{z}) &= \frac{1}{|z_1 - z_2|^4} \left[\frac{(1 - |z_1|^2)^{\lambda + 1} (1 - |z_2|^2)^{\lambda + 1}}{2 |1 - \bar z_1 z_2|^{2 \lambda}} \left(\sum_{ n \geq 0 } { \lambda \choose n + 1 } \frac{ | z_1 - z_2 |^{ 2n } }{ ( 1 - | z_1 |^2 )^{ n + 1 } ( 1 - | z_2 |^2 )^{ n + 1 } }\right)\right.\\
& \left. \left( \frac{|1 - \bar z_1 z_2|^{2\lambda + 2}}{(1 - |z_1|^2)^{\lambda + 1} (1 - |z_2|^2)^{\lambda + 1}} + 1\right) - \lambda \right]\\
&= \frac{1}{|z_1 - z_2|^4} \left[\frac{(1 - |z_1|^2)^{\lambda + 1} (1 - |z_2|^2)^{\lambda + 1}}{2 |1 - \bar z_1 z_2|^{2 \lambda}} \left(\sum_{ n \geq 0 } { \lambda \choose n + 1 } \frac{ | z_1 - z_2 |^{ 2n } }{ ( 1 - | z_1 |^2 )^{ n + 1 } ( 1 - | z_2 |^2 )^{ n + 1 } }\right)\right.\\
& \left. \left( 2 + \sum_{n \geq 1} { \lambda + 1 \choose n } \frac{ | z_1 - z_2 |^{ 2n } }{ ( 1 - | z_1 |^2 )^n ( 1 - | z_2 |^2 )^n } \right) - \lambda \right]\\
&= \frac{1}{|z_1 - z_2|^4} \left[-\frac{\lambda (1 - |z_1|^2)^\lambda (1 - |z_2|^2)^\lambda}{|1 - \bar z_1 z_2|^{2\lambda}} \left(\frac{|1 - \bar z_1 z_2|^{2\lambda}}{(1 - |z_1|^2)^\lambda (1 - |z_2|^2)^\lambda} - 1\right)  \right.\\
& \left. + \frac{\lambda (1 - |z_1|^2)^\lambda (1 - |z_2|^2)^\lambda}{2|1 - \bar z_1 z_2|^{2\lambda}} \sum_{n \geq 1} { \lambda + 1 \choose n } \frac{ | z_1 - z_2 |^{ 2n } }{ ( 1 - | z_1 |^2 )^n ( 1 - | z_2 |^2 )^n } ~  \right.\\
&\left.+ \frac{(1 - |z_1|^2)^{\lambda+1} (1 - |z_2|^2)^{\lambda+1}}{|1 - \bar z_1 z_2|^{2\lambda}} \sum_{n \geq 2} { \lambda \choose n } \frac{ | z_1 - z_2 |^{ 2n - 2 } }{ ( 1 - | z_1 |^2 )^n ( 1 - | z_2 |^2 )^n } ~  \right.\\
&\left. + \frac{ (1 - |z_1|^2)^{\lambda+1} (1 - |z_2|^2)^{\lambda+1}}{2|1 - \bar z_1 z_2|^{2\lambda}} \left(\sum_{ n \geq 2 } 
{ \lambda \choose n } 
\frac{ | z_1 - z_2 |^{ 2n -2 } }{ ( 1 - | z_1 |^2 )^{ n } ( 1 - | z_2 |^2 )^{ n } }\right)\right.\\
& \times \left.\left(\sum_{ m \geq 1 } { \lambda + 1 \choose m } \frac{ | z_1 - z_2 |^{ 2m } }{ ( 1 - | z_1 |^2 )^{ m } ( 1 - | z_2 |^2 )^{ m } }\right)\right]\\
\end{flalign*}
 \begin{flalign*}
&= \frac{1}{|z_1 - z_2|^4} \left[-\frac{\lambda (1 - |z_1|^2)^\lambda (1 - |z_2|^2)^\lambda}{|1 - \bar z_1 z_2|^{2\lambda}} \sum_{n \geq 1} { \lambda \choose n } \frac{ | z_1 - z_2 |^{ 2n } }{ ( 1 - | z_1 |^2 )^n ( 1 - | z_2 |^2 )^n }  \right.\\
&\left. + \frac{\lambda (1 - |z_1|^2)^\lambda (1 - |z_2|^2)^\lambda}{2|1 - \bar z_1 z_2|^{2\lambda}} \sum_{n \geq 1} { \lambda + 1 \choose n } \frac{ | z_1 - z_2 |^{ 2n } }{ ( 1 - | z_1 |^2 )^n ( 1 - | z_2 |^2 )^n } ~  \right.\\
&\left. + \frac{(1 - |z_1|^2)^{\lambda+1} (1 - |z_2|^2)^{\lambda+1}}{|1 - \bar z_1 z_2|^{2\lambda}} \sum_{n \geq 2} { \lambda \choose n } \frac{ | z_1 - z_2 |^{ 2n - 2 } }{ ( 1 - | z_1 |^2 )^n ( 1 - | z_2 |^2 )^n } ~  \right.\\
&\left. + \frac{ (1 - |z_1|^2)^{\lambda+1} (1 - |z_2|^2)^{\lambda+1}}{2|1 - \bar z_1 z_2|^{2\lambda}} \left(\sum_{ n \geq 2 } 
\binom{ \lambda }{ n } 
\frac{ | z_1 - z_2 |^{ 2n -2 } }{ ( 1 - | z_1 |^2 )^{ n } ( 1 - | z_2 |^2 )^{ n } }\right)\right.\\
& \times \left.\left(\sum_{ m \geq 1 } \binom{ \lambda + 1 }{ m } \frac{ | z_1 - z_2 |^{ 2m } }{ ( 1 - | z_1 |^2 )^{ m } ( 1 - | z_2 |^2 )^{ m } }\right)\right] 
\end{flalign*}
A direct computation verifies that the coefficient of $|z_1 - z_2|^2$ in the above expression is $0$. Now, the expression of $H^{(\lambda)}_{  \boldsymbol{ s } }(\boldsymbol{z}, \boldsymbol{z})$, as given in the statement, is obtained by collecting all the coefficients of $|z_1 - z_2|^{2p - 4}$ for every $p \geq 2$. Equation \eqref{resdetcurv} is obtained by substituting $z_1 = z_2$ in Equation \eqref{H(z)1}.
\end{proof}

We now prove that if  $\lambda_i > 0$ and $\nu_i \in \mathcal{W}_{\mathbb G_2, a}(B^{(\lambda_i)}) \cup \{ 0 \}$, $i = 1, 2$, then the Hilbert modules $\mathbb H^{(\lambda_1, \nu_1)}$ and $\mathbb H^{(\lambda_2, \nu_2)}$ are unitarily inequivalent for distinct pairs $(\lambda_1, \nu_1)$ and $(\lambda_2, \nu_2)$. 

\begin{theorem}\label{inequivofdetker}
 For $\lambda_i > 0$ and $\nu_i \in \mathcal{W}_{\mathbb G_2, a}(B^{(\lambda_i)}) \cup \{ 0 \}$, $i = 1, 2$, the Hilbert modules $\mathbb H^{(\lambda_1, \nu_1)}$ and $\mathbb H^{(\lambda_2, \nu_2)}$ are unitarily equivalent if and only if $\lambda_1 = \lambda_2$ and $\nu_1 = \nu_2$.
\end{theorem}

\begin{proof}
Suppose that $\mathbb H^{(\lambda_1, \nu_1)}$ and $\mathbb H^{(\lambda_2, \nu_2)}$ be unitarily equivalent. Then there exists a non-vanishing holomorphic function $ F : \mathbb{ G }_2 \rightarrow \mathbb{ C } $ such that
    \begin{equation} \label{firstdetkerequiv1}
        \det \mathbb B^{ ( \lambda_1, \nu_1 ) } ( \boldsymbol{s} ( \boldsymbol{ z } ), \boldsymbol{s} ( \boldsymbol{ w } ) ) =  F (  \boldsymbol{s} ( \boldsymbol{ z } ) ) \det \mathbb B^{ ( \lambda_2, \nu_2 ) } ( \boldsymbol{s} ( \boldsymbol{ z } ), \boldsymbol{s} ( \boldsymbol{ w } ) ) \overline{ F ( s ( \boldsymbol{ w } ) ) }, ~~~ \boldsymbol{ z }, \boldsymbol{ w } \in \mathbb{ D }^2 .
    \end{equation}
We now restrict this equation to the diagonal subset $ \Delta $  of $ \mathbb{ D }^2 $ and substitute the values of  $\det \mathbb B^{ ( \lambda_i, \nu_i ) } ( \boldsymbol{s} ( z,z ), \boldsymbol{s} ( z,z ) )$, $i = 1, 2$, from Equation \eqref{resdetcurv} to obtain
\begin{align*} 
       &  \left( \frac{ \lambda_1 }{ 2 } \right)^{ 2 (\nu_1+2) } \times \frac{ ( \lambda_1 + 1 ) ( \lambda_1 + 2 ) ( 2 \lambda_1 + 1 ) }{ 6 } \times \frac{ 1 }{ ( 1 - | z |^2 )^{ 4 [ (\nu_1+2) ( \lambda_1 + 1 ) + 2 ] } }  \\
     \nonumber  &\phantom{gadadh} = | F ( s ( z, z ) ) |^2 \left( \frac{ \lambda_2 }{ 2 } \right)^{ 2 (\nu_2+2) } \times \frac{ ( \lambda_2 + 1 ) ( \lambda_2 + 2 ) ( 2 \lambda_2 + 1 ) }{ 6 } \times \frac{ 1 }{ ( 1 - | z |^2 )^{ 4 [ (\nu_2+2) ( \lambda_2 + 1 ) + 2 ] } } .
\end{align*}
Next, applying $ \partial \overline{ \partial } \log $ to both sides of this resulting equation and evaluating at $ z = 0 $, we have
\begin{equation} \label{detkerequiv1x}
(\nu_1+2)( \lambda_1 + 1 ) = (\nu_2 + 2)( \lambda_2 + 1 ) .
\end{equation}
We let $F_s = F \circ s$ and rewrite Equation \eqref{firstdetkerequiv1} at $(\boldsymbol{z}, \boldsymbol{z}),$ $\boldsymbol{z} \in \mathbb D^2$ as
\begin{equation*}
\det \mathbb B_s^{ ( \lambda_1, \nu_1 ) } (  \boldsymbol{ z } , \boldsymbol{ z }  ) =  F_s ( \boldsymbol{ z } ) \det \mathbb B_s^{ ( \lambda_2, \nu_2 ) } (\boldsymbol{ z }, \boldsymbol{ z } ) \overline{ F_s ( \boldsymbol{ z } ) }, ~~~ \boldsymbol{ z }, \in \mathbb{ D }^2 .
\end{equation*}
After operating $\partial_2 \bar{\partial}_2$ followed by taking the logarithm to both sides of this equation and then evaluating at $(x, 0) \in (0, 1) \times \{ 0 \} $, it yields that 
\begin{align*}
\nonumber &(2\nu_1 + 1) b_{2\bar 2}^{(\lambda_1)} (x, 0) + b_{2 \bar 2}^{(\lambda_1)}(x, 0) + \lambda_1 + 1 + \partial_2\bar \partial_2 \log H^{(\lambda_1)}_{ \boldsymbol{ s } }((x, 0), (x, 0))\\
 & \phantom{gadadhar} = (2\nu_2 + 1) b_{2\bar 2}^{(\lambda_2)} (x, 0) + b_{2 \bar 2}^{(\lambda_2)}(x, 0) + \lambda_2 + 1 + \partial_2\bar \partial_2 \log H^{(\lambda_2)}_{ \boldsymbol{ s } }((x, 0), (x, 0)).
\end{align*}
Here, $H^{(\lambda_i)}_{ \boldsymbol{ s }}( \boldsymbol{ z }, \boldsymbol{ z })$, $i = 1, 2$, is as defined in Equation \eqref{H(z)}. For $i = 1, 2$, consider the function
    \begin{equation*} 
        L_i(\boldsymbol{z}) = B_{ \boldsymbol{ s } }^{(\lambda_i)}(\boldsymbol{z}, \boldsymbol{z})\left(|1 - \bar z_1 z_2|^{2\lambda_i + 2} + (1 - |z_1|)^{\lambda_i+1}(1 - |z_2|)^{\lambda_i+1} \right) - \lambda_i,\,\,\boldsymbol{z} = (z_1, z_2) \in \mathbb D^2.
    \end{equation*}
Note that $L_i(\boldsymbol{z}) = |z_1 - z_2|^4 H^{(\lambda_i)}_{ \boldsymbol{ s } }(\boldsymbol{z}, \boldsymbol{z})$ and therefore, for any $x \in (0, 1)$, we have
$$\partial_2 \bar \partial_2 \log H^{(\lambda_i)}_s((x, 0), (x, 0)) = \partial_2 \bar \partial_2 \log L_i(x, 0) = \frac{L_i(x, 0) \partial_2 \bar \partial_2 L_i(x, 0) - \partial_2 L_i(x, 0)\bar \partial_2 L_i(x, 0)}{L_i(x, 0)^2}.$$
A direct computation shows that 
    \begin{equation*} 
      \partial_2 \bar \partial_2  L_i(x, 0) = \frac{G^{(i)}_1(x)}{2x^4(1 - x^2)^{\lambda_i}} - \frac{G^{(i)}_2(x)}{x^2(1 - x^2)^{\lambda_i}} + \frac{(\lambda_i + 1)^2\left(1 - (1 - x^2)^{\lambda_i}\right)}{2(1 - x^2)^{\lambda_i}}
    \end{equation*}
and
    \begin{equation*} 
      \partial_2 L_i(x, 0) \bar \partial_2  L_i(x, 0) = \frac{G^{(i)}_3(x)}{4x^6(1 - x^2)^{2\lambda_i}} - \frac{G^{(i)}_4(x)}{2x^4(1 - x^2)^{2\lambda_i}} + \frac{(\lambda_i + 1)^2\left(1 - (1 - x^2)^{\lambda_i}\right)^2}{4x^2(1 - x^2)^{2\lambda_i}}
    \end{equation*}
where $$G^{(i)}_1(x) = \left[1 + \lambda_i x^2 - (1 + \lambda_i x^2)^2(1 - x^2)^{\lambda_i}\right] \left(1 + (1 - x^2)^{\lambda_i + 1}\right),$$ $$G^{(i)}_2(x) = (\lambda_i + 1)\left[1 - (1 - x^2)^{\lambda_i} -\lambda_i x^2(1 - x^2)^{\lambda_i}\right],$$ 
$$G^{(i)}_3(x) = \left[1 - (1 - x^2)^{\lambda_i} - \lambda_i x^2(1 - x^2)^{\lambda_i} \right]^2 \left(1 + (1 - x^2)^{\lambda_i + 1}\right)^2\,\,\mbox{and}$$ 
$$G^{(i)}_4(x) = (\lambda_i + 1)\left(1 + (1 - x^2)^{\lambda_i + 1}\right)\left(1 - (1 - x^2)^{\lambda_i}\right) \left[1 - (1 - x^2)^{\lambda_i} - \lambda_i x^2 (1 - x^2)^{\lambda_i}\right].$$
Also, note that 
$$L_i(x, 0) = \frac{\tilde L_i(x, 0)}{2x^2(1 - x^2)^{\lambda_i}},$$
where $\tilde L_i(x, 0) = \left(1 - (1 - x^2)^{\lambda_i}\right)\left(1 + (1 - x^2)^{\lambda_i+1}\right) - 2\lambda_i x^2(1 - x^2)^{\lambda_i}$ for $x \in (0, 1)$. Now, substituting the values of $L_i(x, 0)$, $\partial_2 \bar \partial_2  L_i(x, 0)$, $\partial_2 L_i(x, 0) \bar \partial_2  L_i(x, 0)$ in Equation \eqref{eqn:4.22} and then multiplying both sides of the resulting equation by $16L_1(x, 0)^2 L_2(x, 0)^2 x^{10} (1 - x^2)^{2(\lambda_1 + \lambda_2)}$, we obtain
\begin{flalign*}
&\tilde L_1(x, 0)^2 \tilde L_2(x, 0)^2 \left((2\nu_1 + 1)x^2 b_{2\bar 2}^{(\lambda_1)}(x, 0) + x^2 b_{2 \bar 2}^{(\lambda_1 + 2)}(x, 0) + x^2 (\lambda_1 + 1) \right) + \tilde L_1(x, 0) \tilde L_2(x, 0)^2\\
& \left( G^{(1)}_1(x)- 2x^2 G^{(1)}_2(x) +
(\lambda_1 + 1)^2x^4 \left(1 - (1 - x^2)^{\lambda_1}\right) \right) - \tilde L_2(x, 0)^2 \left( G^{(1)}_3(x) - 2x^2G^{(1)}_4(x)
+ \right.\\
&\left. (\lambda_1 + 1)^2x^4 \left(1 - (1 - x^2)^{\lambda_1}\right)^2 \right) = \tilde L_1(x, 0)^2 \tilde L_2(x, 0)^2 \left((2\nu_2 + 1)x^2 b_{2\bar 2}^{(\lambda_2)}(x, 0) + x^2 b_{2 \bar 2}^{(\lambda_2 + 2)}(x, 0)  \right.\\ 
&\left. + x^2 (\lambda_2 + 1) \right) +
\tilde L_1(x, 0)^2 \tilde L_2(x, 0) \left( G^{(2)}_1(x)- 2x^2 G^{(2)}_2(x) +
(\lambda_2 + 1)^2x^4 \left(1 - (1 - x^2)^{\lambda_2}\right) \right)\\
&- \tilde L_1(x, 0)^2 \left( G^{(2)}_3(x) - 2x^2G^{(2)}_4(x)
+ (\lambda_2 + 1)^2x^4 \left(1 - (1 - x^2)^{\lambda_2}\right)^2 \right).   
\end{flalign*}
Finally, taking the limit $x \to 1$ to both sides of the equation above and observing that $b_{2 \bar 2}^{(\lambda_i)}(x, 0) \to \lambda_i$, $\tilde L_i(x, 0) \to 1$, $G^{(i)}_1(x) \to \lambda_i + 1$, $G^{(i)}_2(x) \to \lambda_i + 1$, $G^{(i)}_3(x) \to 1$, $G^{(i)}_4(x) \to \lambda_i + 1$, we obtain
\begin{equation}\label{detker5x}
 (2\nu_1 + 1)\lambda_1 + 3\lambda_1 = (2\nu_2 + 1)\lambda_2 + 3\lambda_2.   
\end{equation}
Then Equations \eqref{detkerequiv1x} and \eqref{detker5x} together imply that $\lambda_1 = \lambda_2$ and $\nu_1 = \nu_2$.
\end{proof}

As promised in the introduction, we now prove that no two of the Hilbert modules $\mathbb H^{(\lambda, \nu)}$ and $\mathbb A^{(\mu, \eta)}$ are unitarily equivalent.

\begin{theorem} \label{thminequiv}
For any pair $\lambda, \mu > 0$, pick $\nu \in \mathcal{W}_{\mathbb G_2, a}(B^{(\lambda)}) \cup \{ 0 \} $ and $\eta \in \mathcal{W}_{\mathbb G_2, a}(B^{(\mu)})$. Then \rm{Aut}$(\mathbb G_2)$-homogeneous analytic Hilbert modules $\mathbb H^{(\lambda, \nu)}$ and $\mathbb A^{(\mu, \eta)}$ are unitarily inequivalent. 
\end{theorem}

\begin{proof}
Assume that the Hilbert modules $\mathbb H^{(\lambda, \nu)}$ and $\mathbb A^{(\mu, \eta)}$ are unitarily equivalent. Then there exists a non-vanishing holomorphic function $ F : \mathbb{ G }_2 \rightarrow \mathbb{ C } $ such that
    \begin{equation} \label{kerequiv1}
        \det \mathbb B^{ ( \lambda, \nu ) } ( \boldsymbol{s} ( \boldsymbol{ z } ), \boldsymbol{s} ( \boldsymbol{ w } ) ) =  F (  \boldsymbol{s} ( \boldsymbol{ z } ) ) \left( B^{ ( \mu ) } ( \boldsymbol{s} ( \boldsymbol{ z } ), \boldsymbol{s} ( \boldsymbol{ w } ) ) \right)^{ \eta } \overline{ F ( \boldsymbol{s} ( \boldsymbol{ w } ) ) }, ~~~ \boldsymbol{ z }, \boldsymbol{ w } \in \mathbb{ D }^2 .
    \end{equation}
Substituting $ \boldsymbol{ z } $ and $ \boldsymbol{ w } $ by $ ( z, z ) \in \mathbb{ D }^2 $ in Equations \eqref{resBerg}, \eqref{resdetcurv} and \eqref{kerequiv1},  we obtain  
    \begin{align} \label{kerequiv2}
       & | F ( \boldsymbol{s} ( z, z ) ) |^2 \left( \frac{ \mu }{ 2 } \times \frac{ 1 }{ ( 1 - | z |^2 )^{ 2 \mu + 2 } } \right)^{ \eta } \\
     \nonumber  &\phantom{gadadharmis} = \left( \frac{ \lambda }{ 2 } \right)^{ 2 (\nu+2) } \times \frac{ ( \lambda + 1 ) ( \lambda + 2 ) ( 2 \lambda + 1 ) }{ 6 } \times \frac{ 1 }{ ( 1 - | z |^2 )^{ 4 [ (\nu+2) ( \lambda + 1 ) + 2 ] } } .
    \end{align}
    Applying the operator $ \partial \overline{ \partial } \log $ to both sides of Equation \eqref{kerequiv2}, we get
    \begin{equation} \label{eqn:4.21}
      \eta ( \mu + 1 ) = 2[ (\nu+2)
      ( \lambda + 1 ) + 2 ] .
    \end{equation}
We restrict Equation \eqref{kerequiv1} to $(\boldsymbol{z}, \boldsymbol{z}),$ $\boldsymbol{z} \in \mathbb D^2$, and rewrite it as
    \begin{equation} \label{equn:4.23} 
       \det \mathbb B_{ \boldsymbol{ s } }^{ ( \lambda, \nu ) } ( \boldsymbol{ z } ,  \boldsymbol{ z }  ) =  F_s (  \boldsymbol{ z }  ) \left( B^{ ( \mu ) } ( \boldsymbol{ s } ( \boldsymbol{ z } ) , \boldsymbol{ s } ( \boldsymbol{ z } ) ) \right)^{ \eta } \overline{ F_{\boldsymbol{s}} (  \boldsymbol{ z }  )} , ~~~ \boldsymbol{ z } \in \mathbb{ D }^2 ,
    \end{equation} 
 where $ F_{\boldsymbol{s}} = F \circ \boldsymbol{s}$. Applying $\partial_2\bar \partial_2 \log $ to the both sides of this equation and evaluating at $(x, 0) \in (0, 1) \times \{ 0 \} $, we obtain
    \begin{equation} \label{eqn:4.22}
        \eta b_{2\bar 2}^{(\mu)}(x,0) = (2\nu + 1) b_{2\bar 2}^{(\lambda)} (x, 0) + b_{2 \bar 2}^{(\lambda+2)}(x, 0) + \lambda + 1 + \partial_2\bar \partial_2 \log H^{(\lambda)}_{ \boldsymbol{ s } }((x, 0), (x, 0)).
    \end{equation}
Here, $H^{(\lambda)}_{ \boldsymbol{ s }}( \boldsymbol{ z }, \boldsymbol{ z })$ is as defined in Equation \eqref{H(z)}. Now consider the function as in the proof of Theorem \ref{inequivofdetker}
    \begin{equation*} \label{kerequiv4}
        L(\boldsymbol{z}) = B_{ \boldsymbol{ s } }^{(\lambda)}(\boldsymbol{z}, \boldsymbol{z})\left(|1 - \bar z_1 z_2|^{2\lambda + 2} + (1 - |z_1|)^{\lambda+1}(1 - |z_2|)^{\lambda+1} \right) - \lambda,\,\,\boldsymbol{z} = (z_1, z_2) \in \mathbb D^2,
    \end{equation*}
and observe that $L(\boldsymbol{z}) = |z_1 - z_2|^4 H^{(\lambda)}_{ \boldsymbol{ s } }(\boldsymbol{z}, \boldsymbol{z})$. Therefore, for any $x \in (0, 1)$, we have
$$\partial_2 \bar \partial_2 \log H^{(\lambda)}_{\boldsymbol{s}}((x, 0), (x, 0)) = \partial_2 \bar \partial_2 \log L(x, 0) = \frac{L(x, 0) \partial_2 \bar \partial_2 L(x, 0) - \partial_2 L(x, 0)\bar \partial_2 L(x, 0)}{L(x, 0)^2}.$$
Multiplying both sides of Equation \eqref{eqn:4.22} by $4L(x, 0)^2 x^6 (1 - x^2)^{2\lambda}$ and then, following a similar computation as given in the proof of Theorem \ref{inequivofdetker}, we obtain
    \begin{equation} \label{kerequiv7}
      \eta \mu = 2(\nu + 2) \lambda + 3.
    \end{equation}
Equations \eqref{eqn:4.21} and \eqref{kerequiv7} together imply that 
$$\mu = \frac{2(\nu+2) \lambda + 3 }{2\nu + 5}.$$

For $z \in \mathbb D$, substituting $z_1 = z$ and $z_2 = -z$ in Equations \eqref{defofwbergker} and \eqref{H(z)}, we have
$$B_{ \boldsymbol{ s } }^{(\lambda)}((z, -z), (z, -z)) = \frac{1}{4} \displaystyle \sum_{n=0}^\infty \frac{(2\lambda)_{2n+1}}{(2n+1)!} |z|^{4n},$$
$$H^{(\lambda)}_{ \boldsymbol{ s }}((z, -z), (z, -z)) = \frac{1}{32} \displaystyle \sum_{p=1}^\infty \left(\sum_{n+m = p} \frac{(2\lambda)_{2n+1}}{(2n + 1)!} \binom{2\lambda + 2}{2m}\right) |z|^{4(p-1)}.$$
The Equation \eqref{equn:4.23} implies that
    \begin{equation} \label{kerequiv8}
        \det \mathbb B_{ \boldsymbol{ s } }^{ ( \lambda, \nu ) } ((z, -z), (z, -z)) =  F_{\boldsymbol{s}} (z,-z) \left( B_{ \boldsymbol{ s }}^{ ( \mu ) } ((z,-z), (z, -z)) \right)^{ \eta } \overline{ F_{\boldsymbol{s}} (z,-z)}, ~~~ z \in \mathbb{ D }.
    \end{equation}
Since both $ \det \mathbb B_{ \boldsymbol{ s }}^{ ( \lambda, \nu ) } ((z, -z), (z, -z))$ and $\left( B_{ \boldsymbol{ s }}^{ ( \mu ) } ((z,-z), (z, -z)) \right)^{ \eta }$ are functions of $|z|^2$, it follows that $ F_{\boldsymbol{s}}(z, -z)$ is a constant function and 
\begin{equation} \label{eqn:4.27a}
|F_{\boldsymbol{s}}(0,0)|^2 = \frac{ \det \mathbb B_{ \boldsymbol{ s } }^{ ( \lambda, \nu ) }(0,0)}{\left( B_{ \boldsymbol{ s } }^{ ( \mu ) } (0,0) \right)^{ \eta }}.\end{equation}
Substituting the power series representations of $B_{ \boldsymbol{ s } }^{(\lambda)}((z, -z), (z, -z))$, $B_{ \boldsymbol{ s } }^{(\lambda+2)}((z,-z), (z, -z))$, $B_{ \boldsymbol{ s } }^{(\mu)}((z,-z), (z, -z))$ and $H^{(\lambda)}_{ \boldsymbol{ s } }((z, -z), (z, -z))$ in Equation \eqref{kerequiv8}, and using Equation \eqref{eqn:4.27a}, we obtain
\begin{flalign*}
 \nonumber &\left(\frac{\mu}{2}\right)^\eta \frac{\lambda^2}{2^{4\nu + 10}}  \left( \displaystyle \sum_{n=0}^\infty \frac{(2\lambda)_{2n+1}}{(2n+1)!} |z|^{4n} \right)^{2\nu + 1} \left( \displaystyle \sum_{n=0}^\infty \frac{(2\lambda + 4)_{2n+1}}{(2n+1)!} |z|^{4n} \right)\\
 \nonumber &\phantom{SomnathSom}\left[ \displaystyle \sum_{p=1}^\infty \left( \displaystyle \sum_{n+m=p} \frac{(2\lambda)_{2n+1}}{(2n+1)!} \binom{2\lambda + 2}{2m} \right) |z|^{4(p-1)} \right]\\
 & = \frac{ \det \mathbb B_{ \boldsymbol{ s } }^{ ( \lambda, \nu ) } (0,0)}{2^{2\eta}} \left(1 - |z|^4\right)^{2\lambda + 2} \left(\displaystyle \sum_{n=0}^{\infty} \frac{(2\mu)_{2n+1}}{(2n+1)!} |z|^{4n}\right)^\eta.
\end{flalign*}
Equating the coefficient of $|z|^4$ from both sides of the equation above, we obtain
    \begin{equation*}
        ((\nu+2)(\lambda + 1) + 2)\mu = (\nu+2) \lambda (\lambda + 1) + \frac{-2\lambda^2 + 13\lambda + 19}{5}
    \end{equation*}
and after substituting $\mu = \frac{2(\nu+2) \lambda + 3 }{2\nu + 5}$, we get
    \begin{equation}\label{kerequiv9}
    \nu\lambda^2 + (5 - 4\nu)\lambda + (23\nu + 35) = 0.
    \end{equation}
Equation \eqref{kerequiv9} is  quadratic in the variable $\lambda$. The discriminant is
$$-76\nu^2 - 180\nu + 25,$$
which is negative if $\nu \geq 1$. Thus, there does not exist any $\lambda > 0$ satisfying Equation \eqref{kerequiv9} if $\nu \geq 1$. If $0 \leq \nu < 1$, the coefficients are non-negative and at least one of them is positive. This implies that there does not exist $\lambda > 0$ satisfying Equation \eqref{kerequiv9}.  
\end{proof}

\begin{remark} \label{vb} Let $E^{(\lambda, \nu)}$ be the  trivial holomorphic line bundle over $\mathbb{G}_2$ equipped with the hermitian metric $\det \mathbb B^{(\lambda, \nu)} (\boldsymbol{u}, \boldsymbol{u})$, $\boldsymbol{u} \in \mathbb G_2$  and $F^{(\mu,\eta)}$ be the trivial holomorphic line bundle over $\mathbb{G}_2$ equipped with the hermitian metric  $\left(B^{(\mu)}(\boldsymbol{u}, \boldsymbol{u})\right)^\eta$, $\boldsymbol{u} \in \mathbb G_2$. It follows from the proof of Theorems \ref{powerwtdberg} and \ref{detkernelofberg} that $E^{(\lambda, \nu)}$ and $F^{(\mu, \eta)}$ are homogeneous.  Similarly, from Theorems \ref{inequivofweightedbergker} and  \ref{inequivofdetker}, it follows that $E^{(\lambda, \nu)}$ and $F^{(\mu, \eta)}$ are inequivalent as holomorphic hermitian line bundles among themselves. Finally, from Theorem \ref{thminequiv}, it follows that $E^{(\lambda, \nu)}$ is not equivalent to $F^{(\mu, \eta)}$ for any pair $(\lambda,\nu)$ and $(\mu, \eta)$ with $ \lambda, \mu, \eta > 0$ and $\nu \geq 0$.    
\end{remark}

We recall that the Ricci form $ \mathrm{Ric} $ of an $ n $ dimensional K\"ahler manifold $ ( M, \omega ) $ with the K\"ahler form $ \omega $ given by $ \omega = \sum_{ i, j = 1 }^n g_{ i \bar{ j } } dz_i \wedge d\bar{z_j} $ is defined as 
    $$ \mathrm{Ric} : = - \sqrt{ - 1 } \sum_{ i, j = 1 }^n \partial_i \bar{ \partial_j } \log \det \big( \! \!\big( g_{ i \bar{ j } } \big) \!\! \big)_{ i, j = 1 }^n dz_i \wedge d\bar{z_j}  $$
(cf. \cite{Mok}). A K\"ahler manifold is said to be K\"ahler-Einstein if the Ricci form $ \mathrm{Ric} $ is a scalar multiple of the  K\"ahler form $ \omega $ (cf. \cite{Mok}).  It is recently shown in \cite{CY} that the Bergman metric $ \sum_{ i, j = 1 }^2 \mathbb{ B }_{ i \bar{ j } }^{ ( 2 ) } dz_i \wedge d \bar{ z_j } $ is not a K\"ahler-Einstein metric on the symmetrized bidisc $ \mathbb{ G }_2 $. Here, as a direct consequence of Theorem \ref{thminequiv}, we show that none of the weighted Bergman metrics $B^{(\lambda)}$, $\lambda >0$, on $\mathbb{G}_2$ are K\"ahler-Einstein. 

\begin{cor} \label{ke}
The weighted Bergman metric $ \sum_{ i, j = 1 }^2 \mathbb{ B }_{ i \bar{ j } }^{ ( \lambda ) } dz_i \wedge d \bar{ z_j } $ on $ \mathbb{ G }_2 $ is not K\"ahler-Einstein for any $ \lambda > 0 $.
\end{cor}

\begin{proof}
Suppose that there exists $ c \in \mathbb{ R } $ such that the Ricci form $ \mathrm{Ric}^{ ( \lambda ) } $ on $ \mathbb{ G }_2 $ of the weighted Berman metric satisfies
$$ \mathrm{Ric}^{ ( \lambda ) } = c \sum_{ i, j = 1 }^2 \mathbb{ B }_{ i \bar{ j } }^{ ( \lambda ) } dz_i \wedge d \bar{ z_j } . $$
It follows from the definition of $ \mathrm{ Ric } $ that
$$ \sum_{ i, j = 1 }^n \partial_i \bar{ \partial_j } \log \det \big( \! \!\big( \mathbb{ B }_{ i \bar{ j } }^{ ( \lambda ) } \big) \!\! \big)_{ i, j = 1 }^n dz_i \wedge d\bar{z_j}  = c \sum_{ i, j = 1 }^2 \mathbb{ B }_{ i \bar{ j } }^{ ( \lambda ) } dz_i \wedge d \bar{ z_j } . $$
Since $ \mathbb{ B }_{ i \bar{ j } }^{ ( \lambda ) }  = \partial_i \bar{ \partial_j } \log B^{ ( \lambda ) } $, $ 1 \leq i, j \leq 2 $, it follows that 
\begin{equation}\label{kheqn}
\partial_i \bar{ \partial_j } \log \det \big[ \big( B^{ ( \lambda ) } \big)^{ \nu } \big( \! \!\big( \mathbb{ B }_{ i \bar{ j } }^{ ( \lambda ) } \big) \!\! \big)_{ i, j = 1 }^n \big] =   \partial_i \bar{ \partial_j } \log \big( B^{ ( \lambda ) } \big)^{ 2 \nu + c } , ~ 1 \leq i, j \leq 2 .    
\end{equation} 
From Remark \ref{vb}, we know that the trivial hermitian holomorphic line bundles $E^{(\lambda, \nu)}$ and $F^{(\mu, \eta)}$ over $\mathbb G_2$ are inequivalent for all choices of $ \lambda, \mu, \eta > 0$ and $\nu \geq 0$.  Thus, if we choose $\nu > 0$ such that $2\nu + c > 0$, then Equation \eqref{kheqn} leads to a contradiction.
\end{proof}
\subsubsection*{Acknowledgement} Priority for Equation \eqref{factorquasiinvbergker} and Corollary \ref{cor:4.5} is due to Subrata Shyam Roy, who proved them in a 2012 communication to G. Misra. We thank him warmly. The proofs given in Section 3 were written afresh.
\bibliographystyle{amsplain}

\end{document}